\documentclass[11pt]{amsart}
\usepackage{amsmath,amsthm,amsfonts,amssymb,times,epsf}
\usepackage{epsfig, afterpage,titlesec}

\titleformat{\section}
{\normalfont\Large\bfseries}
{\thesection}{1em}{}

\titleformat{\subsection}
{\normalfont\large\bfseries}
{\thesubsection}{1em}{}

\DeclareMathAlphabet{\curly}{U}{rsfs}{m}{n}

\theoremstyle{remark}
\newtheorem{remark}{Remark}
\theoremstyle{plain}

\newtheorem{lem}{Lemma}[section]
\newtheorem{thm}{Theorem}
\newtheorem{cor}[thm]{Corollary}
\newtheorem{lemcor}[lem]{Corollary}
\newtheorem{conj}{Conjecture}
\newtheorem{defn}{Definition}
\numberwithin{equation}{section}

\newcommand{\RR}{{\mathbb R}}

\newcommand{\NN}{{\mathbb N}}

\newcommand{\be}{\begin{equation}}
\newcommand{\ee}{\end{equation}}
\newcommand{\benn}{\begin{equation*}}
\newcommand{\eenn}{\end{equation*}}
\newcommand{\bal}{\begin{align*}}
\newcommand{\ea}{\end{align*}}
\newcommand{\eal}{\ensuremath{\end{align*}}}
\newcommand{\bea}{\begin{eqnarray}}
\newcommand{\eea}{\end{eqnarray}}

\newcommand{\lam}{\ensuremath{\lambda}}

\renewcommand{\a}{\ensuremath{\alpha}}
\renewcommand{\b}{\ensuremath{\beta}}
\newcommand{\del}{\ensuremath{\delta}}
\newcommand{\eps}{\ensuremath{\varepsilon}}
\newcommand{\e}{\ensuremath{\varepsilon}}
\renewcommand{\(}{\left(}
\renewcommand{\)}{\right)}
\newcommand{\ds}{\displaystyle}
\newcommand{\pfrac}[2]{\left(\frac{#1}{#2}\right)}

\newcommand{\fl}[1]{{\ensuremath{\left\lfloor {#1} \right\rfloor}}}
\newcommand{\cl}[1]{{\ensuremath{\left\lceil #1 \right\rceil}}}

\newcommand{\fancyA}{{\curly{A}}}

\newcommand{\fancyC}{{\curly C}}
\newcommand{\fancyS}{{\curly S}}
\newcommand{\fancyT}{{\curly T}}
\newcommand{\fancyV}{\curly V}
\newcommand{\fancyU}{\curly U}
\newcommand{\fancyB}{{\curly{B}}}

\renewcommand{\le}{\leqslant}
\renewcommand{\leq}{\leqslant}
\renewcommand{\ge}{\geqslant}
\renewcommand{\geq}{\geqslant}

\newcommand{\vone}{{\mathbf{1}}}
\newcommand{\vzz}{\mathbf{z}}
\renewcommand{\d}{\delta}
\renewcommand{\rho}{\varrho}
\newcommand{\om}{\Omega}
\newcommand{\bz}{\boldsymbol\zeta}
\newcommand{\vta}{\boldsymbol\theta}

\newcommand{\bxi}{\boldsymbol\xi}
\newcommand{\vxi}{\bxi}
\newcommand{\order}{\asymp}
\newcommand{\vx}{\mathbf{x}}
\newcommand{\bx}{\mathbf{x}}
\newcommand{\vy}{\mathbf{y}}
\newcommand{\vm}{\mathbf{m}}

\DeclareMathOperator{\vvol}{Vol}
\newcommand{\vv}{\mathbf{v}}
\newcommand{\vu}{\mathbf{u}}
\newcommand{\vz}{\mathbf{0}}
\newcommand{\vp}{\mathbf{p}}
\newcommand{\ve}{\mathbf{e}}
\renewcommand{\th}{\ensuremath{\theta}}

\begin{document}

\title{The distribution of totients}
\author{Kevin Ford}
\begin{abstract}
This paper is a comprehensive study of the set of totients, i.e.
the set of values taken by Euler's $\phi$-function. 
We fist determine the true order of magnitude of $V(x)$, the
number of totients $\le x$.  We also show that if there is a totient
with exactly $k$ preimages under $\phi$ (a totient with ``multiplicity`` $k$), 
then the counting function for such totients,
$V_k(x)$, satisfies $V_k(x)\gg_k V(x)$.
Sierpi\'nski conjectured that every multiplicity $k\ge 2$
is possible, and we deduce this from the Prime $k$-tuples Conjecture.
We also make some progress toward an
older conjecture of Carmichael, which states that no totient has
multiplicity 1. The
lower bound for a possible counterexample is extended to $10^{10^{10}}$ and the
bound $\liminf_{x\to \infty} V_1(x)/V(x) \le 10^{-5,000,000,000}$ is shown.
Determining the order of $V(x)$ and $V_k(x)$ also provides a description
of the ``normal'' multiplicative structure of totients.  This takes 
the form of bounds on the sizes of the prime factors
of a pre-image of a typical totient.  One corollary is that the
normal number of prime factors of a totient $\le x$ is $c\log\log x$,
where $c\approx 2.186$.  Similar results are proved for 
the set of values taken by a general multiplicative arithmetic function,
such as the sum of divisors function, whose behavior is similar to that of
Euler's function.
\end{abstract}



\thanks{
Much of the early
work for this paper was completed while the author was enjoying the
hospitality of the Institute for Advanced Study, supported by National
Science Foundation grant DMS 9304580.}

\dedicatory{Dedicated to the memory of Paul Erd\H os (1913--1996)}

\maketitle


\section{Introduction}


Let $\fancyV$ denote the set of values taken by Euler's
$\phi$-function (totients), i.e.
$$
\fancyV = \{ 1,2,4,6,8,10,12,16,18,20,22,24,28,30, \cdots \}.
$$
Let
\be\label{basic def}
\begin{split}
\fancyV(x) &= \fancyV \cap [1,x], \\
V(x) &= |\fancyV(x)|, \\
\phi^{-1}(m) &= \{n:\phi(n)=m\}, \\
A(m) &= |\phi^{-1}(m)|, \\
V_k(x) &= |\{m \le x : A(m)=k \}|.
\end{split}
\ee

We will refer to $A(m)$ as the multiplicity of $m$.  This paper
is concerned with the following problems.
\bigskip

{
\parindent .5in
1.  What is the order of $V(x)$?

2.  What is the order of $V_k(x)$ when the multiplicity $k$ is possible?

3.  What multiplicities are possible?

4.  What is the normal multiplicative structure of totients?
}
\bigskip

\subsection{The order of $V(x)$} 
 The fact that $\phi(p)=p-1$ for primes $p$
implies $V(x) \gg x/\log x$ by the Prime Number Theorem.
Pillai \cite{Pi} gave the first non-trivial
upper bound on $V(x)$, namely
$$
V(x) \ll \frac{x}{(\log x)^{(\log 2)/e}}.
$$
Using sieve methods, Erd\H os \cite{E1} improved this to
$$
V(x) \ll_\e \frac{x}{(\log x)^{1-\e}}
$$
for every $\e>0$.  Upper and lower bounds for $V(x)$ were sharpened
in a series of papers by  Erd\H os \cite{E2}, 
Erd\H os and Hall \cite{EH1, EH2}, Pomerance \cite{P1}, and finally by
Maier and Pomerance \cite{MP},  who showed that
\be\label{MP}
V(x) = \frac{x}{\log x} \exp\{ (C+o(1))(\log_3 x)^2 \}
\ee
for a constant $C$ defined below.
Here $\log_k x$ denotes the $k$th iterate of the logarithm.  Let
\be\label{F def}
F(x) = \sum_{n=1}^\infty a_nx^n, \qquad a_n = (n+1)\log(n+1)-n\log n-1.
\ee
Since $a_n \sim \log n$ and $a_n>0$, it follows that $F(x)$ is defined and
strictly increasing on $[0,1)$, $F(0)=0$ and $F(x) \to\infty$ as $x \to 1^-$.
Thus, there is a unique number $\rho$ such that
\be\label{rho def}
F(\rho) = 1 \qquad (\rho = 0.542598586098471021959\ldots).
\ee
In addition, $F'(x)$ is strictly increasing, and
$$
F'(\rho) = 5.69775893423019267575 \ldots
$$
Let
\be\label{C def}
C = \frac{1}{2|\log \rho|} = 0.81781464640083632231\ldots 
\ee
and
\be\label{D def}
\begin{split}
D &= 2C(1+\log F'(\rho) - \log(2C)) - 3/2 \\
&= 2.17696874355941032173 \ldots
\end{split}
\ee

Our main result is a determination of the true order of $V(x)$.
\begin{thm}\label{V(x)} We have
$$
V(x) = \frac{x}{\log x} \exp\{ C(\log_3 x-\log_4 x)^2 +D \log_3 x
- (D+1/2-2C)\log_4 x + O(1) \}.
$$
\end{thm}

\subsection{The order of $V_k(x)$}

Erd\H os \cite{E3} showed by sieve methods that
if $A(m)=k$, then for most primes $p$, $A(m(p-1))=k$.
If the multiplicity $k$ is possible, then
$V_k(x) \gg x/\log x$.  Applying the machinery used to prove
Theorem \ref{V(x)}, we show that if there exists $m$ with
$A(m)=k$, then a positive proportion of totients have multiplicity $k$.

\begin{thm}\label{V_k(x)} 
For every $\eps>0$, if $A(d)=k$, then 
$$
V_k(x) \gg_\e d^{-1-\e} V(x) \qquad (x\ge x_0(d)).
$$
\end{thm}

\begin{conj}
For $k \ge 2$,
$$
\lim_{x\to \infty} \frac{V_k(x)}{V(x)} = C_k.
$$
\end{conj}

\begin{table}
\begin{tabular}{|l|l|l|l|l|l|l|l|}
\hline
$x$ &  $V(x)$  &  $V_2/V$ & $V_3/V$ & $V_4/V$   &
   $V_5/V$   & $V_6/V$    & $V_7/V$  \\
\hline
 1M &      180,184 & 0.380727 & 0.140673 & 0.098988 &
                              0.042545 & 0.062730 & 0.020790 \\
  5M &     840,178 & 0.379462 & 0.140350 & 0.102487 &
                              0.042687 & 0.063193 & 0.020373 \\
 10M &   1,634,372 & 0.378719 & 0.140399 & 0.103927 &
                              0.042703 & 0.063216 & 0.020061 \\
 25M &   3,946,809 & 0.378198 & 0.140233 & 0.105466 &
                              0.042602 & 0.063414 & 0.019819 \\
 125M & 18,657,531 & 0.377218 & 0.140176 & 0.107873 &
                              0.042560 & 0.063742 & 0.019454 \\
 300M & 43,525,579 & 0.376828 & 0.140170 & 0.108933 &
                              0.042517 & 0.063818 & 0.019284 \\
 500M & 71,399,658 & 0.376690 & 0.140125 & 0.109509 &
                             0.042493 & 0.063851 & 0.019194 \\
\hline
\end{tabular}
\smallskip
\caption{$V_k(x)/V(x)$ for $2\le k\le 7$}
\end{table}

Table 1 lists values of $V(x)$ and the ratios $V_k(x)/V(x)$ for
$2\le k\le 7$.
Numerical investigations seem to
indicate that $C_k \order 1/k^2$.   In fact,
 at $x=500,000,000$ we have
$1.75 \le V_k(x)/V(x)  \le 2.05$ for $20\le k\le 200$.
This data is very misleading, however. 
Erd\H os \cite{E1} showed that there are infinitely many totients for which
$A(m) \ge m^{c_4}$
for some positive constant $c_4$.  The current record is
$c_4=0.7039$ \cite{BH}.  Consequently, by 
Theorem \ref{V_k(x)},  for infinitely many $k$ we have
$$
\frac{V_k(x)}{V(x)} \gg k^{-1/c_4+\e} \gg k^{-1.42} \qquad (x>x_0(k)).
$$
Erd\H os
has conjectured that every $c_4<1$ is admissible.

We also show that most totients have ``essentially bounded'' multiplicity.

\begin{thm}\label{A(m) bounded}
Uniformly for $x\ge 2$ and $N\ge 2$, we have
$$
\frac{| \{ m\in \fancyV(x) : A(m)\ge N \}| }{V(x)} =
\sum_{k\ge N} \frac{V_k(x)}{V(x)} \ll \exp\{ - \tfrac14 (\log_2 N)^2 \}.
$$
\end{thm}

{\bf Remark.}  The proof of \cite[Theorem 3]{F98} contains an error, and the corrected
proof (in Sec. 7.1 below) gives the weaker estimate given in  Theorem \ref{A(m) bounded}.

In contrast, the average value of $A(m)$ over totients $m\le x$
is clearly $\ge x/V(x) = (\log x)^{1+o(1)}$.
The vast differences between the ``average'' behavior and the
``normal'' behavior is a result of some
totients having enormous multiplicity.

A simple modification of the proof of Theorems \ref{V(x)} and \ref{V_k(x)}
also gives bounds for totients in short intervals.  A real number $\theta$ is
said to be admissible if $\pi(x+x^\theta)-\pi(x) \gg x^\theta/\log x$
 with $x$ sufficiently large.  Here, $\pi(x)$ is the number of primes $\le x$.
 The current record is due to Baker,
Harman and Pintz \cite{BHP}, who showed that $\theta=0.525$ is admissible.

\begin{thm}\label{short interval}
If $\theta$ is admissible, $y\ge x^\theta$
and the multiplicity $k$ is possible, then
$$
V_k(x+y) - V_k(x) \order V(x+y)-V(x) \order \frac{y}{x+y} V(x+y).
$$
Consequently, for every fixed $c>1$, $V(cx)-V(x) \order_c V(x)$.
\end{thm}

Erd\H os has asked if $V(cx) \sim cV(x)$ for each fixed $c>1$, which would
follow from an asymptotic formula for $V(x)$.  The method of proof of
Theorem 1, however, falls short of answering Erd\H os' question.

It is natural to ask what the maximum totient gaps are, in other words
what is the behavior of the function $M(x)=\max_{v_i\le x} (v_i-v_{i-1})$
if $v_1, v_2, \cdots$ denotes the sequence of totients?  Can it be 
shown, for example,
 that for $x$ sufficiently large, that there is a totient between
$x$ and $x+x^{1/2}$?

\subsection{The conjectures of Carmichael and Sierpi\'nski}

 In 1907, Carmichael \cite{C1} announced that for every $m$, the equation
$\phi(x)=m$ has either no solutions $x$ or at least two solutions.  In
other words, no totient can have multiplicity 1.  His proof of this
assertion was flawed, however, and the existence of such numbers
remains an open problem.  In \cite{C2}, Carmichael did show that no
number $m<10^{37}$ has multiplicity 1, and conjectured that no such $m$
exists (this is now known as Carmichael's Conjecture).  Klee \cite{K}
improved the lower bound for a counterexample to $10^{400}$,
Masai and Valette \cite{MV} improved it to $10^{10,000}$ and recently
Schlafly and Wagon \cite{SW} showed that a counterexample must exceed
$10^{10,000,000}$.
An immediate corollary of Theorem \ref{V_k(x)} 
(take $d=1, k=2$ for the first part) is

\begin{thm}\label{Car Conj}
We have
$$
\limsup_{x\to\infty} \frac{V_1(x)}{V(x)} < 1.
$$
Furthermore, Carmichael's Conjecture
is equivalent to the bound
$$
\liminf_{x\to\infty} \frac{V_1(x)}{V(x)} = 0.
$$
\end{thm}

Although this is a long way from proving Carmichael's Conjecture, Theorem
\ref{Car Conj} show that the set of counterexamples
cannot be a ``thin'' subset of $\fancyV$.  Either
there are no counterexamples or a positive fraction of totients are
counterexamples.

The basis for the computations of lower bounds for
a possible counterexample is a lemma of Carmichael
and Klee (Lemma \ref{CK} below), which allows one to show that if $A(m)=1$ then $x$
must be divisible by the squares of many primes. Extending the method
outlined in \cite{SW}, we push the lower bound for a
counterexample to Carmichael's Conjecture further.

\begin{thm}\label{CC lower}  If $A(m)=1$,
then $m \ge 10^{10^{10}}$.
\end{thm}

As a corollary, a variation of an argument of Pomerance \cite{P2} gives the
following.

\begin{thm}\label{liminf V_1/V} We have
$$
\liminf_{x\to\infty} \frac{V_1(x)}{V(x)} \le 10^{-5,000,000,000}.
$$
\end{thm}

The proof of these theorems
motivates another classification of totients.  Let $V(x;k)$
be the number of totients up to $x$, all of whose pre-images are divisible
by $k$.  A trivial corollary to the proof of Theorem \ref{V_k(x)} is

\begin{thm}\label{V(x;k)}
If $d$ is a totient, all of whose pre-images are
divisible by $k$, then 
$$
V(x;k) \gg_\e d^{-1-\e} V(x).
$$
Thus, for each $k$, either $V(x;k)=0$ for all $x$ or
$V(x;k)\gg_k V(x)$.
\end{thm}

In the 1950's, Sierpi\'nski
conjectured that all multiplicities $k\ge 2$ are possible (see \cite{S1}
and \cite{E3}),
and in 1961, Schinzel \cite{S2}
deduced this conjecture from his well-known Hypothesis H.  Schinzel's
Hypothesis H \cite{SS}, a generalization of Dickson's Prime $k$-tuples
Conjecture \cite{D}, states that any set of polynomials $F_1(n),\ldots,
F_k(n)$, subject to certain restrictions, are simultaneously prime for
infinitely many $n$.  Using a much simpler, iterative argument, we show that
Sierpi\'nski's Conjecture follows from the Prime $k$-tuples Conjecture.

\begin{thm}\label{Sierp Conj} The Prime $k$-tuples Conjecture implies 
that for each $k \ge 2$, there is a number $d$ with $A(d)=k$.
\end{thm}

Shortly after \cite{F98} was published, the author and S. Konyagin
proved Sierpi\'nski's conjecture unconditionally for even $k$ \cite{FK}.
The conjecture for odd $k$ was subsequently proved by the author \cite{F99}
using a variant of Lemma \ref{k to k+2} below.

\subsection{The normal multiplicative structure of totients}

Establishing Theorems \ref{V(x)} and \ref{V_k(x)} requires a
determination of what a ``normal''
totient looks like.  This will initially take the form of
a series of linear inequalities in the prime factors of a pre-image of
a totient.  An analysis of these inequalities reveals the normal
sizes of the prime factors of a pre-image of a typical totient.
To state our results, we first define
\be\label{L_0}
L_0 = L_0(x) = \lfloor 2C(\log_3 x - \log_4 x) \rfloor. 
\ee
In a simplified form, we show that for
all but $o(V(x))$ totients $m\le x$, every pre-image $n$ satisfies
\be\label{NS}
\log_2 q_i(n) \sim \rho^i(1-i/L_0) \log_2 x \qquad (0\le i\le L_0),
\ee
where $q_i(n)$ denotes the $(i+1)$st largest prime factor of $n$.
For brevity, we write $V(x;\fancyC)$ for the number of totients $m\le x$
which have a pre-image $n$ satisfying condition $\fancyC$.  Also, let
\[
\b_i = \rho^i (1-i/L_0) \qquad (0\le i\le L_0-1).
\]

\begin{thm}\label{qi normal}
Suppose $1\le i\le L_0$.  (a) If $0 < \eps \le \frac{i}{3L_0}$, then
\[
 V\( x; \left| \frac{\log_2 q_i(n)}{\b_i\log_2 x} - 1 \right| \ge \eps \) 
\ll V(x) \exp \left\{ - \frac{L_0(L_0-i)}{4i} \eps^2 + \log\pfrac{i}{\eps L_0} \right\}.
\]
(b) If $ \frac{i}{3L_0} \le \eps \le \frac18$, then
\[
 V\( x; \left| \frac{\log_2 q_i(n)}{\b_i\log_2 x} - 1 \right| \ge \eps \) 
\ll V(x) \exp \left\{ - \tfrac1{13} L_0 \eps \right\}.
\]
\end{thm}

Using Theorem \ref{qi normal}, we obtain a result about simultaneous approximation of 
$q_1(n), q_2(n), \ldots$.

\begin{thm}\label{normal structure}
Suppose $L_0=L_0(x)$, $0\le g \le \frac13 \sqrt{\frac{L_0}{\log L_0}}$ and 
$0\le h\le \frac12 L_0$.
The number of totients $m\le x$ with a pre-image $n$ not satisfying
\be\label{qi interval}
\left| \frac{\log_2 q_i(n)}{\b_i\log_2 x} - 1 \right| \ge  
g \sqrt{\frac{i \log(L_0-i)}{L_0(L_0-i)}} \qquad
(1 \le i \le L_0 - h)
\ee
 is
$$
\ll V(x) \( e^{-h/96} + e^{-\frac12 g^2\log g} + e^{-\frac{1}{14} g\sqrt{\log L_0}} \).
$$
\end{thm}

 Notice that the intervals
in \eqref{qi interval} are not only disjoint, but the gaps between them are
rather large.  In particular, this ``discreteness phenomenon'' means
 that for any $\eps>0$ and most totients  $m\le x$, no pre-image $n$ has any prime factors
$p$ in the intervals
$$
1-\eps \ge \frac{\log_2 p}{\log_2 x} \ge \rho+\eps, \quad \rho-\eps \ge \frac{\log_2 p}
{\log_2 x} \ge \rho^2+\eps, \text{ etc.}
$$
This should be compared to the distribution of the prime factors of a
normal integer $n\le x$ (e.g. Theorem 12 of \cite{HaT}; see also subsection 1.5 below).

For a preimage $n$ of a typical totient,
we expect each $q_i(n)$ to be ``normal'', that is, $\omega(q_i(n)-1) \approx
\log_2 q_i(n)$, where $\omega(m)$ is the number of distinct prime factors of $m$.
This suggests that for a typical totient $v\le x$,
$$
\Omega(v) \approx \omega(v) \approx (1+\rho+\rho^2+\cdots)\log_2 x = \frac{\log_2 x}{1-\rho}.
$$

\begin{thm}\label{Omega normal}
Suppose $\eta=\eta(x)$ satisfies $0\le \eta \le 1/3$.  Then
$$
\# \left\{ m\in \fancyV(x) :
 \left| \frac{\om(m)}{\log_2 x} - \frac{1}{1-\rho} \right| \ge
\eta \right\} \ll \frac{V(x)}{(\log_2 x)^{\eta/10}}.
$$
Consequently, if $g(x)\to\infty$ arbitrarily
slowly, then almost all totients $m\le x$ satisfy
$$
\left| \frac{\om(m)}{\log_2 x} - \frac1{1-\rho} \right| \le \frac{g(x)}
{\log_3 x}.
$$
Moreover, the theorem holds with $\om(m)$ replaced by $\omega(m)$.
\end{thm}

\begin{cor}\label{Omega normal cor} 
If either $g(m)=\omega(m)$ or $g(m)=\om(m)$, then
$$
\sum_{m\in \fancyV(x)} g(m) = \frac{V(x)\log_2 x}{1-\rho} \( 1 + O \(
\frac1{\log_3 x} \) \).
$$
\end{cor}

By contrast, Erd\H os and Pomerance \cite{EP} showed that the average
of $\om(\phi(n))$, where the average is taken over all $n\le x$, is
$\frac12 (\log_2 x)^2 + O((\log_2 x)^{3/2})$.

\subsection{Heuristic arguments}

As the details of the proofs of these results are very complex,
 we summarize the central ideas here.
For most integers $m$, the
prime divisors of $m$ are ``nicely distributed'', meaning the
number of prime factors of $m$ lying between $a$ and $b$ is about
$\log_2 b - \log_2 a$.  This is a more precise version of the classical result 
of Hardy and Ramanujan \cite{HRa} that most numbers $m$ have
about $\log_2 m$ prime factors.
Take an integer $n$ with prime factorization
$p_0 p_1 \cdots $, where for simplicity we assume $n$ is square-free,
and $p_0>p_1> \cdots$.  By sieve methods it can be shown that for most
primes $p$, the prime divisors of $p-1$ have the same ``nice'' distribution.
If $p_0, p_1, \ldots$ are such ``normal'' primes, it follows
that $\phi(n) = (p_0-1)(p_1-1)\cdots$ has about $\log_2 n - \log_2 p_1$
prime factors
in $[p_1,n]$, about $2(\log_2 p_1 - \log_2 p_2)$ prime factors in
$[p_2,p_1]$, and in general, $\phi(n)$ will have $k(\log_2 p_{k-1} -
\log_2 p_k)$ prime factors in $[p_k,p_{k-1}]$.  That is, $n$ has $k$
times as many prime factors in the interval $[p_k,p_{k-1}]$ as does a
``normal'' integer of its size.
If $n$ has many ``large'' prime divisors, then the prime
factors of $m=\phi(n)$ will be much denser than normal, and the number,
$N_1$, of such
integers $m$ will be ``small''.  On the other hand, the number, $N_2$ of
integers
$n$ with relatively few ``large'' prime factors is also ``small''. 
 Our objective
then is to precisely define these concepts of ``large'' and ``small'' so as
to minimize $N_1+N_2$.

The argument in \cite{MP} is based on the heuristic that a normal totient is
generated from a number $n$ satisfying 
\be\label{NS2}
\log_2 q_i(n) \approx\rho^i \log_2 x 
\ee
for each $i$ (compare with \eqref{NS}).
As an alternative to this heuristic, assuming all prime
factors of a pre-image $n$ of a totient are normal leads to
consideration of a series
of inequalities among the prime factors of $n$.  We show that such
$n$ generate
``most'' totients.  By mapping the $L$ largest prime factors of $n$
(excluding the largest) to a
point in $\RR^L$, the problem of counting the number of such $n\le x$
reduces to the problem of finding the volume of a certain
region of $\RR^L$, which we call the fundamental simplex. 
Our result is roughly
$$
V(x) \approx \frac{x}{\log x} \max_{L} T_L (\log_2 x)^L,
$$
where $T_L$ denotes the volume of the simplex.  It turns out that the maximum
occurs at $L=L_0(x) + O(1)$.
Careful analysis of these inequalities reveals that
``most'' of the integers $n$ for which they are satisfied
satisfy \eqref{NS}.  Thus, the heuristic \eqref{NS2} gives numbers
$n$ for which the smaller prime factors are slightly too large.
The crucial observation that the $L$th largest prime factor
($L=L_0-1$) satisfies
$\log_2 p_L \approx \frac{1}{L}\rho^L \log_2 x$
is a key to determining the true order of $V(x)$.

In Section 2 we define ``normal'' primes and show that most primes
are ``normal''.
The set of linear inequalities used in the aforementioned heuristic are
defined and analyzed in Section 3.
The principal result is a determination of
the volume of the simplex defined by the inequalities, which
requires excursions into linear algebra and complex analysis.
Section 4 is devoted to proving the upper bound for
$V(x)$, and in section 5, the lower bound for $V_k(x)$ is deduced.
Together these bounds establish Theorems \ref{V(x)} and \ref{V_k(x)},
as well as Theorems \ref{short interval}, \ref{Car Conj} and
\ref{V(x;k)} as corollaries.
The distribution of the prime factors of a pre-image of a typical totient
are detailed in Section 6, culminating in the proof of Theorems
\ref{qi normal}--\ref{Omega normal} and Corollary \ref{Omega normal cor}.

In Section 7, we summarize the
computations giving Theorem \ref{CC lower}, present very
elementary proofs of Theorems \ref{liminf V_1/V} and \ref{Sierp Conj},
prove Theorem \ref{A(m) bounded} and discuss 
other problems about $V(x;k)$.  Lastly, Section 8 outlines
an extension of all of these results to more general
multiplicative arithmetic functions such as $\sigma(n)$, the sum of
divisors function.  Specifically, we prove

\begin{thm}\label{general f}
Suppose $f:\NN\to \NN$ is a multiplicative function satisfying
\begin{align}
&\{ f(p)-p:p \text{ prime} \} \text{ is a finite set not containing 0,}
 \label{fp} \\
&\sum_{h\text{ square-full} }
 \frac{h^\del}{f(h)} \ll 1, \qquad \text{ for some } \del>0.
\label{f lower}
\end{align}
Then the analogs of
Theorems \ref{V(x)}--\ref{short interval}, \ref{V(x;k)}, 
\ref{qi normal}--\ref{Omega normal cor} and \ref{vx norm}
hold with $f(n)$ replacing $\phi(n)$, with the exception
of the dependence on $d$ in Theorems \ref{V_k(x)} and
\ref{V(x;k)}, which may be different.
\end{thm}

Some functions appearing in the literature which satisfy the conditions of
Theorem \ref{general f} are $\sigma(n)$, the sum of divisors function,
$\phi^*(n)$, $\sigma^*(n)$ and $\psi(n)$.  Here $\phi^*(n)$ and $\sigma^*(n)$
are the unitary analogs of $\phi(n)$ and $\sigma(n)$, defined by
$\phi^*(p^k)=p^k-1$ and $\sigma^*(p^k)=p^k+1$ \cite{Co},
 and $\psi(n)$ is Dedekind's
function, defined by $\psi(p^k)=p^k+p^{k-1}$.
Now consider, for fixed $a\ne 0$, the function defined by $f(p^k)=(p+a)^k$ 
for $p \ge p_0:=\min\{p: p+a\ge 2\}$ and $f(p^k)=(p_0+a)^k$ for $p < p_0$.
Then the range of $f$ is the multiplicative semigroup generated by the shifted primes
$p+a$ for $p>1-a$.  

\begin{cor}
 For a fixed nozero $a$, let $V^{(a)}(x)$ be the counting function of the multiplicative
semigroup generated by the shifted primes $\{p+a : p+a\ge 2\}$.  Then
\[
 V^{(a)}(x) \asymp_a \frac{x}{\log x} \exp\{ C(\log_3 x-\log_4 x)^2 +D \log_3 x
- (D+1/2-2C)\log_4 x \}.
\]

\end{cor}

One further theorem, Theorem \ref{vx norm}, depends on the definition of the fundamental
simplex, and is not stated until Section 6.

\medskip

\noindent
{\bf Acknowledgement:}  The author is grateful to Paul Pollack for carefully proofreading
of the manuscript and for catching a subtle error in the proof of the lower bound in
Theorem \ref{V(x)}.

%
%
%
%

\section{Preliminary lemmata}

Let $P^+(n)$ denote the largest prime factor of $n$ and
let $\om(n,U,T)$ denote the total number of prime factors $p$ of $n$
such that $U < p \le T$, counted according to multiplicity.
Constants implied by the Landau $O$ and Vinogradov $\ll$ and $\gg$
symbols are absolute unless otherwise specified, and
 $c_1, c_2, \ldots$ will denote absolute constants, not depending
on any parameter.  Symbols in boldface type indicate vector quantities.

A small set of additional
symbols will have constant meaning throughout this paper.
These include the constants $\fancyV$, $\rho$, $C$, $D$, $a_i$,
defined respectively in \eqref{basic def}, \eqref{rho def}, \eqref{C def},
\eqref{D def}, and \eqref{F def}, as well as the constants
$\fancyS_L$, $T_L$, $g_i$ and $g_i^*$,
defined in section 3.  Also included are
the following functions:  the functions defined in \eqref{basic def},
$L_0(x)$ \eqref{L_0}, $F(x)$ \eqref{F def}; the functions
$Q(\alpha)$ and $W(x)$ defined respectively in Lemma
\ref{exp partial} and \eqref{W def} below; and 
$\fancyS_L(\bxi)$, $T_L(\bxi)$, $\curly R_L(\vxi;x)$,
$R_L(\vxi;x)$ and $x_i(n;x)$ defined in section 3.
Other variables are considered ``local'' and may change meaning from 
section to section, or from lemma to lemma. 

A crucial tool in the proofs of Theorems \ref{V(x)} and \ref{V_k(x)}
is a more precise version
of the result from \cite{MP} that for most primes $p$, the larger
prime factors of $p-1$ are nicely distributed (see Lemma \ref{normal lem}
below).  We begin with three basic lemmas.

\begin{lem} \label{exp partial} If $z>0$ and $0<\alpha<1<\b$ then
$$
\sum_{k\le \alpha z} \frac{z^k}{k!}
< e^{(1-Q(\alpha))z}, \qquad
\sum_{k\ge \b z} \frac{z^k}{k!}
< e^{(1-Q(\b))z},
$$
where $Q(\lambda) = \int_1^\lambda \log t\, dt =
\lambda \log(\lambda) - \lambda + 1$.
\end{lem}

\begin{proof} We have
$$
\sum_{k\le \alpha z} \frac{z^k}{k!} =
\sum_{k\le \alpha z} \frac{(\alpha z)^k}{k!} \pfrac{1}{\alpha}^k \le
\pfrac{1}{\alpha}^{\alpha z} \sum_{k \le \alpha z} \frac{(\alpha z)^k}{k!}
< \pfrac{e}{\alpha}^{\alpha z} = e^{(1-Q(\alpha))z}.
$$
The second inequality follows in the same way.
\end{proof}

\begin{lem} \label{Omega lem}
The number of
integers $n \le x$ for which $\om(n) \ge \a \log_2 x$ is
$$
\ll_\a \begin{cases} x (\log x)^{-Q(\a)} & 1<\a<2 \\
x (\log x)^{1-\a \log 2}\log_2 x & \a \ge 2. \end{cases}
$$
\end{lem}

\begin{proof}  This can be deduced from the Theorems in Chapter 0 of
\cite{HaT}. \end{proof}

\begin{lem} \label{large prime small} The number of $n \le x$ divisible
by a number $m \ge \exp\{(\log_2 x)^2\}$ with $P^+(m) \le
m^{2/\log_2 x}$ is $\ll x/\log^2 x$.
\end{lem}
\begin{proof}
Let $\Psi(x,y)$ denote the number of integers $\le x$ which have no prime
factors $> y$.  For $x$ large, standard estimates
(\cite{HiT}, Theorem 1.1 and Corollary 2.3) give
$$
\Psi(z,z^{2/\log_2 x}) \ll z \exp\{-(\log_2 x  \log_3 x)/3 \}
$$
uniformly for $z\ge \exp\{(\log_2 x)^2 \}$.
The lemma follows by partial summation.
\end{proof}

We also need basic sieve estimates (\cite{HRi}, Theorems 4.1, 4.2).

\begin{lem} \label{basic sieve} 
Uniformly for $1.9\le y\le z\le x$, we have
$$
|\{ n\le x : p|n \implies p\not\in (y,z] \}| \ll x \frac{\log y}{\log z}.
$$
\end{lem}

\begin{lem} \label{sieve linear factors}
 Suppose $a_1,\ldots,a_h$ are positive integers and
$b_1,\ldots,b_h$ are integers such that
$$
E = \prod_{i=1}^h a_i \prod_{1\le i<j\le h} (a_ib_j-a_jb_i) \ne 0.
$$
Then
$$
\#\{ n\le x: a_in+b_i \text{ prime } (1\le i\le h) \}
\ll_h \frac{x(\log_2 (|E|+10))^h}{(\log z)^{h}}.
$$
\end{lem}

Next, we examine the normal multiplicative structure of shifted primes
$p-1$.

\begin{defn}
When $S\ge 2$, a prime $p$ is said to be $S$-normal if 
\be\label{1S}
\om(p-1,1,S) \le 2\log_2 S 
\ee
and, for every pair of real numbers $(U,T)$ with $S \le U<T\le p-1$, we have
\be\label{normal}
|\om(p-1,U,T) - (\log_2 T - \log_
2 U)| < \sqrt{\log_2 S \log_2 T}. 
\ee
\end{defn}

We remark that \eqref{1S} and \eqref{normal} imply that
for an $S$-normal prime $p\ge S$,
\be\label{omp-1}
\om(p-1) \le 3\log_2 p.
\ee
This definition is slightly weaker than, and also simpler than, the definition of 
$S$-normal given in \cite{F98}.

\begin{lem} \label{normal lem} Uniformly for $x\ge 3$ and $S\ge 2$,
the number of primes $p \le x$ which are not $S$-normal is 
\[
\ll \frac{x(\log_2 x)^5}{\log x} (\log S)^{-1/6}.
\]
\end{lem}

\begin{proof}
Assume $x$ is sufficiently large and $S \ge \log^{1000} x$, otherwise the
lemma is trivial.  Also, if $\log S > (\log x)^6$, then \eqref{1S} implies
that the number of $p$ in question is
$$
\le x \sum_{n\le x} \frac{(3/2)^{\om(n)-2\log_2 S}}{n}
\ll x \frac{(\log x)^{3/2}}{(\log S)^{2\log(3/2)}} \ll \frac{x}{(\log x)(\log S)^{0.3}}.
$$  
Next, assume $\log S \le (\log x)^6$.
By Lemmas \ref{Omega lem} and \ref{large prime small},
the number of
primes $p\le x$ with either $p<\sqrt{x}$, $q:=P^+(p-1)\le x^{2/\log_2 x}$,
$\om(p-1) \ge 10\log_2 x$ or $p-1$ divisible by
the square of a prime $\ge S$, is $O(x/\log^2 x)$.
Let $p$ be a prime not in these categories, which is also not $S$-normal.
Write $p-1=qb$.  By \eqref{1S}
and \eqref{normal}, either (i) $\om(b,1,S)\ge 2\log_2 S-1$ or (ii) for some
$S\le U<T\le x$, $|\om(b,U,T)-(\log_2 T-\log_2 U)|\ge \sqrt{\log_2 S\log_2 T}
-1$.
By Lemma \ref{sieve linear factors}, for each $b$, the number of $q$ is
$$
\ll \frac{x}{\phi(b)\log^2 (x/b)} \ll \frac{x(\log_2 x)^3}{b\log^2 x}.
$$

If $S\le x$, the sum of $1/b$ over $b$ satisfying (i) is
\begin{align*}
&\le \sum_{\substack{P^+(b')\le S \\
  \om(b')\ge 2\log_2 S-1 }} \frac1{b'} \prod_{S<p\le x} \(1+\frac{1}{p} \)
  \ll \frac{\log x}{\log S} \pfrac32^{1-2\log_2 S} \sum_{P^+(b')\le S}
  \frac{(3/2)^{\om(b')}}{b'} \\
&\ll (\log x)(\log S)^{1/2-2\log(3/2)} \ll (\log x)(\log S)^{-0.3},
\end{align*}
and otherwise the sum is
$$
\le \sum_{\substack{b'\le x \\ \Omega(b')\ge 2\log_2 S-1}} \frac{1}{b'}
\ll \pfrac32^{1-2\log_2 S} \sum_{b'\le x} \frac{(3/2)^{\om(b')}}{b'}
\ll \frac{(\log x)^{3/2}}{(\log S)^{2\log (3/2)}} \ll \frac{\log x}{(\log S)^{0.3}}.
$$

Consider $b$ satisfying (ii).  In particular, $S\le x$.
For positive integers $k$, let $t_k=e^{e^k}$.
For some integers $j,k$ satisfying
$\log_2 S-1 \le j<k \le \log_2 x+1$, we have
\be\label{omtjtk}
|\om(b,t_j,t_k)-(k-j+1)| \ge \sqrt{(k-1)\log_2 S} - 4,
\ee
for otherwise if $t_{j} \le U \le t_{j+1}$ and $t_k
\le T < t_{k+1}$, then $\om(b,t_{j+1},t_k) \le \om(b,U,T) \le \om(b,t_j,t_{k+1}),$ 
implying \eqref{normal}.  Now fix $j,k$ and let 
$h=\sqrt{(k-1)\log_2 S} - 4$.  For any integer $l\ge 0$,
\[
\sum_{\om(b,t_j,t_k)=l} \frac{1}{b} \le \prod_{p\le t_j} \(1+\frac{1}{p}\)
  \prod_{t_k<p\le x} \(1+\frac{1}{p}\) \frac{1}{l!} \Bigg( \sum_{t_j<p\le t_k}  
  \frac{1}{p} \Bigg)^l \ll e^{j-k} \log x \, \frac{(k-j+1)^l}{l!}.
\]
Summing over $|l-(k-j+1)| \ge h$ using Lemma \ref{exp partial}, 
we see that for each pair $(j,k)$, there are
\[
 \ll \frac{x(\log_2 x)^3}{\log x} e^{-(k-j)Q(\beta)}
\]
primes satisfying (ii),  where $\beta=1+\frac{h}{k-j+1}$.
Here we used the fact that $Q(1-\lam)>Q(1+\lam)$ for $0 < \lam \le 1$.
By the integral representation of $Q(x)$, we have $Q(1+\lam)\ge \frac{\lam}{2}
\log(1+\lam)$.  Also, $h\ge 0.99 \sqrt{k\log_2 S} \ge 990.$
If $h\ge k-j+1$, then 
$$
(k-j)Q(\b) \ge \frac{h(k-j)\log 2}{2(k-j+1)} \ge \frac{h\log 2}{4} \ge 
\frac{\log_2 S}{6},
$$
and if $h<k-j+1$, then
$$
(k-j) Q(\b) \ge \frac{(k-j)\log 2}{2} \pfrac{h}{k-j+1}^2 \ge
\frac{h^2}{3(k-j+1)} \ge \frac{\log_2 S}{4}.
$$
 As there are 
$\le (\log_2 x)^2$ choices for the pair $(j,k)$, the proof is complete.
\end{proof}

\begin{lem} \label{divisor squares} There are $O(\frac{x\log_2 x}{Y})$ numbers
$m \in \fancyV(x)$ with either $m$ or some $n \in \phi^{-1}(m)$ divisible by $d^2$
for some $d>Y$.
\end{lem}

\begin{proof}  If $\phi(n)\le x$, then from a standard estimate,
$n\ll x\log_2 x$.  Now $\sum_{d>Y} z/d^2 \ll z/Y$. 
\end{proof}

Our next result says roughly that most totients have a preimage which is $S$-normal 
for an appropriate $S$, and that neither the totient nor preimage has a large square factor or a large number of prime factors.

\begin{defn}
 A totient $m$ is said to be $S$-nice if
\begin{enumerate}
 \item[(a)] $\Omega(m)\le 5\log_2 m$,
\item[(b)] $d^2|m$ or $d^2|n$ for some $n\in \phi^{-1}(m)$ implies $d\le S^{1/2}$,
\item[(c)] for all $n\in \phi^{-1}(m)$, $n$ is divisible only by $S$-normal primes.
\end{enumerate}
\end{defn}

Now let
\be\label{W def}
W(x) = \max_{2\le y\le x} \frac{V(y) \log y}{y}.
\ee

\begin{lem} \label{divisor normal} Uniformly for $x\ge 3$ and $2\le S\le x$, the
number of $m \in \fancyV(x)$  which are not $S$-nice is
$$
O \( \frac{x W(x) (\log_2 x)^6}{\log x} (\log S)^{-1/6} \).
$$
\end{lem}

\begin{proof} We may suppose $S\ge \exp \{ (\log_2 x)^{36} \}$,
for otherwise the lemma is trivial. 
By Lemmas \ref{Omega lem} and \ref{divisor squares}, the number of totients 
failing (a) or
failing (b) is $O(x/\log^2 x)$.
Suppose $p$ is a prime divisor of $n$ for some $n\in
\phi^{-1}(m)$.  If $n=n'p$ then either $\phi(n)=(p-1)\phi(n')$ or
$\phi(n)=p\phi(n')$, so in either case $\phi(n') \le x/(p-1)$.
Let $G(t)$ denote the number of primes $p\le t$ which are not
$S$-normal.  By Lemma \ref{normal lem},
the number of $m$ failing (c) is at most
\begin{align*}
2 \sum_p V\pfrac{x}{p-1} &\ll \sum_p \frac{x W(x/(p-1))}{(p-1)\log(x/p)} \\
&\ll x W(x) \int_{2}^{x/2} \frac{G(t) dt}{t^2\log(x/t)}
\ll \frac{x W(x) (\log_2 x)^6}{\log x} (\log S)^{-1/6}.
\qedhere
\end{align*}
\end{proof}

%
%
%
\section{The fundamental simplex}
%
%
%

For a natural number $n$, write $n=q_1 q_2 \cdots$, where $q_1 \ge q_2 \ge \cdots$,
$q_i$ are prime for $i\le \om(n)$ and $q_i=1$ for $i>\om(n)$.
For $\fancyS \subseteq [0,1]^L$, let
$\curly R_L(\fancyS;y)$ denote the set of integers $n$ with $\om(n)\le L$ and
\[
\( \frac{\max(0,\log_2 q_i)}{\log_2 y}, \ldots, \frac{\max(0,\log_2 q_L)}{\log_2 y}
\)  \in \fancyS,
\]
where $\max(0,\log_2 1)$ is defined to be 0.  Also set
\be\label{RL def}
R_L(\fancyS;y) = \sum_{n \in \curly R_L(\fancyS;y)} \frac{1}{\phi(n)}.
\ee
Heuristically, $R_L(\fancyS;x) \approx
(\log_2 y)^L \vvol(\fancyS).$
Our result in this direction relates $R_L(\fancyS;y)$ to the volume of
perturbations of $\fancyS$.  Specifically, letting $\fancyS + \vv$ denote the translation of $\fancyS$ by the vector $\vv$, for $\e>0$ let
$$
\fancyS^{+\eps} = \bigcup_{\vv \in [-\e,\e]^L} \( \fancyS + \vv \),
\qquad \fancyS^{-\eps} = \bigcap_{\vv \in [-\e,\e]^L} \( \fancyS + \vv \).
$$

\begin{lem}\label{sumvol}
Let $y\ge 2000$, $\e = 1/\log_2 y$ and suppose $\fancyS \subseteq
\{ \vx\in \RR^L : 0\le x_L \le \cdots \le x_1\le 1\}$. 
Then
$$
(\log_2 y)^L \vvol\(\fancyS^{-\e}\) \ll R_L(\fancyS;y) \ll (\log_2 y)^L \vvol\(\fancyS^{+\e}\).
$$
\end{lem}

\begin{proof}
For positive integers $m_1,\ldots,m_L$, let $B(\vm)=\prod_{i=1}^L
 [(m_i-1)\e,m_i\e)$.  If $\fancyB$ is the set of boxes $B(\vm)$ entirely
 contained in $\fancyS$, then the union of these boxes contains
 $\fancyS^{-\e}$.  Moreover, for each box, $m_1>m_2>\ldots>m_L\ge 1$.
For $m\ge 1$, there is at least one prime 
in $I_m := [\exp (e^{m-1}),\exp (e^m))$, thus
\begin{align*}
R_L(\fancyS;y)  &\ge \sum_{B(\vm)\in \fancyB} \prod_{i=1}^L \;\;
\sum_{m_i-1 \le \log_2 p < m_i} \frac{1}{p-1} \\
&=\sum_{B(\vm)\in \fancyB} \prod_{i=1}^L \max\(\exp \{-e^{m_i}\}, 1+O(e^{-m_i})\) 
\gg|\fancyB| \ge (\log_2 y)^L \vvol(\fancyS^{-\e}).
\end{align*}
For the second part, suppose $\fancyS$ is nonempty and let $\fancyB$ 
be the set of boxes $B(\vm)$ which
intersect $\fancyS$, so that their union is contained in $\fancyS^{+\e}$.
For $B(\vm)\in \fancyB$,
let $j_m=|\{i : m_i=m\}|$.  Then
\[
R_L(\fancyS;y)  \le \sum_{B(\vm)\in \fancyB} \prod_{m\ge 1} U(m,j_m),
\qquad U(m,j)=\sum_{r_1\le \cdots \le r_j, r_i\in I_m}
\frac{1}{\phi(r_1\cdots r_j)}.
\]
Here each $r_i$ is prime, except that when $m=0$ we allow $r_i=1$ also.
We have $U(0,j)\le \sum_{P^+(n)\le 13} 1/\phi(n) \ll 1$.  Now suppose $m\ge 1$ and let
$j=j_m$.  For each $r_1,\ldots,r_j$, write $r_1\cdots r_j=kh$, where $(k,h)=1$,
$k$ is squarefree and $h$ is squarefull.  Also let  $\ell = \omega(k)$.
Setting
\[
 t_m=\sum_{\substack{h\text{ squarefull}\\ p|h\implies p\in I_m}} \frac{1}{\phi(h)},
\qquad s_m=\sum_{p\in I_m} \frac{1}{p-1}=1+O(e^{-m}),
\]
we have
\[
 U(m,j)\le \frac{s_m^j}{j!} + t_m \sum_{\ell=0}^{j-2} \frac{s_m^\ell}{\ell!}
\le  \frac{s_m^j}{j!}+ t_m e^{s_m} \le 1 + O(e^{-m}).
\]
We conclude that
\[
 R_L(\fancyS;y) \ll \sum_{B(\vm)\in \fancyB} \prod_{m\ge 1} (1+O(e^{-m})) \ll
|\fancyB| \le (\log_2 y)^L \vvol(\fancyS^{+\e}). \qedhere
\]
\end{proof}

Suppose $\xi_i > 0$ for $0\le i\le L-1$.
Recall \eqref{F def} and
let $\fancyS_L^*(\bxi)$ be the set of $(x_1, \ldots, x_L) \in\RR^L$ satisfying
\begin{align*}
(I_0) & &a_1x_1 + a_2x_2 + \cdots + a_L x_L &\le \xi_0, \hfill \\
(I_1) & &a_1x_2 + a_2x_3 + \cdots + a_{L-1}x_L &\le \xi_1 x_1, \\
\vdots\;\;\; && &\vdots\\
(I_{L-2}) & &a_1 x_{L-1} + a_2 x_{L} &\le \xi_{L-2} x_{L-2}, \\
(I_{L-1}) & & 0\le x_{L} &\le \xi_{L-1} x_{L-1}.
\end{align*}
and let $\fancyS_L(\bxi)$ be the subset of $\fancyS_L^*(\bxi)$ satisfying
$0 \le x_L \le \cdots \le x_1 \le 1$.
Define
$$
T_L^*(\bxi) = \vvol(\fancyS_L^*(\bxi)), \qquad T_L(\bxi) =
\vvol(\fancyS_L(\bxi)).
$$
For convenience, let $\vone=(1,1,\ldots,1)$,
$\fancyS_L = \fancyS_L(\vone)$ (the
``fundamental simplex''), $T_L = \vvol(\fancyS_L)$,
$\fancyS_L^*=\fancyS_L^*(\vone)$, and $T_L^*=\vvol ( \fancyS_L^*)$.
We first relate
$\fancyS_L(\vxi)$ to $\fancyS_L$.  The next lemma is trivial.

\begin{lem}\label{xi1}  
If $\xi_i\ge 1$ for all $i$, and $\vx\in\fancyS_L(\vxi)$,
then $\vy\in\fancyS_L$, where $y_i = (\xi_0\cdots \xi_{i-1})^{-1} x_i$.
If $0<\xi_i\le 1$ for all $i$ and $\vy\in \fancyS_L$, then $\vx\in\fancyS_L(\vxi)$, where $x_i = (\xi_0 \cdots \xi_{i-1})y_i$.
\end{lem}

\begin{lemcor}\label{TLxi} Define
$H(\vxi) = \xi_0^L \xi_1^{L-1} \cdots \xi_{L-2}^2 \xi_{L-1}.$
 We have $T_L \le T_L(\vxi) \le H(\vxi) T_L$
when $\xi_i\ge 1$ for all $i$, and
$H(\vxi) T_L \le T_L(\vxi) \le T_L$ when $0<\xi_i\le 1$ for all $i$.
\end{lemcor}

In applications, $H(\vxi)$ will be close to 1, so we
concentrate on bounding $T_L$.

\begin{lem}\label{TL}
We have
$$
T_L^* \order T_L \order  \frac{\rho^{L(L+3)/2}}{L!}(F'(\rho))^L.
$$
\end{lem}

\begin{lemcor}\label{TL cor} If $H(\vxi) \order 1$, then
$$
T_L(\bxi) \order T_L(\bxi) \order  \frac{\rho^{L(L+3)/2}}{L!}(F'(\rho))^L.
$$
Furthermore, if $L=2C (\log_3 x - \log_4 x) - \Psi$, where $0 \le \Psi \ll
\sqrt{\log_3 x}$, then
\begin{multline*}
(\log_2 x)^L T_L(\vxi) = \exp \{ C(\log_3 x-\log_4 x)^2
 + D\log_3 x - (D+1/2-2C) \log_4 x \\
-\Psi^2/4C - (D/2C-1)\Psi + O(1) \}.
\end{multline*}
If $L=[2C (\log_3 x - \log_4 x)] - \Psi > 0$, then
\begin{multline*}
(\log_2 x)^L T_L(\bxi) \ll \exp \{ C(\log_3 x-\log_4 x)^2
 + D\log_3 x - (D+1/2-2C) \log_4 x \\
-\Psi^2/4C - (D/2C-1)\Psi \}.
\end{multline*}
\end{lemcor}

\begin{proof}
The second and third parts follow from \eqref{C def}, \eqref{D def} and
Stirling's formula.
\end{proof}

To prove Lemma \ref{TL}, we first give a variant of a standard formula for the volume of
tetrahedra, then an asymptotic for a sequence which arises in the proof.

\begin{lem}\label{volume} Suppose $\vv_0, \vv_1, \ldots, \vv_L \in
\RR^L$, any $L$ of which are linearly independent, and
\be\label{v0 sum}
\vv_0 + \sum_{i=1}^L b_i \vv_i = \vz, 
\ee
where $b_i >0$ for every $i$.  Also suppose $\a>0$.
  The volume, $V$, of the simplex
$$
\{ \vx\in \RR^L : \vv_i \cdot \vx \le 0\, (1\le i \le L),
\vv_0 \cdot \vx \le \a \}
$$
is
$$
V = \frac{\a^L}{L! (b_1 b_2 \cdots b_L) |\det(\vv_1, \ldots, \vv_L)|}.
$$
\end{lem}

\begin{proof}
We may assume that $\a=b_1=b_2=\cdots=b_L=1$,
for the general case follows by suitably scaling the vectors $\vv_i$.
The vertices of the simplex are $\vz,\vp_1, \cdots, \vp_L$, where $\vp_i$ satisfies
\[
\begin{cases}
\vp_i \cdot \vv_j = 0 \quad (1\le j\le L, j\ne i); &\\
\vp_i \cdot \vv_0 = 1.&
\end{cases}
\]
Taking the dot product of $\vp_i$ with each side of \eqref{v0 sum} yields
$\vv_i \cdot \vp_i = -1,$ so $\vp_i$
lies in the region $\{\vv_i \cdot \vx \le 0 \}$.  Also, $\vz$ lies in 
the half-plane $\vv_0\cdot \vx \le \a$.
  The given region is thus an
$L$-dimensional ``hyper-tetrahedron'' with volume
$|\det(\vp_1, \cdots,\vp_L)|/L!$, and
$(\vp_1, \cdots, \vp_L) (\vv_1, \cdots, \vv_L)^{T} = -I,$
where $I$ is the identity matrix.  Taking determinants gives the lemma.
\end{proof}

Having $2L-2$ inequalities defining $\fancyS_L$ creates complications
estimating $T_L$, so we devise a scheme where only $L+1$ inequalities
are considered at a time, thus allowing the use of Lemma \ref{volume}.
The numbers $b_i$ occurring in that lemma will come from the sequence
$\{g_i\}$, defined by
\be\label{g def}
g_0=1, \qquad g_i = \sum_{j=1}^i a_j g_{i-j} \quad (i\ge 1). 
\ee

\begin{lem}\label{gh} For every $i\ge 1$,
$\displaystyle |g_i - \lam \rho^{-i}| \le 5$, where 
$\lam=\frac1{\rho F'(\rho)}$.
\end{lem}

\begin{proof}
Write $1-F(z)=(1-z/\rho) l(z)$ and $l(z)=\sum_{n=0}^\infty l_n z^n$.
By \eqref{rho def},
\[
l_n=\rho^{-m}\( 1 - \sum_{k=1}^n a_k \rho^k \) =\sum_{k=1}^\infty \rho^k a_{n+k} >0.
\]
Next consider
$k(z)=(1-z)^2 l(z) = \sum_{n=0}^\infty k_n z^n$.  We have $k_0=1$,
$k_1=l_1-2=\rho^{-1}-a_1-2<0$ and
$$
k_n=l_n-2l_{n-1}+l_{n-2} = \sum_{k=1}^\infty \rho^k \( a_{n+k}-2
 a_{n+k-1}+a_{n+k-2} \) < 0 \qquad (n\ge 2).
$$
Also, $k_n=O(1/n^2)$, and $\sum_{n\ge 1} k_n = -1$.  Thus, $k(z)$ is analytic for
$|z|<1$, continuous on $|z| \le 1$, and nonzero for $|z|\le 1, z\ne 1$.
 Further,
$$
\Re k(z) \ge 1 + k_1 \Re z - (1+k_1) = |k_1| \Re (1-z),
$$
so that for $|z|<1$,
$$
\left|\frac{1}{l(z)}\right| \le \frac{|1-z|^2}{|k_1| \Re (1-z)} \le 
\frac{1}{|k_1|} \max_{|z|=1} \frac{|1-z|^2}{\Re (1-z)} = \frac{2}{|k_1|} < 3.7.
$$
Now let 
$$
e(z) = \sum_{n=0}^\infty \( g_n - \lam \rho^{-i} \) z^n = \frac{1}{1-F(z)} - 
\frac{\lam}{1-z/\rho} = \frac{1/l(z)-1/l(\rho)}{1-z/\rho}.
$$
From the preceding arguments, we see that $e(z)$ is analytic for $|z|<1$ and
continuous on $|z|\le 1$.  By the maximum modulus principle,
$\max_{|z|=1}|e(z)| \le (3.7+\lam)/|1/\rho-1| \le 5$.  By Cauchy's integral formula,
 the Taylor coefficients of $e(z)$ are all
bounded by 5 in absolute value.
\end{proof}

\begin{remark}
The above proof is based on \cite{AMM}, and is much simpler than the original proof 
given in \cite{F98}.
With more work, one can show that
for $i\ge 1$, the numbers $g_i - \lam \rho^{-i}$ are negative, 
increasing and have sum $-1+\lam/(1-\rho)=-0.2938\ldots$
\end{remark}

\begin{proof}[Proof of Lemma \ref{TL}]
 The basic idea
is that $\fancyS_L^*$
is only slightly larger than $\fancyS_L$.  In other words,
the inequalities $1 \ge x_1 \ge \cdots \ge x_{L-1}$ are relatively
insignificant.  Set
\[
\fancyU_0 = \fancyS_L^* \cap \{ x_1 > 1 \}, \qquad
\fancyU_i = \fancyS_L^* \cap \{ x_i < x_{i+1} \} \quad (1\le i\le L-2)
\nonumber
\]
and $V_i = \vvol(\fancyU_i)$.  Evidently
\be\label{TL main}
T_L^* - \sum_{i=0}^{L-2} V_i \le T_L \le T_L^*.
\ee

Let $\ve_1, \cdots, \ve_L$ denote the standard basis for $\RR^L$, i.e.
 $\ve_i \cdot \vx = x_i$.
For $1\le i\le L-2$, set
\be\label{v def}
\vv_i = -\ve_i + \sum_{j=1}^{L-i} a_j \ve_{i+j} 
\ee
and also
$$
\vv_0 = \sum_{j=1}^L a_j \ve_j, \qquad \vv_{L-1} = -\ve_{L-1} +\ve_L,
 \qquad\vv_L = -\ve_L.
$$
For convenience, define
\be\label{gh star}
g_0^*=1,\quad g_i^* = g_i + (1-a_1) g_{i-1} \qquad (i\ge 1).
\ee

Thus, for $1\le j \le L-2$, inequality ($I_j$) may be abbreviated as
$\vv_j \cdot \vx \le 0$.  Also, inequality ($I_0$) is equivalent to
$\vv_0 \cdot \vx \le 1$ and ($I_{L-1}$)
is represented by $\vv_{L-1} \cdot \vx \le 0$ and $\vv_L \cdot \vx \le 0$.
By \eqref{g def}, \eqref{v def}
and \eqref{gh star},
\be\label{e rep}
\ve_i = - \sum_{j=i}^{L-1} g_{j-i} \vv_j - g_{L-i}^* \vv_L.
\ee
It follows that
\be\label{v0 rep}
\vv_0 + \sum_{j=1}^{L-1} g_j \vv_j  + g_L^* \vv_L = \vz.
\ee
Since $|\det(\vv_1, \cdots, \vv_L)| = 1$, Lemma \ref{volume} and
\eqref{v0 rep} give
\be\label{V*}
T_L^* = \frac{1}{L! (g_1 \cdots g_{L-1}) g_L^*}.
\ee
Lemma \ref{gh} now implies the claimed estimate for $T_L^*$.

For the remaining argument, assume $L$ is sufficiently large.
We shall show that
\be\label{Vi sum}
\sum_{i=0}^{L-2} V_i < 0.61 T_L^*,
\ee
which, combined with \eqref{TL main}, \eqref{V*} and Lemma
\ref{gh}, proves Lemma \ref{TL}.

Combining $x_1\ge 1$ with $\vv_0 \cdot \vx \le 1$
gives $\vu \cdot \vx \le 0$, where
$\vu = \vv_0 - \ve_1.$  By
\eqref{e rep} and \eqref{v0 rep},
$$
\vu = \sum_{j=1}^{L-1} (g_{j-1}-g_j) \vv_j + (g_{L-1}^*-g_L^*)\vv_L.
$$
Thus
$$
\vv_0 + \frac{a_1}{1-a_1} \vu + \sum_{j=2}^L b_j \vv_j = \vz,
$$
where
\begin{align*}
b_j &= g_j + \frac{a_1}{1-a_1}(g_j- g_{j-1}) \qquad (2\le j \le L-1), \\
b_L &= g_L^* + \frac{a_1}{1-a_1}(g_L^*-g_{L-1}^*).
\end{align*}
Lemma \ref{gh} implies $b_j > (9/7)g_j$ for large $j$,
In addition, $|\det(\vu, \vv_2, \ldots, \vv_L)| = (1-a_1)$.
By Lemma \ref{volume},
\be\label{V0}
V_0 \ll \frac{1}{L! (b_2 b_3 \cdots b_L)} \ll \pfrac{7}{9}^L T_L^*.
\ee
We next show that
\be\label{Vi}
V_i = \frac{1}{(1-a_1) L! (g_1 \cdots g_{i-1})A_iB_i} \prod_{j=i+2}^{L-1}
\pfrac{1}{g_j+B_ih_{j-i}} \frac{1}{g_L^* + B_i h_{L-i}^*},
\ee
where
$$
A_i = g_i + \frac{g_{i+1}}{1-a_1}, \quad B_i=\frac{g_{i+1}}{1-a_1}, \quad
h_l=g_l-g_{l-1}, \quad h_l^*=h_l+(1-a_1)h_{l-1}.
$$
In $\fancyU_i$ we have ($I_i$) and $x_i \le x_{i+1}$, hence
\[
x_{i+1} \ge \frac{1}{1-a_1} (a_2 x_{i+2} + \cdots + a_{L-i} x_L) 
\ge x_{i+2} + a_2 x_{i+3} + \cdots + a_{L-i-1} x_L.
\]
The condition $\vv_{i+1} \cdot \vx \le 0$ is therefore
implied by the other
inequalities defining $\fancyU_i$, which means
$$
V_i = \vvol\{\vv_0 \cdot \vx \le 1; \vv_j \cdot \vx \le 0\ (1\le
 j\le L, j\ne i+1);\, (\ve_i-\ve_{i+1}) \cdot \vx \le 0 \}.
$$
We note $|\det(\vv_1, \cdots, \vv_i, \ve_i-\ve_{i+1}, 
\vv_{i+2}, \cdots, \vv_L)|=
(1-a_1)$.  It is also easy to show from \eqref{v0 rep} that
$$
\vz = \vv_0 + \sum_{j=1}^{i-1} g_j \vv_j + A_i\vv_i + B_i(\ve_i-\ve_{i+1}) +
\sum_{j=i+2}^L b_j \vv_j,
$$
where $b_j = g_j + B_i h_{j-i}$ for $i+2 \le j \le L-1$, and  $b_L=g_L^*+B_i
h_{L-i}^*.$
An application of Lemma \ref{volume} completes the proof of \eqref{Vi}.

We now deduce numerical estimates for $V_i/T_L^*$.
Using Lemma \ref{gh}, plus
explicit computation of $g_i$ for small $i$, gives $A_i > 4$ for all $i$ and
\begin{align*}
g_j + B_i h_{j-i} &> 1.44 g_j \qquad (i \text{ large, say } i \ge L-100), \\
g_j + B_i h_{j-i} &> 1.16 g_j \qquad (i\ge 1, j\ge i+2), \\
g_L^* + B_i h_{L-i}^* &> 1.44 g_L^* \qquad (i<L-2), \\
g_L^* + B_{L-2} h_2^* &> 1.19 g_L^*.
\end{align*}
From these bounds, plus \eqref{V*} and \eqref{Vi}, it follows that
\begin{align*}
V_{L-2}/T_L^* &< (4 \cdot 1.19)^{-1}, \\
V_i/T_L^* &< (4 \cdot 1.44^{L-i-1})^{-1} \qquad (L-99 \le i\le L-3), \\
V_i/T_L^* &< (4 \cdot 1.44^{99} \cdot 1.16^{L-i-100})^{-1} \qquad
(1 \le  i \le L-100).
\end{align*}
Combining these bounds with \eqref{V0} yields
$$
\sum_{i=0}^{L-2} V_i/T_L^* < O((4/5)^L) + \frac{1}{4} \( \frac{1}{1.19}
 + \frac{1.44^{-2}}{1-1.44^{-1}} +
 \frac{1.44^{-99}}{(1-1/1.16)} \) < 0.61,
$$
which implies \eqref{Vi sum}.  This completes the proof of Lemma
\ref{TL}.
\end{proof}

Important in the study of $\fancyS_L$ and $\fancyS_L^*$ are both global bounds on the
numbers $x_i$ (given below) as well as a determination of where ``most'' of the volume
lies (given below in Lemma \ref{xL most} Section \ref{sec:structure}).

\begin{lem}\label{xz}  Let $x_0=1$.  
If $\vx\in \fancyS_L^*$, then
$x_i \ge g_{j-i} x_j$ for $0\le i\le j\le L$.
If $\vx \in \fancyS_L(\vxi)$
and $\xi_i\ge 1$ for all $i$, then $x_j \le 
4.771 \xi_i \cdots \xi_{j-1} \rho^{j-i}x_i$ for $0\le i<j\le L$. 
\end{lem}

\begin{proof}
Fix $i$ and note that the first inequality is trivial
for $j=i$.
Assume $k\le i-2$ and it holds for $j\ge k+1$.  Then by ($I_k$) and the induction
hypothesis,
$$
x_k \ge \sum_{h=1}^{L-k} a_h x_{k+h} \ge \sum_{h=1}^{i-k} a_h g_{i-k-h} x_i
= g_{i-k} x_i.
$$
By Lemma \ref{gh},
the maximum of $\rho^{-i}/g_i$ is $4.7709\ldots$, occurring at $i=2$.
The second inequality follows by Lemma \ref{xi1}.
\end{proof}

Careful analysis of $\fancyS_L$ reveals that most of the volume occurs
with $x_i \approx \frac{L-i}{L}\rho^i$ for each $i$, with the
``standard deviation'' from the mean increasing with $i$.
This observation plays an important role in subsequent arguments.
For now, we restrict our attention to the variable $x_L$,
which will be useful in estimating sums over numbers $n$, whose $L$
largest prime factors lie in a specific set, and whose other prime factors 
are unconstrained.
Results concerning the size of $x_i$ for $i<L$
will not be needed until section 6. 

\begin{lem}\label{SLxieps}
 Let $L\ge 3$, $\a \ge 2\eps > 0$ and $\xi_i\ge 1$ for each $i$.  
If $\vx \in [ \fancyS_L^*(\bxi) \cap \{ x_L \ge \a \}]^{+\eps}$, then
$\vy \in \fancyS_L^* \cap \{y_L \ge \a' \}$, where
$\a'=(\a-\eps)/(\xi_0'\cdots \xi_{L-1}')$, $\xi_{L-1}'=3\xi_{L-1}$ and
for $1\le i\le L-2$,
\[
 y_i=\frac{x_i}{\xi_0'\cdots \xi_{i-1}'}, 
\qquad \xi_i' = \xi_i\(1 + \frac{10\rho^{L-i}(1+a_1+\cdots+a_{L-i})\xi_0
\cdots \xi_{L-1}}{\a/\eps}\).
\]
\end{lem}

\begin{proof}
 By assumption, for some $\bx' \in \fancyS_L^*(\bxi)$ with $x_L'\ge \a$, 
$|x_i-x_i'|\le \eps$ for all $i$.  By Lemma \ref{xz},
\[
 x_i \ge x_i'-\eps \ge \frac{x_i'}{2} \ge \frac{\rho^{i-L} x_L'}{10\xi_1\cdots \xi_{L-1}}
\ge \frac{\rho^{i-L}\a}{10\xi_0\cdots \xi_{L-1}} \qquad (i\le L-1).
\]
Hence, by $(I_i)$, if $i\le L-2$ then
\begin{align*}
 \sum_{j=1}^{L-i} a_j x_{i+j} &\le \sum_{j=1}^{L-i} a_j(x_{i+j}'+\eps)  \le \xi_i x_i' + 
\eps(a_1+\cdots+a_{L-i}) \\
&\le \xi_i(x_i+\eps(a_1+\cdots+a_{L-i})) \le \xi_i' x_i.
\end{align*}
Lastly, 
\[
 x_{L-1} \ge x'_{L-1}-\eps \ge \xi_{L-1}^{-1} x_L' - \eps \ge
\xi_{L-1}^{-1} \max(\eps,x_L-2\eps) \ge \frac{x_L}{3\xi_{L-1}}=\frac{x_L}{\xi'_{L-1}}.
\]
This shows that $\vx \in \fancyS_L^*(\bxi')$ and $x_L\ge \a-\eps$.  Finally, 
by Lemma \ref{xi1}, $\vy \in \fancyS_L^*$ and $y_L\ge \a'$.
\end{proof}

The next lemma shows that $x_L \approx \rho^L/L$ for most of
$\fancyS_L$, significantly smaller than the global upper
bound given by Lemma \ref{xz}.

\begin{lem}\label{xL most}
(i) If $\a\ge 0$, then
\[
 \vvol(\fancyS_L^* \cap \{x_L\le \a\}) \ll T_L \a L \rho^{-L}
\]
and
\[
 \vvol(\fancyS_L^* \cap \{x_L\ge \a\}) \ll e^{-\a L g_L} T_L.
\]
(ii) If $\a\ge 0$, $\xi_i\ge 1$ for each $i$, $H(\bxi)\le 2$ and
$\eps \le 10 \rho^L/L$, then
\[
 \vvol([\fancyS_L^*(\bxi) \cap \{x_L\ge \a\}]^{+\eps}) \ll e^{- C_0 \a L g_L} T_L
\]
for some absolute constant $C_0>0$.
\end{lem}

\begin{proof}
Consider first $\vx\in \fancyS_L^*\cap\{x_L\le\a\}$.  Since
 $(x_1, \ldots, x_{L-1}) \in \fancyS_{L-1}^*$, 
the volume  is $\le \a T_{L-1}^*$.
Applying Lemma \ref{TL} gives the first part of (i).
Next, suppose $\vx\in\fancyS_L^*\cap\{x_L\ge\a\}$.
If $\a\ge 1/g_L$, the volume is zero by Lemma \ref{xz}.  Otherwise, set
$y_i = x_i - \a g_{L-i}$ for each $i$.
We have $y_{L-1}\ge y_L\ge 0$, $\vv_j\cdot \vy \le 0$
for $1\le j\le L-2$, and $\vv_0 \cdot \vy \le 1 - \a g_L$.
By Lemmas \ref{TL} and \ref{volume}, the volume of such $\vy$ is
$\le (1-\a g_L)^L T_L^* \ll (1-\a g_L)^L T_L$. 
 The second part of (i) now follows.

For (ii), first suppose $\a \ge 2\eps$.  By Lemma \ref{SLxieps}, Corollary 
\ref{TLxi} and part (i),
\[
 \vvol([\fancyS_L^*(\bxi) \cap \{x_L\ge \a\}]^{+\eps}) \le H(\bxi') \vvol(\fancyS_L^*
\cap \{y_L \ge \a'\}) \ll T_L e^{-\a'Lg_L},
\]
where $\a'$ is defined  in Lemma \ref{SLxieps}.
Since $H(\bxi)\le 2$, $H(\bxi')\ll 1$ and hence $\a' \gg \a$.  Next, assume $\a < 2\eps$.
Without loss of generality suppose $\a=0$, since $e^{-2\eps Lg_L} \gg 1$ by Lemma
\ref{gh}, \eqref{gh star} and the assumed upper bound on $\eps$.  For $\vx$ in question, let
$r=\max\{i\le L: x_i \ge 2\eps\}$.  Using Lemma \ref{TL} and part (i),
\begin{align*}
 \vvol([\fancyS_L^*(\bxi) \cap \{x_L\ge \a\}]^{+\eps}) &\ll \sum_{r=0}^L
(2\eps)^{L-r} \vvol\( \fancyS_{r}^*((\xi_0,\ldots,\xi_{r-1})) \) \\
&\ll T_L \sum_{h=0}^L (2\eps)^h \pfrac{T_{L-h}}{T_L} 
\ll T_L  \sum_{h=0}^L \pfrac{2\eps L \rho^{10-L}}{F'(\rho)}^h \ll T_L.
\qedhere
\end{align*}
\end{proof}
%
%
%
\section{The upper bound for $V(x)$}
%
%
%

In this section, we prove that
\be\label{V upper}
V(x) \ll \frac{xZ(x)}{\log x}, \quad Z(x)= \exp\{ C(\log_3 x - \log_4 x)^2
 + D(\log_3 x) - (D+1/2-2C)\log_4 x\}. 
\ee
We begin with the basic tools needed for the proof, which show immediately
the significance of the set $\fancyS_L(\vxi)$.  First, recall the definition
of an $S$-normal prime \eqref{1S}--\eqref{normal}.  Also, factor each positive
integer 
\[ 
n=q_0(n)q_1(n)\cdots, \quad q_0(n)\ge q_1(n) \ge \cdots,
\]
$q_i(n)$ is prime for $i<\om(n)$ and $q_i(n)=1$ for $i \ge \om(n)$.  Define
\be\label{xinx}
x_i(n;x) = \frac{\max(0,\log_2 q_i(n))}{\log_2 x}.
\ee

\begin{lem} \label{AB} Suppose  $y$ is sufficiently large,
$k\ge 2$ and 
$$
1\ge \theta_1 \ge \cdots \ge \theta_{k} \ge \frac{\log_2 S}{ \log_2 y},
$$
where $S\ge \exp\{(\log_2 y)^{36}\}$.
Let $\log_2 E_j = \theta_j \log_2 y$ for
each $j$.  The number
of $S$-nice totients $v\le y$ with a pre-image satisfying
$$
q_j(n) \ge E_j \qquad (1\le j\le k)
$$
is
$$
\ll y(\log y)^{A+B}(\log_2 y)(\log S)^{k\log k} + \frac{y}{(\log y)^2},
$$
where
$$
A = -\sum_{j=1}^k a_j \theta_j, \qquad
B = 4 \sqrt{\frac{\log_2 S}{\log_2 y}}\,
 \sum_{j=2}^{k+1} \theta_{j-1}^{1/2} j\log j.
$$
\end{lem}

\begin{proof}
Let $F=\min(E_1,y^{1/(20\log_2 y)})$, $E_{k+1}=S$, and
$\theta_{k+1}=\log_2 S/\log_2 y$.
Let $m$ be the part of $v$ composed of primes in $(S,F]$.  Then
$m \le F^{\om(v)} \le y^{1/2}.$
By Lemma \ref{basic sieve}, the number of totients with a
given $m$ is
$$
\ll \frac{y}{m} \frac{\log S}{\log F} \le \frac{y}{m}(\log y)^
{\theta_{k+1}-\theta_1} (\log_2 y).
$$
Let 
$$
\d_j=\frac{\sqrt{\log_2 S \log_2 E_{j-1}}}{\log_2 y}
$$
for each $j$.
Since the primes $q_i(n)$ are $S$-normal, by \eqref{normal} 
$$
\om(m,E_j,E_{j-1}) \ge j(\theta_{j-1}-\theta_j-\d_j)
\log_2 y =: R_j \qquad (2\le j\le k+1).
$$
Therefore, the total number, $N$, of totients counted satisfies
$$
N \ll y(\log y)^{\theta_{k+1}-\theta_1} (\log_2 y)\prod_{j=2}^{k+1}
 \sum_{r\ge R_j} \frac{t_j^r}{r!},\
$$
where
$$
t_j = \sum_{E_j < p \le E_{j-1}} \frac1{p}\le (\theta_{j-1}-\theta_j)\log_2 y+1 
:= s_j.
$$
If $\d_j \le \frac12 (\theta_{j-1}-\theta_j)$, then
$$
\frac{s_j}{R_j} \le \frac{1}{j} \( 1 + \frac{3\d_j}{\theta_{j-1}-\theta_j}\)
$$
and Lemma \ref{exp partial} implies
\begin{align*}
\sum_{r \ge R_j} \frac{s_j^r}{r!} \le \pfrac{es_j}{R_j}^{R_j}
&\le (\log y)^{j(\theta_{j-1}-\theta_j-\d_j)(1-\log j + 3\d_j/
(\theta_{j-1}-\theta_j))} \\
&\le (\log y)^{(j-j\log j)(\theta_{j-1}-\theta_j)+(j\log j+2j)\d_j}.
\end{align*}
If  $\d_j > \frac12 (\theta_{j-1}-\theta_j)$, then the sum on $r$ is
$$
\le e^{s_j} \le e (\log y)^{(j-j\log j)(\theta_{j-1}-\theta_j)+(2j\log j)\del_j}.
$$
Therefore, 
$$
N \ll y (\log y)^{ A + B}(\log_2 y)(\log S)^{(k+1)\log(k+1)-k} e^k.
$$
\end{proof}

\begin{lem}\label{Yk} 
Recall definitions \eqref{F def}.
Suppose $k\ge 2$, $0<\omega<1/10$ and $y$ is
sufficiently large (say $y\ge y_0$).  Then the number of totients $v\le y$
with a pre-image $n$ satisfying
$$
a_1 x_1(n;y) + \cdots + a_k x_k(n;y) \ge 1+\omega
$$
is
$$
\ll y (\log_2 y)^6 W(y) (\log y)^{-1-\omega^2/(600k^3\log k)}.
$$
\end{lem}

\begin{proof}  Assume that
\be\label{omega lower}
\omega^2 > 3600 \frac{\log_3 y}{\log_2 y} k^3 \log k,
\ee
for otherwise the lemma is trivial.  Define $S$ by
\be\label{S def}
\log_2 S = \frac{\omega^2}{100k^3 \log k} \log_2 y,
\ee
so that $S\ge \exp\{(\log_2 y)^{36} \}$.
Let $U(y)$ denote the number of totients in question which are $S$-nice.  
By \eqref{S def} and Lemma \ref{divisor normal}, the number of totients not counted by
$U(y)$ is
\[
\ll \frac{y(\log_2 y)^6 W(y)}{\log y} (\log S)^{-1/6} + \frac{y\log_2 y}{S}
\ll y(\log_2 y)^6 W(y) (\log y)^{-1-\omega^2/(600k^3\log k)}.
\]
Let $\e=\omega/10$, $\a=a_1+\cdots+a_k<k\log k$,
and suppose $n$ is a pre-image of a totient counted
in $U(y)$.
  Let $x_i=x_i(n;y)$ for $1\le i\le k$.  Then there are numbers
$\theta_1, \ldots, \theta_k$ so that $\theta_i\le x_i$ for each $i$,
 each $\theta_i$ is
an integral multiple of $\e/\a$, $\theta_1\ge \cdots \ge \theta_k$, and
\be\label{star}
1+\omega-\e \le a_1 \theta_1 + \cdots + a_k \theta_k \le
1+\omega.
\ee
For each admissible $k$-tuple $\vta$,
let $T(\vta;y)$ denote the number of totients counted in $U(y)$ which
have some pre-image $n$ satisfying $x_i(n;y) \ge \theta_i$ for $1\le i\le k$.
Let $j$ be the largest index with $\theta_j \ge \log_2 S/\log_2 y$.
By Lemma \ref{AB},
$$
T(\vta;y) \ll y(\log y)^{A+B}(\log_2 y)(\log S)^{k\log k} + y(\log y)^{-2},
$$
where, by \eqref{star},
$$
A = - \sum_{i=1}^j a_i \theta_i \le - (1+0.9 \omega)
 + \a \frac{\log_2 S}{\log_2 y}
$$
and, by \eqref{S def}, \eqref{star} and the Cauchy-Schwartz inequality,
\[
B \le 4\(\frac{\log_2 S}{\log_2 y}(1+\omega) \sum_{j=2}^{k+1}
\frac{j^2\log^2 j}{a_{j-1}} \)^{1/2}
\le 6\pfrac{k^3 \log k \log_2 S}{\log_2 y}^{1/2} \le \frac{3\omega}{5}.
\]
Also
$$
(\log S)^{2k\log k}=(\log y)^{\omega^2/(50k^2)} \le (\log y)^{\omega/2000}.
$$
Using \eqref{omega lower}, the number of vectors $\vta$ is trivially at most
\[
\pfrac{\a}{\e}^k \le \pfrac{10k\log k}{\omega}^k \le (\log_2 y)^{k/2} 
\le (\log y)^{\omega^2/3000}
\le(\log y)^{\omega/30000}.
\]
Therefore,
$$
U(y) \le \sum_{\vta} T(\vta;y) \ll y (\log y)^{-1-\omega/4},
$$
which finishes the proof.
\end{proof}

Before proceeding with the main argument, we prove a crude upper bound
for $V(x)$ to get things started using the method
of Pomerance \cite{P1}.  For a large $x$ let $x'\le x$ be such that $V(x')=x'W(x)/\log x'$.
Let $v\le x'$ be a totient with pre-image $n$.
By Lemma \ref{divisor squares}, the
 number of $v$ with  $p^2|n$  for some 
prime $p>e^{\sqrt{\log x'}}$  is $O(x'/\log x')$. 
 By Lemma \ref{Yk}, the number of $v$ with
$a_1 x_1(n;x')+a_2 x_2(n;x') > 1.01$ is $O(x' W(x') (\log x')^{-1-c})$ for some $c>0$.
On the other hand, if $a_1 x_1(n;x')+a_2 x_2(n;x') \le 1.01$, then
$x_2(n;x')\le 0.8$.   Write $v=\phi(q_0 q_1) m$, so that
$m\le \exp\{(\log x')^{0.8} \}$, $p_0^2\nmid n$ and $p_1^2\nmid n$. Therefore,
$$
W(x) \ll 1 + \frac{W(x)}{(\log x')^c} + 
\sum_{q_1} \sum_{m} \frac{1}{(q_1-1)m} \ll
(\log_2 x)^2  W(\exp\{(\log x)^{0.8} \}).
$$
Iterating this inequality yields 
\be\label{crude}
W(x) \ll \exp \{ 9(\log_3 x)^2 \}.
\ee
\medskip

\begin{lem} \label{sum 1/v}  We have
$$
\sum_{\substack{ v\in \fancyV \\ P^+(v)\le y }} \frac1{v} \ll W(y^{\log_2 y})
\log_2 y \ll \exp\{10(\log_3 y)^2\}.
$$
\end{lem}
\begin{proof}
Let $f(z)$ denote the number of totients $v\le z$ with $P^+(v)\le y$,
and set $y'=y^{\log_2 y}$.  First suppose $z\ge y'$.  If $v>z^{1/2}$,
then $P^+(v) < v^{2/\log_2 y}$, so Lemma \ref{large prime small} gives
$f(z)\ll z/\log^2 z$.  For $z<y'$, use the trivial bound $f(z)\le V(z)$.
The lemma follows from $\log_2 y' = \log_2 y + \log_3 y$, \eqref{crude}
and partial summation.
 \end{proof}

\begin{proof}[Proof of \eqref{V upper}]
Let $L=L_0(x)$ and for $0\le i\le L-1$, let
\be\label{omega xi}
\omega_i = \frac{1}{10000} \exp \left\{ - \frac{L-i}{40} \right\}, \qquad \xi_i=1+\omega_i.
\ee
Then $H(\vxi)\le 1.1$.  Let $v$ be a generic
totient $\le x$ with a pre-image $n$ satisfying
$n\ge x/\log x$ and $\om(n)\le 10\log_2 x$, and set
$x_i=x_i(n;x)$ and $q_i=q_i(n)$ for $i\ge 0$.  By Lemma \ref{Omega lem},
$$
V(x) \le \sum_{j=0}^{L-2} M_j(x) + N(x) + O\pfrac{x}{\log x},
$$
where $M_j(x)$ denotes the number of such totients $\le x$
with a pre-image satisfying
inequality ($I_i$) for $i<j$ but not satisfying inequality ($I_j$), and
$N(x)$ denotes the number of such totients with every pre-image satisfying
 $\vx\in \fancyS_L(\vxi)$.
By Lemma \ref{Yk} (with $\omega=\omega_0$) and \eqref{crude},
$M_0(x) \ll x/\log x.$
Now suppose $1\le j\le L-2$, and set $k=L-j$.
Let $n$ be a pre-image of a totient counted in $M_j(x)$, and set
$w = q_j q_{j+1} \cdots, m=\phi(w).$
Since ($I_0$) holds, $x_2 \le \xi_0/
(a_1+a_2) <0.9$.  It follows that $q_0 > x^{1/3}$, whence $m<x^{2/3}$.
By the definition of $M_j(x)$ and \eqref{omega xi},
$$
x_j \le \xi_j^{-1} (a_1x_{j+1} + a_2x_{j+2} + \cdots) < \xi_{j-1}^{-1}
(a_1 x_j + a_2 x_{j+1} + \cdots) \le x_{j-1},
$$
whence $q_{j-1}>q_j$ and
$\phi(n)=\phi(q_0 \cdots q_{j-1})m$.
For each $m$, the number of choices for $q_0, \ldots, q_{j-1}$ is
$$
\ll \frac{x}{m\log x} 
R_{j-1}(\fancyS_{j-1}(\xi_0, \ldots,\xi_{j-3});x),
$$
where we set $\fancyS_0=\{0\}$, $\fancyS_1=[0,1]$ and $\fancyS_2=[0,1]^2$.
Let $f(y)$ be the number of $m\le y$.
Define $Y_j$ by $\log_3 Y_j = k/20+1000$.
Since $m$ is a totient, we have $f(y) \le V(y)$, but when $y>Y_j$
we can do much better.
First note that $w\ll y\log_2 y$.
By Lemma \ref{large prime small}, the number of such $w$ with
$P^+(w) < y^{1/\log_2 y}$ is $O(y/(\log y)^3)$.
Otherwise, we have $q_j=P^+(w) \ge y^{1/\log_2 y}$ and
$$
x_j \ge \frac{\log_2 y-\log_3 y}{\log_2 x} \ge \frac{\log_2 y}{\log_2 x} 
\(1 - \frac{k/20+1000}{e^{k/20+1000}} \).
$$
For $0\le i\le k$, let
$$
z_i = x_i(w;y) = \frac{\log_2 x}{\log_2 y} x_{i+j}.
$$
Since ($I_j$) fails and $y>Y_j$, it follows that
\[
a_1 z_1 + \cdots + a_k z_k \ge \frac{\log_2 x}{\log_2 y}(1+\omega_j)x_j
\ge (1+\omega_j/2).
\]
By Lemma \ref{Yk} and \eqref{crude},
 when $y\ge \max(y_0,Y_j)$ we have
\[
f(y) \ll \frac{y W(y) (\log_2 y)^6}{\log y} \exp\left\{ - 
\frac{\omega_j^2}{600k^3\log k} \log_2 y\right\}
\ll \frac{y}{\log y (\log_2 y)^2}.
\]
By partial summation and Lemma \ref{sum 1/v}, 
$$
\sum_{m} \frac{1}{m} \ll 1 + \sum_{m\le Y_j}
 \frac1{m} \ll W(Y_j) \log_2 Y_j \ll \exp\{ k^2/40+O(k) \}.
$$
Therefore, by Lemma \ref{sumvol}, Corollary \ref{TL cor} (with $\Psi=k+1$) and Lemma \ref{xL most} 
(ii) with $\a=0$,
\be\label{Mj}\begin{split}
M_j(x) &\ll \frac{x}{\log x} R_{j-1}(\fancyS_{j-1}(\xi_0,\ldots,\xi_{j-1});x)
 \exp \{ k^2/40 \} \\
&\ll \frac{x}{\log x} T_{j-1}(\log_2 x)^{j-1} \exp \{k^2/40\}\\
&\ll \frac{x}{\log x} \exp\{ -k^2/4-((D+1)/2C-1)k \} Z(x).
\end{split}\ee
Thus
\be\label{Mj sum}
\sum_{j=0}^{L-1} M_j(x) \ll \frac{x}{\log x} Z(x).
\ee

Next, suppose $n$ is a pre-image of a totient
counted in $N(x)$. 
 By Lemma \ref{xz},
$x_L \le 5\rho^L \le \frac{20\log_3 x}{\log_2 x}.$
If $b$ is a nonnegative integer, let $N_b(x)$ be the number of totients
counted in $N(x)$ with a pre-image $n>x/\log^2 x$ satisfying
 $b/\log_2 x \le x_L \le (b+1)/\log_2 x$.
Let $w=q_{L+1} \cdots $ and $q=q_1 \cdots q_{L}w$.
Since $x_2<0.9$ we have $q<x^{2/3}$.  As $\phi(q)\ge \phi(q_1\cdots q_L)\phi(w)$,
for a fixed $q$ the number of possibilities for $q_0$ is
$$
\ll \frac{x}{\log x} \frac{1}{\phi(q)} \le \frac{x}{\log x} \, \frac{1}{\phi(q_1
\cdots q_L) v}, \; v=\phi(w).
$$
By Lemma \ref{sumvol} and Lemma \ref{xL most} (ii),
$$
\sum \frac{1}{\phi(q_1 \cdots q_{L})} \ll R_L(\fancyS_L(\bxi) \cap \{ x_L \ge b/\log_2 x\};x)
\ll Z(x) e^{-C_0 b/4}.
$$
By Lemma \ref{sum 1/v} and \eqref{crude},
$\sum \frac1{v} \ll \exp\{ 10\log^2 b \}.$
Combining these estimates gives
\be\label{Nb}
N_b(x) \ll \frac{x}{\log x} Z(x) \exp\{ -C_0 b/4 + 10\log^2 b \}.
\ee
Summing on $b$ gives $N(x)\ll \frac{x}{\log x} Z(x)$, which
together with \eqref{Mj sum} gives \eqref{V upper}.
\end{proof}

%
%
\section{The lower bound for $V_\kappa(x)$}
%
%
%

Our lower bound for $V_\kappa(x)$ is obtained by constructing a set of numbers
with multiplicative structure similar to the numbers counted by $N(x)$ in
the upper bound argument.  At the core is the following estimate, which
is proved using the lower bound method from \cite{MP}.

\begin{lem}\label{prodpiqi}
Let $y$ be large, $k\ge 1$, 
$e^e \le S \le v_k < u_{k-1} < v_{k-1} < 
u_{k-2} < \cdots < u_0 < v_0=y$, $v_1 \le y^{1/10\log_2 y}$, 
$l\ge 0$,  $1\le r\le y^{1/10}$,
$\del=\sqrt{\log_2 S/\log_2 y}$.  Set $\nu_j=\log_2 v_j/\log_2 y$ 
and $\mu_j=\log_2 u_j/\log_2 y$ for each $j$.
Suppose also that $\nu_{j-1}-\nu_j \ge 2\del$ for $2\le j\le k$,
$1\le d\le y^{1/100}$ and $P^+(d)\le v_k$.
The number of solutions of
\be\label{piqi}
(p_0-1)\cdots (p_{k-1}-1)f_1\cdots f_{l}d = (q_0-1)\cdots (q_{k-1}-1)e \le y/r,
\ee
in $p_0,\ldots,p_{k-1},f_1,\ldots,f_l,q_0,\ldots,q_{k-1},e$ satisfying
\begin{enumerate}
\item $p_i$ and $q_i$ are $S$-normal primes,
neither $p_i-1$ nor $q_i-1$ is divisible by $r^2$ for a prime $r\ge v_k$;
\item $p_i\ne q_i$ and $u_i\le P^+(p_i-1), P^+(q_i-1) \le v_i$ for 
$0\le i\le k-1$; 
\item $P^+(ef_1\cdots f_{l})\le v_k$; $\Omega(f_i)\le 10\log_2 v_k$ for all $i$; 
\item $p_0-1$ has a divisor $\ge y^{1/2}$ which is composed of primes $\ge v_1$;
\end{enumerate}
is
$$
\ll \frac{y}{dr} (c_4\log_2 y)^{6k} (k+1)^{\Omega(d)} 
(\log v_k)^{20(k+l)\log(k+l)+1}
(\log y)^{-2+\sum_{i=1}^{k-1} a_i \nu_i + E},
$$
where $c_4$ is a positive constant and
$E=\del \sum_{i=2}^k (i\log i+i) + 2 \sum_{i=1}^{k-1} (\nu_i-\mu_i).$
\end{lem}

\begin{proof}
We consider separately the prime factors of each shifted prime lying in
each interval $(v_i,v_{i+1}]$. 
For  $0\le j \le k-1$ and $0\le i\le k$, let
\[
s_{i,j}(n) = \prod_{\substack{ p^a\parallel (p_j-1) \\p\le v_i }} p^a, \qquad
s'_{i,j}(n) = \prod_{\substack{ p^a\parallel (q_j-1) \\p\le v_i }} p^a, \qquad
s_i = d f_1\cdots f_l \prod_{j=0}^{k-1} s_{i,j} = e\prod_{j=0}^{k-1} s'_{i,j}.
\]
Also, for $0\le j\le k-1$ and $1\le i\le k$, let 
\[
t_{i,j} = \frac{s_{i-1,j}}{s_{i,j}}, \qquad t'_{i,j} = \frac{s'_{i-1,j}}{s'_{i,j}},
\qquad t_i = \prod_{j=0}^{k-1} t_{i,j} = \prod_{j=0}^{k-1} t'_{i,j}.
\]
For each solution $\fancyA=(p_0,\ldots,p_{k-1},f_1,\ldots,f_l,q_0,\ldots,q_{k-1},e)$
of \eqref{piqi}, let
\begin{align*}
\sigma_i(\fancyA) &= \{s_i;s_{i,0},\ldots,
s_{i,k-1},f_1,\ldots,f_l; s_{i,0}', \ldots, 
s_{i,k-1}',e\},\\
\tau_i(\fancyA) &= \{t_i; t_{i,0},\ldots,t_{i,k-1},1,\ldots,1;
t'_{i,0},\ldots,t'_{i,k-1},1\}.
\end{align*}
Defining multiplication of $(2k+l+2)$-tuples by component-wise multiplication,
we have 
\be\label{sigtau}
\sigma_{i-1}(\fancyA) = \sigma_i(\fancyA) \tau_i(\fancyA).
\ee
Let $\mathfrak S_i$ denote the set of $\sigma_i(\fancyA)$ arising from solutions
$\fancyA$ of \eqref{main m} and $\mathfrak T_i$ the corresponding set of
$\tau_i(\fancyA)$. 
By \eqref{sigtau}, the number of solutions of \eqref{piqi} satisfying the required
conditions is 
\be\label{S0}
|\mathfrak S_0| =
\sum_{\sigma \in \mathfrak S_1} \sum_{\substack{ \tau \in \mathfrak T_1 \\
\sigma\tau \in \mathfrak S_0 }} 1.
\ee
We will apply an iterative procedure based on the identity
\be\label{st iter}
\sum_{\sigma_{i-1} \in \mathfrak S_{i-1}} \frac1{s_{i-1}} =
\sum_{\sigma_i \in \mathfrak S_i} \frac1{s_i}
 \sum_{\substack{ \tau_i \in \mathfrak T_i \\ \sigma_i\tau_i \in 
\mathfrak S_{i-1} } }  \frac1{t_i}.
\ee

First, fix $\sigma_1\in \mathfrak S_1$.  By assumption (4) in the lemma,
$t_{1,0} \ge y^{1/2}$.  Also, $t_1=t_{1,0}=t_{1,0}' \le y/(rs_1)$, $t_1$
is composed of primes $> v_1$, and also $s_{1,0}t_1+1$ and $s'_{1,0}t_1+1$
are different primes.
Write $t_1=t_1'Q$, where $Q=P^+(t_1)$.  
Since $p_0-1$ is an $S$-normal prime, $Q\ge t_1^{1/\om(t_1)} \ge
y^{1/6\log_2 y}$ by \eqref{omp-1}.
Given $t_1'$, Lemma \ref{sieve linear factors} implies that the number of
$Q$ is $O(y(\log_2 y)^6/(r s_1 t_1' \log^3 y))$.  Using Lemma \ref{basic sieve}
to bound the sum of $1/t_1'$, we have
for each $\sigma_1\in \mathfrak S_1$, 
\be\label{sigma1}
\sum_{\substack{ \tau_1 \in \mathfrak T_1 \\ \sigma_1\tau_1 \in \mathfrak S_0 }}
 1 \ll \frac{y (\log_2 y)^6}{rs_1 (\log y)^{2+\nu_1}}.
\ee

Next, suppose $2\le i\le k$, $\sigma_i\in \mathfrak S_i$,
$\tau_i\in \mathfrak T_i$ and $\sigma_i\tau_i\in \mathfrak S_{i-1}$.
By assumption (2),
$$
t_i=t_{i,0} \cdots t_{i,i-1} = t'_{i,0} \cdots t'_{i,i-1}.
$$
In addition, $s_{i,i-1}t_{i,i-1}+1=p_{i-1}$ and 
 $s'_{i,i-1}t'_{i,i-1}+1=q_{i-1}$ 
are different primes. 
Let $Q_1 = P^+(t_{i,i-1})$, $Q_2 = P^+(t'_{i,i-1})$, 
$b=t_{i,i-1}/Q_1$ and $b'=t'_{i,i-1}/Q_2$.

We consider separately $\mathfrak T_{i,1}$, the set of $\tau_i$ with $Q_1=Q_2$
and $\mathfrak T_{i,2}$, the set of $\tau_i$ with $Q_1 \ne Q_2$.
First,
$$
\Sigma_1 := \sum_{\substack{ \tau_i\in \mathfrak T_{i,1} \\ 
\sigma_i\tau_i \in \mathfrak S_{i-1}
} } \frac1{t_i} \le \sum_t \frac{h(t)}{t}\max_{b,b'} \sum_{Q_1} \frac1{Q_1},
$$
where $h(t)$ denotes the number of solutions of
$t_{i,0} \cdots t_{i,i-2} b = t = t'_{i,0} \cdots t'_{i,i-2} b'$,
and in the sum on $Q_1$,
$s_{i,i-1}bQ_1+1$ and $s'_{i,i-1}b'Q_1+1$ are unequal primes.  
By Lemma \ref{sieve linear factors}, the
number of $Q_1\le z$ is $\ll z(\log z)^{-3} (\log_2 y)^3$ uniformly in
$b,b'$.  By partial summation,
\[
\sum_{Q_1\ge u_{i-1}} \frac{1}{Q_1} \ll (\log_2 y)^3 (\log y)^{-2\mu_{i-1}}.
\]
Also, $h(t)$ is at most the number of dual factorizations of $t$ into
$i$ factors each, i.e. $h(t) \le i^{2\om(t)}$.
By \eqref{normal},
$\om(t) \le i(\nu_{i-1} - \nu_{i} +\d)\log_2 y =: I.$
Also, by assumption (1), $t$ is squarefree.
Thus
$$
\sum_t \frac{h(t)}{t} \le \sum_{j\le I} \frac{i^{2j}H^j}{j!},
$$
where
$$
\sum_{v_{i} < p \le v_{i-1}} \frac{1}{p} \le (\nu_{i-1} -
 \nu_{i})\log_2 y + 1 =: H.
$$
By assumption, $\nu_{i-1}-\nu_{i} \ge 2\d,$
hence $I \le \frac32 iH \le \frac34 i^2H$.
Applying Lemma \ref{exp partial}  (with $\a \le \frac34$) yields
\be\label{sum h(t)}
\sum_t \frac{h(t)}{t} \le \pfrac{eHi^2}{I}^I \le (ei)^I 
= (\log y)^{(i+i\log i)(\nu_{i-1}-\nu_{i} + \d)}.
\ee
This gives
\[
\Sigma_1 \ll (\log_2 y)^3 (\log y)^{-2\mu_{i-1} +
(i+i\log i)(\nu_{i-1}-\nu_i + \d) }.
\]

For the sum over $\mathfrak T_{i,2}$, set $t_{i}=tQ_1Q_2$.  
Note that
$$
t Q_2 = t_{i,0} \cdots t_{i,i-2}b, \qquad tQ_1 = t'_{i,0} \cdots 
t'_{i,i-2} b',
$$
so $Q_1|t'_{i,0} \cdots t'_{i,i-2} b'$ and $Q_2|t_{i,0} \cdots t_{i,i-2}b$.
 If we fix the factors divisible by $Q_1$ and by $Q_2$, then the
number of possible ways to form $t$ is $\le i^{2\om(t)}$ as before.
Then
$$
\Sigma_2 := \sum_{\substack{ \tau_i\in \mathfrak T_{i,2} \\ \sigma_i\tau_i\in \mathfrak S_{i-1}}} \frac{1}{t_i}
 \le \sum_t \frac{i^{2\om(t)+2}}{t} \max_{b,b'} \sum_{Q_1,Q_2} \frac1{Q_1 Q_2},
$$
where $s_{i,i-1}bQ_1+1$ and $s'_{i,i-1}b'Q_2+1$ are unequal primes.
By Lemma \ref{sieve linear factors}, the
number of $Q_1\le z$ (respectively $Q_2\le z$) is $\ll z(\log z)^{-2}
(\log_2 y)^2$. By  partial summation, we have
$$
\sum_{Q_1,Q_2} \frac1{Q_1Q_2} = \sum_{Q_1} \frac1{Q_1} \sum_{Q_2}
\frac1{Q_2} \ll (\log_2 y)^4 (\log y)^{-2\mu_{i-1}}.
$$
Combined with \eqref{sum h(t)} this gives
$$
\Sigma_2 \ll i^2 (\log_2 y)^4 (\log y)^{-2 \mu_{i-1} +
(i+i\log i)(\nu_{i-1}-\nu_{i}+\d) }.
$$
By assumption, $i^2\le k^2 \le (\log_2 y)^2$.  Adding
$\Sigma_1$ and $\Sigma_2$ shows that for each $\sigma_i$,
\be\label{siti}
\sum_{\substack{ \tau_i \in \mathfrak T_i \\  \sigma_i\tau_i \in \mathfrak S_{i-1} }} \frac1{t_i}
\ll (\log_2 y)^6 (\log y)^{-2\mu_{i-1} +
 (i\log i + i)(\nu_{i-1}-\nu_{i}+\d)}.
\ee
Using \eqref{S0} and \eqref{st iter} together with the inequalities
\eqref{sigma1} and \eqref{siti},
the number of solutions of \eqref{piqi} is
\[
\ll \frac{y}{r}(c_4 \log_2 y)^{6k}
(\log y)^{-2-\nu_{1} + \sum_{i=2}^k (\nu_{i-1} - \nu_{i}+\d)(i\log i + i)
 - 2\mu_{i-1} } \sum_{\sigma_k \in \mathfrak S_k}
\frac1{s_k},
\]
where $c_4$ is some positive constant.  Note that the exponent 
of $(\log y)$ is $\le -2+\sum_{i=1}^{k-1} a_i \nu_i + E$.

It remains to treat the sum on $\sigma_k$.
Given $s_k'=s_k/d$, the number of possible $\sigma_k$ is
 at most the number of factorizations of $s_k'$ into $k+l$ factors
times the number of factorizations of $ds_k'$ into $k+1$ factors,
which is at most $(k+1)^{\Omega(ds_k')} (k+l)^{\Omega(s_k')}$.
By assumptions (1) and (3), $\Omega(s_k')\le 10(k+l)\log_2 v_k$.
Thus,
\[
\sum_{\sigma_k \in \mathfrak S_k} \frac1{s_k} \le \frac{(k+1)^{\Omega(d)}
(k+l)^{20(k+l)\log_2 v_k}}{d} \sum_{P^+(s_k')\le v_k} \frac{1}{s_k'}
\ll \frac{(k+1)^{\Omega(d)} (\log v_k)^{20(k+l)\log(k+l)+1}}{d}.
\qedhere
\]
\end{proof}


\begin{lem}\label{RL}
If  $\xi_i=1-\omega_i$, $\omega_i=\frac{1}{10(L_0-i)^3}$
 for each $i\le L-2$, then there is an absolute constant
$M_1$ so that whenever $1\le A\le (\log y)^{1/2}$,
 $M=[M_1+2C \log A]$ and $L\le L_0(y)-M$,
we have
\be\label{RL lower}
R_L(\fancyS;y) \gg (\log_2 y)^L T_L, 
\ee
where $\fancyS$ is the subset of $\fancyS_L(\bxi)$
with the additional restrictions
\be\label{qi}
x_{i+1} \le (1-\omega_i)x_i \quad (i\ge 1), \qquad
x_L \ge \frac{A}{\log_2 y}.
\ee
\end{lem}

\begin{proof} 
 By Lemma \ref{sumvol},
$R_L(\fancyS;y) \gg (\log_2 y)^L \vvol(\fancyS^{-\e})$.
For $1\le i\le L-1$, put
$$
\omega_i'=\frac{6(2+(L-i)\log(L-i))\rho^{L-i}}{100+A}, \qquad
\xi_i' = 1 - \omega_i-\omega_i'.
$$
Let $\fancyT$ be the subset of $\fancyS_L(\bxi')$ with the additional
 restrictions $x_{i+1} \le \xi_i' x_i$ for each
$i$ and $x_L\ge(200+A)/\log_2 y$.
Suppose $\vx \in \fancyT$
and $|x_i'-x_i|\le \e$ for each $i$. By Lemma \ref{xz},
$$
x_i' \ge \frac{x_i}{2} \ge \frac{\rho^{L-i}}{6} x_L \ge 
\frac{\rho^{L-i}(A+200)}{6 \log_2 y}.
$$
Thus, for $0\le i\le L-1$,
$$
x_{i+1}' \le x_{i+1}+\e \le \xi_i'(x_i'+\e)+\e \le \(\xi_i'+
\frac{2\e}{x_i'}\)x_i' \le \xi_i x_i'
$$
and
\begin{align*}
a_1 x_{i+1}'+\cdots + a_{L-i} x_L' &\le \xi_i' x_i + \e (a_1+\cdots a_{L-i}) \\
&\le \xi_i'(x_i'+\e)+\e (1+(L-i)\log(L-i)) \le \xi_i x_i'.
\end{align*}
Therefore, $\vx'\in \fancyS$ and hence $\fancyT \subseteq \fancyS^{-\e}$.
Make the substitution $x_i=(\xi_0'\cdots \xi_{i-1}')y_i$ for $1\le i\le L$.
By Lemma \ref{xi1}, $\vy \in \fancyT' := \fancyS_L \cap \{y_L\ge (A+200)/\log_2 y \}$.  By
Lemma \ref{xL most} (i), if $M_1$ is large enough then
\[
\vvol(\fancyS^{-\e}) \ge \vvol(\fancyT) \ge H(\bxi') \vvol(\fancyT')
\ge H(\bxi') \left[ T_L -
O(A \rho^{M} T_L) \right] \gg T_L. \qedhere
\]
\end{proof}


Now we proceed to the lower bound argument for Theorems \ref{V(x)} 
and \ref{V_k(x)}.
Suppose $A(d) = \kappa$ and $\phi(d_i) = d \quad (1\le i \le \kappa)$.  Assume
throughout that $x\ge x_0(d)$.  The variable $k$ is reserved as an index for
certain variables below.
Define
\begin{align}
M  &= M_2 + [(\log d)^{1/9}],  M_2 \text{ is a
sufficiently large absolute constant} \label{M} \\
L &= L_0(x) - M, \label{L def} \\
\xi_i &= 1-\omega_i, \quad \omega_i = \frac{1}{10(L_0-i)^{3}} 
\qquad (0\le i\le L-2).
\label{omega def}
\end{align}
Let $\fancyB$ denote the set of integers $n=p_0 p_1 \cdots p_L>x^{9/10}$
with each $p_i$ prime and
\begin{align}
\phi(n) &\le x/d, \label{n range} \\
(x_1(n;x/d), \cdots, x_L(n;x/d)) &\in \fancyS_L(\vxi), \label{vx} \\
\log_2 p_i &\ge (1+\omega_i) \log_2 p_{i+1} \qquad (0\le i\le L-1),
\label{xi spacing} \\
p_L &\ge \max(d+2, 17). \label{pL lower}
\end{align}
By  Corollary \ref{TL cor} and Lemma \ref{RL} (with $y=x/d$, $A=\log_2 \max(d+2,17)$),
\be\label{B lower}
|\fancyB| \gg \frac{x}{d\log (x/d)} (\log_2 (x/d))^L T_L \gg
\frac{x}{d\log x} (\log_2 x)^L T_L.
\ee

Consider the equation
\be\label{phi eq}
d\phi(n) = \phi(n_1),
\ee
where $n\in \fancyB$.
Let $q_0 \ge q_1 \ge \cdots $ be the prime factors of $n_1$, and
for $j \ge \Omega(n_1)$, put $q_j=1$.  If $n|n_1$,
then none of the primes $q_i$ ($0\le i\le L$)
occur to a power greater than 1, for otherwise
\eqref{pL lower} gives
$\phi(n_1) \ge \phi(n) p_L > \phi(n)d$. 
Also, $P^+(d_i)<p_L$ for all $i$.  Therefore
$\phi(n_1) = \phi(n_1/n) \phi(n) = \phi(n) d,$
which implies $n_1=n d_i$ for some $i$.  These we will call the trivial
solutions to \eqref{phi eq}.  We then have $A(d\phi(n))=\kappa$ for each
$n\in \fancyB$ for which \eqref{phi eq} has no non-trivial solutions, i.e.
solutions with $n\nmid n_1$.  In particular, for such $n$ we have
$\phi(n')\ne \phi(n)$ for $n'\ne n$ and $n'\in \fancyB$.

The numbers $n$ which give rise to non-trivial solutions are grouped
as follows.  For $0\le j\le L$, let $\fancyB_j$ 
be the set of $n\in \fancyB$ such that \eqref{phi eq} holds for some $n_1$
with $q_i = p_i$ ($0\le i\le j-1$) and $p_j\ne q_j$, and such that 
\eqref{phi eq} does not hold for any $n_1$ with $n\nmid n_1$
and $q_i=p_i$ ($0\le i\le j$).
We then have
\be\label{Vl lower}
V_\kappa(x) \ge |\fancyB| - \sum_{j=0}^L |\fancyB_j|. 
\ee

For $n\in\fancyB_j$ with $j\ge 1$, write $n=p_0 n_2 n_3$, where
$n_2=p_1 \cdots p_{j-1}$ 
and $n_3=p_j \cdots p_L$.  When $j=0$, set $n_3=n$.  
If $q_{j-1}=q_j$, then $p_{j-1}|d\phi(n_3)$, which is impossible.
Therefore $q_{j-1}>q_j$ and $\phi(n_1)=\phi(p_0 \cdots p_{j-1})\phi(q_j\cdots)$
and
\be\label{n3n4}
d\phi(n_3) = \phi(n_4)
\ee
has a nontrivial solution $n_4$ (that is, with $n_3\nmid n_4$).  In addition, all such solutions satisfy $P^+(n_4) \ne P^+(n_3)$.
Fix $j$ and let $\curly{A}_j$ be the set of such $n_3$.
It will be useful to associate a particular $n_4$ to each 
$n_3\in \curly{A}_j$ as follows.
Let $v=\phi(n_3)$ for some $n_3\in \curly{A}_j$.
If there is only one such $n_3$, then
take $n_4$ to be the smallest nontrivial solution of \eqref{n3n4}.
Otherwise, suppose there are exactly $k\ge 2$ members of $\curly{A}_j$,
$n_{3,i}$ with $\phi(n_{3,i})=v$ ($1\le i\le k$).
Take a permutation $\sigma$ of
$\{1,\ldots,k\}$ with no fixed point and associate $n_4=d_1 n_{3,\sigma(i)}$
with $n_{3,i}$.  Since $i\ne \sigma(i)$, $n_{3,i}\nmid n_4$, so the associated
$n_4$ is a nontrivial solution of \eqref{n3n4}.
In addition, distinct $n_3\in \curly{A}_j$ are associated
with distinct $n_4$.

For $x$ large, \eqref{vx} and \eqref{xi spacing} imply $p_0>x^{3/4}$. 
By the prime number theorem,
for each fixed $n_2n_3$, the number of choices for $p_0$ is
$O(x/(d\phi(n_2n_3)\log x))$.  Hence
\[
|\fancyB_j| \ll \frac{x}{d \log x} \sum_{n_2} \frac1{\phi(n_2)} \sum_{n_3}
\frac1{\phi(n_3)} \quad (1\le j\le L).
\]
Since $n_2\in \curly{R}_{j-1}(\fancyS_{j-1};x)$ when $j\ge 2$,
Lemma \ref{RL} gives
\[
\sum_{n_2} \frac1{\phi(n_2)} \ll (\log_2 x)^{j-1} T_{j-1} \qquad (1\le j\le L).
\]
To attack the sum on $n_3$, let $B_j(y)$ denote the number of possible
$n_3$ with $\phi(n_3) \le y$.  In particular, $|\fancyB_0| = B_0(x/d)$.  When 
$j\ge 1$, by partial summation,
\be\label{Bj}
|\fancyB_j| \ll \frac{x(\log_2 x)^{j-1} T_{j-1}}{d\log x} 
\( \sum_{\log_3 \phi(n_3) \le M/10} \frac{1}{\phi(n_3)} +
\sum_{\log_3 y > M/10} \frac{B_j(y)}{y^2} \).
\ee

If $M_2$ is large enough, then
\be\label{smally}
\sum_{\log_3 \phi(n_3) \le M/10} \frac{1}{\phi(n_3)} \le \Biggl( 
\sum_{\log_2 p\le e^{M/10}+1} \frac{1}{p-1} \Biggr)^{L-j+1} \le e^{(L-j+1)M/9}.
\ee
We will show below that
\be\label{B_j(y)}
B_j(y) \ll \frac{y}{\log y (\log_2 y)^2} 
\qquad (\log_3 y \ge M/10, 0\le j\le L).
\ee
In particular, $|\fancyB_0| = B_0(x/d) \ll x/(d\log x)$.
Combining \eqref{B_j(y)} with
\eqref{M}, \eqref{Bj}, 
Corollary \ref{TL cor} and \eqref{smally}, we obtain for $j\ge 1$,
\begin{align*}
|\fancyB_j| &\ll \frac{x}{d\log x} (\log_2 x)^{j-1} T_{j-1}
   \exp\{ (L-j+1)M/9 \} \\
&\ll \frac{x}{d\log x} (\log_2 x)^L T_L \exp \{ 
(L-j+1)(M/9 - M/2C-(L-j+1)/4C) \}.
\end{align*}
Summing over $j$ and using Corollary \ref{TL cor}, \eqref{V upper},
\eqref{B lower} and \eqref{Vl lower} gives
$$
V_\kappa(x) \ge \frac{|\fancyB|}{2} \gg_{\e} d^{-1-\e} V(x).
$$
This completes the proof of Theorem \ref{V_k(x)}.  The lower bound
in Theorem \ref{V(x)} follows by taking $d=1$, $\kappa=2$.

We now prove \eqref{B_j(y)}.  For $j\le L-2$, $p_j\le y$, hence by
\eqref{vx},
\be\label{SL-j}
\( \frac{\log_2 p_{j+1}}{\log_2 y}, \cdots, \frac{\log_2 p_L}{\log_2 y} \)
\in \fancyS_{L-j}((\xi_j, \ldots, \xi_{L-1})).
\ee
Thus, by Lemma \ref{xz} and \eqref{pL lower} (and trivially
when $j\ge L-1$),
$$
1\le \log_2 p_L \le 3 \rho^{L-j} \log_2 y,
$$
which implies
\be\label{hupper}
h:=\omega(n_3)=L-j+1 \le 2C \log_3 y + 3.
\ee
Next define
\be\label{SR}
S = \exp\exp\{(\log_3 y)^{10} \}.
\ee

We remove from consideration those $n_3$ satisfying (i)
 $n_3 \le y/\log^2 y$, (ii) $p^2|\phi(n_3)$ for some prime $p>\log^2 y$,
(iii) there is some $m|n_3$ with $m>\exp((\log_2 y)^2)$ and 
$P^+(m)< m^{1/\log_2 y}$; (iv) $n_3$ is divisible by a prime which is
not $S$-normal.  If $p^2|\phi(n_3)$, then either $p^2|n_3$ or $n_3$ is divisible
by two primes $\equiv 1\pmod{p}$.  Thus, the number of $n_3$ satisfying (ii)
is
\[
\le \sum_{p>\log^2 y} \left[ \frac{y}{p^2} + y \( 
\sum_{q<y,q\equiv 1\!\!\!\pmod{p}} \frac{1}{q} \)^2 \right] \ll
\sum_{p>\log^2 y} \frac{y (\log_2 y)^2}{p^2} \ll \frac{y(\log_2 y)^2}{\log^2 y}
\]
by the Brun-Titchmarsh inequality and partial summation.
By Lemma \ref{large prime small}, the number of $n_3$ satisfying (iii)
is $O(y/\log^2 y)$.   By the Hardy-Ramanujan inequality \cite{HRa},
the number of integers $\le t$ which have $h-1$ prime factors
is $O(t(\log_2 t+O(1))^{h-2}/((h-2)!\log t))$ uniformly for $h\ge 2$.  Thus,
the number of $n_3$ satisfying (iv) is
$$
\ll \sum_{\substack{p\le y \\p\text{ not } S-\text{normal}}}
 \frac{y \exp( 2 (\log_3 y)^2 )}{p\log(2y/p)} \ll
\frac{y}{(\log y)(\log_2 y)^2}
$$
by Lemma \ref{normal lem} and partial summation (if $j=L$, then $h=1$
and we use Lemma \ref{normal lem} directly).

For the remaining $n_3$,
since $\log_3 y \ge M/10$, by \eqref{M} we have
\be\label{dineq}
\log d \le (10\log_3y)^9.
\ee 
Let $n_4$ be the unique number associated with $n_3$.
As $\phi(n_4)\le dy$, we have $n_4\ll y(\log y)^{1/3}$.
Now remove from consideration those $n_3$ with (v) 
$p^2|n_4$ or $p^2|\phi(n_4)$ for
some prime $p>\log^2 y$.  The number of such $n_3$ is 
$O(y/\log^{3/2} y)$.
Also remove from consideration those $n_3$ such that (vi)
$n_4$ is divisible by a prime which is not $S$-normal.  
By the way we chose $n_4$, the only way this is possible is if 
$d_1$ has a prime factor which is not $S$-normal, or if
$\phi(n_3) \ne \phi(n_3')$ for $n_3'\in\curly{A}_j$, $n_3'\ne n_3$.   The first case is not possible,  since by \eqref{dineq},
$d_1 \ll d\log_2 d \ll \log S$,
hence for $p|d_1$, $\om(p-1)\le 2\log p\le 2\log d_1\le \log_2 S+O(1)$.
For $n_3$ in the latter category, the numbers $\phi(n_4)$ are distinct
 totients.  Hence, by Lemma \ref{divisor normal} and
\eqref{V upper}, the number of such $n_3$ is
\[
\ll \frac{y(\log_2 y)^6 W(y)}{\log y} (\log S)^{-1/6}
\ll \frac{y}{\log y (\log_2 y)^2}.
\]
Let $B_j^*(y)$ denote the number of remaining $n_3$
(those not satisfying any of conditions (i)--(vi) above), so that
\be\label{BjBj*}
B_j(y) \ll \frac{y}{\log y (\log_2 y)^2} + B_j^*(y).
\ee
If $j\le L-1$, then
$p_{j+1} \cdots p_L \le p_{j+1}^h$, so by  \eqref{M},
\eqref{xi spacing}, \eqref{hupper}, and $M\le 10\log_3 y$,
\[
\log_2 (n_3/p_j) \le \frac{\log_2 p_j}{1+\frac1{10}(h+M-1)^{-3}} + \log h 
\le \log_2 y - 2\log_3 y \le \log_2 y -10.
\]
In particular, since $n_3>y/\log^2 y$, this shows that
\be\label{pjpj+1}
p_j > y^{9/10}, \qquad p_{j+1}<y^{1/(100\log_2 y)}.
\ee
When $j=L$, the first inequality in \eqref{pjpj+1} holds since $n_3>y/\log^2 y$,
and the second inequality is vacuous.  Note that $p$ is $S-$normal for all $p|n_3n_4$, and hence by \eqref{normal},
\be\label{Pp-1}
P^+(p-1) \ge (p-1)^{1/\Omega(p-1)} \ge p^{1/(4\log_2 y)}.
\ee

We now group the $n_3$ counted in $B_j^*(y)$ according to the sizes
of $P^+(p_i-1)$.  Let $J$ be the largest index with $\log_2 P^+(p_J-1)
> (\log_2 y)^{2/3}$.  By \eqref{pjpj+1}, $J\ge j$.
Set $\eps=1/\log_2 y$.
For each $n_3$, there are numbers $\zeta_{j+1},\ldots, \zeta_{J}$, 
each an integral multiple of $\e$, and with $\zeta_i-\eps \le 
\frac{\log_2 P^+(p_i-1)}{\log_2 y} \le \zeta_i$ for each $i$.
Also  set $\zeta_j=1$ and 
\be\label{zetaJ}
\zeta_{J+1} = \min\( \frac{\zeta_J}{1+\omega_J} + \frac{\log_3 y+\log 4}{\log_2 y}, (\log_2 y)^{-1/3} \).
\ee
By \eqref{Pp-1},
\be\label{Ppi-1}
\log_2 P^+(p_i-1) \le \zeta_{J+1} \qquad (i>J).
\ee

By \eqref{vx} and \eqref{hupper},
\be\label{zetasum}
\sum_{i=1}^{J-j} a_i \zeta_{j+i} \le 1 - \omega_j + h^2 \e
\le 1 - \omega_j/2.
\ee

Let $\d=\sqrt{\log_2 S/\log_2 y}$. 
We claim that
\be\label{piqibounds}
\left| \frac{\log_2 P^+(p_i-1) - \log_2 P^+(q_i-1)}{\log_2 y} \right|
\le (2(i-j)+1)\d \qquad (j\le i\le J).
\ee
To see this, fix $i$, let $k=i-j$ and
$$
\a=\frac{\log_2 P^+(p_{i}-1)}{\log_2 y}, \qquad 
\b=\frac{\log_2 P^+(q_{i}-1)}{\log_2 y}.
$$
By \eqref{normal}, if $\b>\a+(2k+1)\d$, then 
$$
(k+1) (\b-\a-\d) \le \frac{\om(\phi(n_3),P^+(p_i-1),P^+(q_i-1)}{\log_2 y} \le k(\b-\a+\d),
$$
a contradiction.  Assuming $\b<\a-(2k+1)\d$ likewise leads to a 
contradiction.  This establishes \eqref{piqibounds}.
In particular, \eqref{piqibounds} implies that $q_{j+1},\ldots,q_J$ exist.

By \eqref{xi spacing}, \eqref{hupper}, \eqref{Pp-1} and $\log_3 y \ge M/10$,
for $j\le i\le J$,
\be\label{zetagap}
\begin{split}
\zeta_i &\ge \frac{\log_2 p_i-\log_3 y-\log 4}{\log_2 y} \ge
(1+\omega_i)(\zeta_{i+1}-\eps)-\frac{\log_3 y+\log 4}{\log_2 y} \\ &\ge 
\zeta_{i+1} + \frac{(\log_2 y)^{-1/3}}
{10(M+h)^3} - 2\eps\log_3 y \ge \zeta_{i+1} + (\log_2 y)^{-0.35}.
\end{split}\ee

We make a further subdivision of the numbers $n_3$, counting
separately those
with $(p_j \cdots p_{J}, q_j \cdots q_{J})=m$.  Let
$B_j(\bz;m;y)$ be the number of $n_3$ counted by $B_j^*(y)$ satisfying
$$
y^{9/10}\le p_j\le y,\quad
\zeta_i-\e \le \frac{\log_2 P^+(p_i-1)}{\log_2 y} < \zeta_i \qquad
(j+1\le i\le J).
$$
Fix $m,\bz$ and  suppose $n_3$ is counted in $B_j(\bz;m;y)$.
Let $p_j\cdots p_J/m = p_{j_0} \cdots p_{j_{k-1}}$, where
$$
j=j_0<j_1< \cdots <j_{k-1} \le J.
$$
Let $\nu_0=1$, for $1\le i\le k-1$ let $\nu_i=\zeta_{j_i}+(2L+1)\d$, and for 
$0\le i\le k-1$ let $\mu_i=\nu_i-(4L+2)\d - \e$.   Also, put 
$\nu_k = \zeta_{J+1} + (2L+3)\delta.$
For brevity, for $0\le i\le k-1$ set $u_i=\exp [ (\log y)^{\mu_i}]$ and 
for $0\le i\le k$ set $v_i=\exp [ (\log y)^{\nu_i} ]$.
By \eqref{Ppi-1}, $P^+(p_i-1)\le v_k$ for $i>J$.  We also claim that
$P^+(q_i-1)\le v_k$ for $i<J$.  If not, then by the $S-$normality of the primes
$p_i$ and $q_i$,
$$
(J-j+2)(\nu_k-\zeta_{J+1}+\delta\log_2 y) \le \Omega(\phi(n_3),
\exp [ (\log y)^{\zeta_{J+1}}) ],v_k) \le (J-j+1)(\nu_k-\zeta_{J+1}+\delta\log_2 y),
$$
a contradiction.
Hence, $B_j(\bz;m;y)$ is at most the number of
solutions of
\be\label{main m}
(p_{j_0}-1) \cdots (p_{j_{k-1}}-1)(p_{J+1}-1)\cdots(p_L-1)d = (q_{j_0}-1) \cdots
(q_{j_{k-1}}-1)e \le y/\phi(m),
\ee
where $P^+((p_{J+1}-1)\cdots(p_L-1)e)\le v_k$,
 and $p_{j_i}$ and $q_{j_i}$
are $S$-normal primes satisfying
\be\label{piqi ranges}
u_i \le P^+(p_{j_i}-1), P^+(q_{j_i}-1) \le v_i \qquad
(0\le i\le k-1).
\ee
By \eqref{pjpj+1}, $\phi(m)\le y^{1/10}$.  Also, $p_j-1$ cannot be divisible by
a factor $b>y^{1/3}$ with $P^+(b)<y^{1/9\log_2 y}$.
Further, \eqref{zetagap} and the definition of $\nu_k$ imply that $\nu_{i-1}-\nu_i \ge 2\delta$ for $2\le i\le k$.
By Lemma \ref{prodpiqi},
\[
B_j(\bz;m;y) \ll \frac{y}{d\phi(m)} (c_4 \log_2 y)^{6L+6} (L+2)^{\om(d)}
(\log v_k)^{20(L+1)^2}(\log y)^{-2+\sum_{i=1}^{k-1}a_i \zeta_{j_i}+E},
\]
where $E\ll \del L^2 \log L$.  By \eqref{zetasum}, the exponent of $\log y$ 
is at most $-1-\omega_j/2+E$.  By \eqref{dineq}, 
$\Omega(d)\ll \log d \ll (\log_3 y)^9$, hence
\[
 B_j(\bz;m;y) \ll \frac{y}{d\phi(m)}  (\log y)^{-1-\omega_j/2}
\exp \{ O( (\log_2 y)^{2/3}(\log_3 y)^2) \}.
\]
Also,
$$
\sum_{m} \frac1{\phi(m)} \le (\log_2 y + O(1))^{L-j} \ll
\exp\{O((\log_3 y)^2 )\}.
$$
The number of possibilities for $\bz$ is at most
$\e^{-L} \le \exp\{ 2(\log_3 y)^2 \}$.  Summing over all
possible $m$ and $\bz$, and applying $\log_3 y\ge M/10$, we have
\begin{align*}
B_j^*(y) &= \sum_{\bz,m} B_j(\bz;m;y)
\ll \frac{y}{\log y} (\log y)^{-\omega_j/2+ (\log_2
y)^{-1/4}}
\\ &\ll \frac{y}{\log y} \exp \left\{ \frac{-\log_2 y}{20(2C\log_3 y+M+3)^3}
 + (\log_2 y)^{3/4} \right\} \\
&\ll \frac{y}{\log y} \exp \{ - (\log_2 y)^{9/10}  \}.
\end{align*}
Combining this with \eqref{BjBj*} completes the proof of
\eqref{B_j(y)}.

%
%
\section{The normal multiplicative structure of totients}\label{sec:structure}
%
%

The proofs of Theorems \ref{V(x)} and \ref{V_k(x)} suggest that for most
totients $m\le x$, all the pre-images $n$ of $m$ satisfy
$(x_1,x_2,\ldots, x_L) \in S_L(\vxi)$
with $L$ near $L_0$ and $\vxi$ defined as in section 4.
We prove such a result below in Theorem \ref{vx norm}, which is an
easy consequence of Theorem \ref{V(x)} and the machinery created for
its proof.  From this, we deduce the normal size of the numbers $q_i(n)$
and establish Theorems \ref{qi normal} and \ref{normal structure}. 
Using these bounds, we deduce the normal order of $\om(m)$
(Theorem \ref{Omega normal} and Corollary \ref{Omega normal cor}).
 
\begin{thm}\label{vx norm}  Suppose $0\le \Psi < L_0(x)$,
$L=L_0-\Psi$ and let
\be\label{xi def}
\xi_i=\xi_i(x) =1+\frac1{10000}e^{-(L_0-i)/40} \qquad (0\le i\le L-1).
\ee
The number of totients $m\le x$ with a pre-image $n$ satisfying
\be\label{x not SL}
(x_1(n;x), \ldots, x_L(n;x)) \not\in \fancyS_L(\vxi) \quad \text{or} \quad x_L(n;x)\le 
\frac{2}{\log_2 x}
\ee
is $\ll V(x) \exp \{ -\Psi^2/4 \}.$
\end{thm}

\begin{proof}  As in Section 4, define $M_j(x)$ to be the number of totients
$m\le x$ with a pre-image satisfying ($I_i$) for $i<j$, but not satisfying
$(I_j)$, where
$\vx=(x_1(n;x), \ldots, x_{L}(n;x)).$
By  Theorem \ref{V(x)}, Corollary \ref{TL cor}, and
 \eqref{Mj}, the number of totients
$m\le x$ with a pre-image $n$ satisfying 
$\vx \not\in \fancyS_L(\vxi)$ is at most
\[
\sum_{j\le L-1} M_j(x) \ll \frac{x}{\log x} Z(x) e^{-\Psi^2/4} \ll 
V(x) e^{-\Psi^2/4}.
\]
Now suppose that $\vx \in \fancyS_L(\vxi)$ and $x_L\le
2/\log_2 x$.  Then $q_L(n) \le e^{e^2}$.
We can assume that $x/\log{x} \leq n \leq 2x\log_2{x}$ and that $n$ is $S$-nice, where
$S=\exp\{(\log_2 x)^{100} \}$, the number of exceptions being $\ll V(x)/\log_2 x$.
By Lemma \ref{Omega lem},
we can also assume that $\Omega(n) \leq 10\log_2{x}$. Put $p_i := q_i(n)$. 
 Lemma \ref{xz} gives $x_3 < 5 \rho^3 < 0.9$, and so $p_2 \leq \exp((\log{x})^{0.9})$. Thus,
\[ 
n/(p_0 p_1 p_2) = p_3 p_4 \cdots \leq \exp(10 (\log_2{x}) (\log{x})^{0.9}) \ll x^{1/100} 
\]
and so $p_0 \geq x^{1/4}$  for large $x$. In particular, $p_0^2 \nmid n$.

Suppose now that $n$ has exactly $L_0-k+1$  prime factors $> e^{e^2}$, 
where we fix $k > \Psi$. Then
\[ 
v = (p_0-1)\phi(p_1 p_2 \cdots p_{L_0-k}) w 
\]
for some integer $w$ satisfying $P^+(w)\le e^{e^2}$. 
Using the prime number theorem to estimate the number of 
choices for $p_0$ given $p_1 \cdots p_{L_0-k}$ and $w$, we obtain that the number of $v$ 
of this form is
\[ 
\ll \frac{x}{\log{x}} \sum_{p_1, \ldots, p_{L_0-k}}\frac{1}{\phi(p_1 \cdots p_{L_0-k})} 
\sum_{w} \frac{1}{w} \ll  \frac{x}{\log{x}} R_{L_0-k}(\vxi_k; x),
\]
since
$p_1 \cdots p_{L_0-k} \in \curly R_{L_0-k}(\vxi_k; x)$, 
where $\vxi_k:= (\xi_0, \dots, \xi_{L_0-k-1})$. 
By Lemma \ref{sumvol},  Corollary \ref{TL cor}, and Lemma \ref{xL most} (ii),
\[  
R_{L_0-k}(\vxi_k; x) \ll (\log_2{x})^{L_0-k} T_{L_0-k} 
\ll Z(x) \exp(-k^2/4C),
\]
hence the number of totients is
\[ \ll \frac{x}{\log{x}} Z(x) \exp(-k^2/4C) \ll V(x) \exp(-k^2/4C).  \]
Summing over $k > \Psi$ gives the required bound.
\end{proof}

We show below that for most of $\fancyS_L$, 
$x_j \approx \rho^j (1-j/L)$ for $1\le j\le L$.
Let $T_L^*(\curly R) = \vvol\left( \fancyS_L^* : \curly R \right)$, recall
definition
\eqref{g def} and Lemma \ref{gh}. Define
\be\label{lambda def}
\lambda_i=\rho^i g_i \quad (i\ge 0),
\qquad\lambda = \lim_{i\to\infty} \lambda_i =\frac
{1}{\rho F'(\rho)} < \frac13.
\ee
By Lemma \ref{gh} and explicit calculation of $g_i$ for small $i$, we have for large $L$
\be\label{gi13}
\frac15 \le \lam_i \le \frac13, \qquad,
\frac{g_i g_{L-i}}{g_L} \le \frac13, \qquad \frac{g_i g_{L-i}^*}{g_L^*} \le \frac13.
\ee

\begin{lem} \label{xi lem}  Suppose $i\le L-2$, $\b>0$, $\a \ge 0$ and define $\theta$ by
\be\label{beta def}
\b = \frac{\rho^i(1-i/L)}{1+\theta}.
\ee
If $\theta>0$, then
\be\label{xi small}
T_L^*(x_i\le\b,x_L\ge \a) \ll T_L \frac{i}{\theta L} \frac{(1+\theta L/i)^i}
{(1+\theta)^L} e^{-L\a g_L}.
\ee
For $-\lam_i \le \th \le 0$,
\be\label{xi big1}
T_L^*(x_i\ge\b,x_L\ge \a) 
\ll T_L  e^{-\frac23 L \a g_L} \exp\left\{ Ki+\frac{\lam_i}{1-\lam_i} L\theta+L(\theta-\log(1+\theta))
\right\},
\ee
where $K=\frac{\lam}{1-\lam}+\log(1-\lam) = 0.0873\ldots$.
If $-i\lambda_i/L<\theta < 0$, then
\be\label{xi big2}
T_L^*(x_i\ge\b,x_L\ge\a) 
\ll  T_L  e^{- \frac23 L\a g_L} \frac{i}{|\theta|L} \exp\left\{
-\frac{L(L-i)}{2i}\th^2 \right\}.
\ee
\end{lem}

\begin{proof}
For each inequality, we show that the region in question lies inside
a simplex for which we may apply Lemma \ref{volume}.  The volume is then
related to $T_L$ via Lemma \ref{TL}.  By Lemma \ref{xz}, $x_L \le 1/g_L$.  
Hence, we may assumer $\a \ge 1/g_L$, else the volumes are all zero.
Also by Lemma \ref{xz}, $x_i \ge \a g_{L-i}$, so we may assume that
$\b > \a g_{L-i}$ in showing \eqref{xi small}.  Also, if $\b \le \a g_{L-i}$,
then $T_L^*(x_i\ge\b,x_L\ge \a)=T_L(x_L\ge \a)$ (i.e., doesn't depend on $\b$),
while the right sides of \eqref{xi big1} and \eqref{xi big2} are each increasing in $\theta$.
Thus, for \eqref{xi big1} and \eqref{xi big2}, we may assume also that $\b > \a g_{L-i}$ as well.

All three inequalities are proved by a common method.
Consider $\vx\in \fancyS_L$ with $x_L\ge \a$
and let $y_j = x_j - \a g_{L-j}$ for each $j$.
Then $\vv_j \cdot \vy = \vv_j\cdot \vx \le 0$  ($1\le j\le L$) and $\vv_0 \cdot \vy \le 1 - \a
g_L$.  Let $\xi=1-\a g_L$ and $\b' = \b-\a g_{L-i}$. 
Set $z_j=y_j-\b'g_{i-j}$
for $j\le i$ and $z_j=y_j$ for $j>i$. By \eqref{g def},
\be\begin{split}\label{z ineq}
\vv_j \cdot \vzz &\le 0 \qquad (1\le j\le L, j\ne i), \\
\vv_i \cdot \vzz &\le \b', \\
\vv_0 \cdot \vzz &\le \xi-\b'g_i.
\end{split}\ee
With these definitions, $x_i\lesseqgtr\b \iff z_i \lesseqgtr 0$.
Hence, for any $A\ge -g_i$, we have 
\be\label{vv0}\begin{split}
\vv_0' \cdot \vzz &\le \xi + A\b', \quad \vv_0' = (\vv_0 + (g_i+A)\vv_i),\\
\vv_j \cdot \vzz &\le 0 \quad (1\le j\le L, j\ne i), \\
\pm \ve_i \cdot \vzz &\le 0.
\end{split}\ee
In the last inequality, we take $+$ for \eqref{xi big1} and \eqref{xi big2}, 
and $-$ for \eqref{xi small}.
By \eqref{g def}, \eqref{e rep} and \eqref{v0 rep}, 
\be\label{Av}
\vv_0' + \sum_{j<i} g_j \vv_j + A \ve_i + \sum_{j=i+1}^{L-1}
(g_j+Ag_{j-i}) \vv_j + (g_L^*+Ag_{L-i}^*)\vv_L = \vz.
\ee
To ensure that each vector on the left of \eqref{vv0} has a positive
coefficient, we assume that $A>0$ for proving \eqref{xi small}, and $A<0$
otherwise.  We may also assume that $\xi-\b' g_i >0$, else the volume in question
is zero by \eqref{z ineq} (each coordinate of $\vzz$ is non-negative).
By Lemma \ref{volume}, together with \eqref{V*}, Lemma \ref{gh} and \eqref{gi13},
\be\label{main TL}\begin{split}
T_L(x_i \lesseqgtr \b,x_L\ge \a) &\le
 T_L^* \frac{g_i}{|A|} (\xi+A\b')^L \prod_{j=i+1}^{L-1}
\( 1 + A \frac{g_{j-i}}{g_j} \)^{-1}
\( 1 + A \frac{g_{L-i}^*}{g_L^*} \)^{-1} \\
&\ll T_L \frac{g_i}{|A|} \frac{(\xi+A\b')^L}{(1+A\rho^i)^{L-i}}.
\end{split}\ee
Since $\b \le \a g_{L-i} \le g_{L-i}/g_L$, if $A>0$ then
\[
 \xi+A\b' = (1+A\b)\(1-\a g_L \frac{1+Ag_{L-i}/g_L}{1+A\b}\) 
\le (1+A\b)(1-\a g_L)\le (1+A\b)e^{-\a g_L}.
\]
Taking $A = \frac{L\th}{i\rho^i}$ gives \eqref{xi small}.
If $-g_i \le A <0$, then by \eqref{gi13},
\[
 \xi+A\b' \le (1+A\b)\(1-\a g_L(1-g_i g_{L-i}/g_L) \) \le (1+A\b)e^{-\frac23 \a g_L}.
\]
 For \eqref{xi big1}, we take
$A=-g_i$, then use 
\[
(1-\lam_i)^{i-L}(1-\b g_i)^L =  \frac{(1-\lam_i)^i}{(1+\th)^L}  
\( 1+\frac{\th+i\lam_i/L} {1-\lam_i} \)^L 
\le \frac{(1-\lam)^i}{(1+\th)^L} \exp\left\{\frac{\th L+i\lam_i}{1-\lam_i}
\right\}
\]
together with $\frac{i\lam_i}{1-\lam_i}=\frac{i\lam}{1-\lam}+O(1)$ 
(a corollary of Lemma \ref{gh}).
Taking  $A = \frac{L\th}{i\rho^i}$ gives  \eqref{xi big2}, since
\be\label{thLi}
\frac{(1+\th L/i)^i}{(1+\th)^L} = \exp\left\{ \frac{L(L-i)}{i}\th^2 \( -\frac12 -
\sum_{j=1}^\infty (-\th)^{j} \frac{L^j+iL^{j-1} + \cdots+i^j}{(j+2)i^j} \) \right\}
\ee
and all summands in the sum on $j$ are positive.
\end{proof}

We  apply Lemma \ref{xi lem} to determine the
size of $q_i(n)$ when $n$ is a pre-image of a ``normal'' totient.
Recall that
$V(x;\fancyC)$ is the number of totients $m\le x$ with
 a pre-image $n$ satisfying $\fancyC$.  An inequality we will use is
\be\label{sum1v-2}
\sum_{\substack{v\in \fancyV \\ P^+(v)\le y}} \frac{1}{v} \ll e^{C(\log_3 y)^2},
\ee
coming from the first part of Lemma \ref{sum 1/v} and Theorem \ref{V(x)}.


\begin{lem} \label{qi size}
Suppose $x$ is large, $\b>0$, and $1\le i\le L_0=L_0(x)$. 
Define $\th$ by $(1+\th)\b=\rho^i(1-i/L_0)$.

(a) If $0<\th \le \frac{i}{3L_0}$, then $\ds
 V\(x;\tfrac{\log_2 q_i(n)}{\log_2 x}\le \b\) \ll 
V(x) \frac{i}{\th L_0} \exp \left\{ - \frac{L_0(L_0-i)}{4i}\th^2 \right\}.$

(b) If $\frac{i}{3L_0} \le \th \le \frac18$, then
$\ds V\(x;\tfrac{\log_2 q_i(n)}{\log_2 x}\le \b\)  \ll V(x) e^{-\th L_0/13}$.

(c) If $-\frac13 \le \th < -0.29 \frac{i}{L_0}$, then
$\ds V\(x;\tfrac{\log_2 q_i(n)}{\log_2 x}\ge \b\)  \ll V(x) e^{\th L_0/10}.$

(d) If $-\frac{i\lam_i}{L_0} \le \th < 0$, then
$\ds V\(x;\tfrac{\log_2 q_i(n)}{\log_2 x}\ge \b\) \ll V(x) \tfrac{i}{|\th|L_0}
\exp \left\{ - 0.49\tfrac{L_0(L_0-i)}{i}\th^2 \right\}.$
\end{lem}

\begin{proof}
Let $A$ be a sufficiently large, absolute constant.  We may assume that
\be\label{prelim}
\begin{split}
A\le i\le L_0-A, \; |\th| \ge A \pfrac{i}{L_0(L_0-i)}^{1/2} &\qquad \text{ for (a) and (d),}\\
|\th| \ge \frac{A}{L_0} &\qquad   \text{ for (b) and (c)},
\end{split}
\ee
for otherwise the claims are trivial.  Put $\Psi=\cl{|\th|\sqrt{\frac{2L_0(L_0-i)}{i}}}$
for (a) and (d), and put $\Psi=\cl{\sqrt{2|\th| L_0}}$ for parts (b) and (c).
Let $L=L_0-\Psi$.
By \eqref{prelim}, for the range of $\th$ given in each part, we have $i\le L-2$.
Define $\xi_i$ by \eqref{xi def}.  By Theorem \ref{vx norm}, the number
of totients $m\le x$ with a preimage $n$ satisfying $\vx\not \in \fancyS_L(\bxi)$,
 $x_L\le \frac{2}{\log_2 x}$ or $m<\frac{x}{\log x}$ is $O(V(x)e^{-\frac14 \Psi^2})$.
Let $\fancyS=\fancyS_L(\bxi) \cap \{x_i\le \b\}$ for (a) and (b), and
$\fancyS=\fancyS_L(\bxi) \cap \{x_i\ge \b\}$ for (c) and (d).
As in the proof of \eqref{Nb}, for $b\ge 2$ let $N_b(x)$ be the number of totients
 for which $n>\frac{x}{\log x}$, 
$\vx \in \fancyS$, and $\frac{b}{\log_2 x} \le x_L < \frac{b+1}{\log_2 x}$.
By the argument leading to \eqref{Nb} and using \eqref{sum1v-2},
\be\label{Vb1}
V\(x;\frac{\log_2 q_i(n)}{\log_2 x} \lesseqgtr \b \) \ll V(x) e^{-\Psi^2/4} + \frac{x}{\log x}
\sum_{b\ge 2} e^{C\log^2 b} R_L\( \fancyS \cap \left\{x_L\ge \frac{b}{\log_2 x} \right\};x\).
\ee
By Lemma \ref{sumvol}, 
\[
 R_L\( \fancyS \cap \left\{x_L\ge \frac{b}{\log_2 x} \right\};x\) \ll (\log_2 x)^L \vvol \left[
\fancyS \cap \{x_L\ge b/\log_2 x\} \right]^{+\eps}, \quad \eps= \frac1{\log_2 x}.
\]
Let $\a=\frac{b}{\log_2 x}$.  By Lemma \ref{SLxieps} ($\a'$, $y_j$ and $\xi_j'$ defined here), $\vy \in \fancyS_L^*$, 
$y_i \lesseqgtr \b'$ and $y_L \ge \a'$, where
\be\label{beta'}
\b'=\frac{\b}{\xi_0'\cdots \xi_{i-1}'}=\b \( 1 - O \( e^{-(L_0-i)/40} \)\).
\ee
By Lemma \ref{xi1} and Corollary \ref{TLxi},
\be\label{STL}
  \vvol \left[\fancyS \cap \{x_L\ge b/\log_2 x\} \right]^{+\eps} \ll 
T_L^* \( x_i \lesseqgtr \b', x_L \ge \a' \).
\ee
Define $\th'$ by $1+\th'=(1+\th)\xi_0'\cdots \xi_{i-1}'$, so that
$\b'(1+\th') = \rho^i(1-i/L)$.
By \eqref{beta'}, $\th'-\th=(1+\th)(\xi_0'\cdots \xi_{i-1}'-1)\ll e^{-\frac1{40}(L_0-i)}$.
By \eqref{prelim}, if $A$ is large enough then
\be\label{theta'}
0 < \th'-\th \le A e^{-\frac1{40}(L_0-i)} \le \frac{|\th|}{1000}.
\ee
We now apply  Lemma \ref{xi lem} (with $\b,\th$
replaced by $\b',\th'$).  For parts (a) and (b), \eqref{theta'} implies
$0<\th'\le \frac17$ and we may apply \eqref{xi small}.  For (c), \eqref{theta'} implies
$-\frac18 \le \th' \le -0.288 \frac{i}{L_0}$ and we apply \eqref{xi big1}.  For (d),
\eqref{theta'} gives $-\frac{i\lam_i}{L_0} \le \th' < 0$ and we apply \eqref{xi big2}.
Combining these estimates with \eqref{STL}, we arrive at
\be\label{Vb2}
 R_L\( \fancyS \cap \left\{x_L\ge \frac{b}{\log_2 x} \right\};x\) \ll (\log_2 x)^L T_L
B e^{-\frac23 \a' g_L},
\ee
where
\[
 B=\begin{cases}
     \frac{i}{\th' L} \frac{(1+\th' L/i)^i}{(1+\th')^L} & \text{ for (a),(b)} \\
\exp \left\{ Ki + \frac{\lam_i}{1-\lam_i}\th' L +L(\th'-\log(1+\th')) \right\} & \text{ for (c)} \\
\frac{i}{(-L\th')} \exp \left\{ -\frac{L(L-i)}{2i}(\th')^2  \right\}
& \text{ for (d)}.
   \end{cases}
\]
By \eqref{L_0} and Lemma \ref{gh}, we have $\a' L g_L \gg \a L \rho^{-L} \gg \rho^{-\Psi}$.
Hence, for some absolute constant $C_1>0$,
\[
 \sum_{b\ge 2} e^{C\log^2 b-\frac23\a'Lg_L} \ll \rho^{-\Psi} \sum_{k\ge 0} e^{C\log^2 ((k+1)\rho^{-\Psi})
-C_1 k}= \exp \left\{ \frac{\Psi^2}{4C}+O(\Psi) \right\}.
\]
Since Corollary \ref{TL cor} implies that $(\log_2 x)^L T_L \ll Z(x) \exp \{ -\Psi^2/(4C)+O(\Psi)
\}$, inequalities \eqref{Vb1} and \eqref{Vb2} now imply
\[
 V\(x;\frac{\log_2 q_i(n)}{\log_2 x} \lesseqgtr \b \) \ll V(x) \left[ e^{-\frac14 \Psi^2} +
B e^{O(\Psi)} \right].
\]
To complete part (a), observe that the absolute value of the summands in \eqref{thLi} 
(with $\th$ replaced by $\th'$) are decreasing.  From the definition of $\Psi$ and \eqref{theta'},
we obtain $O(\Psi) \le \frac{L_0(L_0-i)}{100i}\th^2+O(1)$ and
\begin{align*}
 B &\le \exp \left\{ \frac{L(L-i)}{i}(\th')^2 \( -\frac12
+ \frac{L+i}{3i}\th' \) \right\} \le \exp \left\{ - \frac{5(\th')^2 L(L-i)}{18i} \right\} \\
&\le  \exp \left\{ - 0.27 \frac{L(L-i)}{i} \th^2 \right\} \ll
 \exp \left\{ -0.26 \frac{L_0(L_0-i)}{i}\th^2 \right\}.
\end{align*}
this gives part (a) of the lemma.  For (b), \eqref{theta'} implies $\th'L/i \ge 0.33$,
so $i\log(1+\th'L/i) \le 0.08642 L \th'$.  Also, $\log(1+\th')\ge 0.9423\th'$.  Therefore,
$B \le e^{-0.0781L\th'} \ll e^{-0.077L_0\th},$
whence $Be^{O(\Psi)} \ll e^{-\frac{1}{13}L_0 \th}$.  For (c), we use
$\th'-\log(1+\th') \le 0.0683\th'$.  If $i\le 100$, $Ki=O(1)$ and
$\frac{\lam_i}{1-\lam_i} \ge \frac{\lam_1}{1-\lam_1}\ge 0.265$, and for $i>100$,
$Ki\le 0.302(-L\th')$ and $\frac{\lam_i}{1-\lam_i} \ge 0.4781$.  In either case,
$B \ll e^{0.106L\th'}$ and therefore $Be^{O(\Psi)}\ll e^{\frac1{10}L_0 \th}$ by \eqref{theta'}.
Finally, part (d) follows from \eqref{theta'} by similar calculations to those in part (a).
\end{proof}

\begin{proof}[Proof of Theorem \ref{qi normal}]
Let $x_i=\frac{\log_2 q_i(n)}{\log_2 x}$.
Consider first the case $0\le \eps \le \frac{i}{3L_0}$.
If $x_i\le (1-\eps)\b_i \le \frac{\b_i}{1+\eps}$, take $\th=\eps$ in Lemma \ref{qi size} (a).
If $x_i\ge (1+\eps)\b_i$, take $\th=-\frac{\eps}{1+\eps}\in [-\eps,-\frac34 \eps]$.
Use  Lemma \ref{qi size} (d) if $\th \ge -\frac{i\lam_i}{L_0}$ and  Lemma \ref{qi size} (c)
otherwise.  This yields the desired bounds, since in the latter case $\th \ge -\frac{4i}
{10(L_0-i)}$.

Next, assume $\frac{i}{3L_0} \le \eps \le \frac18$. 
If $x_i\le (1-\eps)\b_i$, take $\th=\eps$ in  Lemma \ref{qi size} (b).  If $x_i\ge (1+\eps)\b_i$,
take $\th=-\frac{\eps}{1+\eps}\in [-\eps,-\frac89 \eps]$ in Lemma \ref{qi size} (c).
We may do so since $\th \le -0.29 \frac{i}{L_0}$.
\end{proof}

\begin{proof}[Proof of Theorem \ref{normal structure}]
Assume $g\ge 10$ and $h\ge 10$, for otherwise
the conclusion is trivial.  Let
$$
\e_i = g \sqrt{\frac{i\log(L_0-i)}{L_0(L_0-i)}} \qquad (1\le i\le L_0-h)
$$
and let $N_i$ be the number of totients $\le x$ with a preimage satisfying
$|\frac{\log_2 q_i(n)}{\b_i\log_2 x} - 1| \ge \eps_i$.
First, suppose that $\eps_i \le \frac{i}{3L_0}$, and let $k=L_0-i$.
We have $\frac{k}{\log k} \ge 4g^2$, for if not, then $k<4g^2\log L_0 < \frac12 L_0$
and consequently $\eps_i > g\sqrt{\frac{\log k}{2k}} > g^2 > 10$.  By Theorem
\ref{qi normal},
\[
 N_i \ll V(x) \exp \left[ - \frac{g^2 \log (L_0-i)}{4} + \frac12 \log\pfrac{i(L_0-i)}{g^2 L_0}
\right] \ll V(x) (L_0-i)^{\frac12-\frac14 g^2}.
\]
Summing over $i\le L_0-4g^2$ and using $g\ge 10$, we obtain
\be\label{normal-1}
\sum_{\eps_i\le i/(3L_0)} N_i \ll V(x) (4g^2)^{\frac32 - \frac14 g^2} \ll V(x) g^{-\frac12 g^2}.
\ee

Next, suppose that $\frac{i}{3L_0} < \eps_i \le \frac18$.  
Since $i \le 9g^2 \frac{L_0 \log(L_0-i)}{L_0-i} \le 18g^2\log L_0$, 
Theorem \ref{qi normal} gives
\be\label{normal-2}
\sum_{i/(3L_0)<\eps_i\le 1/8} N_i \ll V(x) g^2 (\log L_0) e^{-\frac{g}{13}\sqrt{\log L_0}}
\ll V(x) e^{-\frac{g}{14}\sqrt{\log L_0}}.
\ee

Finally, if $\eps_i>\max(\frac{i}{3L_0},\frac18)$, then
$|\frac{\log_2 q_i(n)}{\b_i\log_2 x} - 1| \ge \eps_i' := \max(\frac{i}{3L_0},\frac18)$.  By Theorem \ref{qi normal},
\be\label{normal-3}
\begin{split}
\sum_{\eps_i>\max(i/(3L_0)),1/8)} N_i &\ll V(x) \Bigg( L_0 e^{-\frac{1}{104}L_0} +
\sum_{\frac38L_0<i\le L_0-h} \exp\left[ - \frac{L_0(L_0-i)}{4i}\pfrac{i}{3L_0}^2\right] \Bigg)\\
&\ll V(x) \( e^{-\frac{1}{105}L_0} + \sum_{i\le L_0-h} e^{-\frac{1}{96}(L_0-i)} \)
\ll V(x) e^{-\frac{h}{96}}.
\end{split}\ee
Together, inequalities \eqref{normal-1}--\eqref{normal-3} give Theorem \ref{normal structure}.
\end{proof}


\begin{proof}
[Proof of Theorem \ref{Omega normal}] Assume $\eta \ge \frac{1000}{\log_3 x}$, 
for otherwise the theorem is trivial. 
Let $\Psi=\Psi(x)=\cl{\sqrt{\eta\log_3 x}}$, $L=L_0(x)-\Psi$, 
define $\xi_i$ by \eqref{xi def} and set
$S=\exp\{(\log_2 x)^{100}\}$.
Let $n$ be a generic pre-image of a totient $m\le x$, and set
$q_i=q_i(n)$ and $x_i=x_i(n;x)$ for $0\le i\le L$. Also, define $r$ by
$m=\phi(q_0\cdots q_L)r$.  Let $\eps_i=\max(0.82\eta, \frac{i}{3L_0})$.
Let $U$ be the set of totients $m\le x$ satisfying one of four
conditions:
\begin{enumerate}
\item $(x_1, x_2, \ldots, x_L) \not\in \fancyS_L(\vxi)$,
\item $m$  is not $S$-nice,
\item $\exists i\le \frac{L_0}3 : \;\; \left| \frac{x_i}{\b_i}-1\right| \ge \eps_i$,
\item $\om(r) \ge (\log_2 x)^{1/2}$. 
\end{enumerate}
By Theorem \ref{vx norm} and Lemma \ref{divisor normal},
 the number of totients $m\le x$ satisfying (1) or (2) is $O(V(x)(\log_2 x)^{-\frac14 \eta})$.
Theorem \ref{qi normal} implies that the number of totients satisfying (3) is
\[
 \ll V(x) \Bigg[(\eta L_0) e^{-0.82\eta L_0/13} + \sum_{i\ge 2.46\eta L_0} e^{-i/39}\Bigg]
 \ll V(x) e^{-\frac1{16}\eta L_0} \ll \frac{V(x)}{(\log_2 x)^{\eta/10}}.
\]
Consider now totients satisfying (4), but neither (1), (2) nor (3).
By (3), $q_1\cdots q_L \le x^{1/3}$.  By Lemma \ref{xz},
\[
 \log_2 P^+(r) \le x_L \log_2 x \le 10\rho^L \log_2 x \le 20\rho^{-\Psi}\log_3 x <
\exp(\sqrt{\log_3 x}).
\]
By Lemma \ref{large prime small}, the number of totients with
$r\ge R := \exp \exp (\frac1{10}\sqrt{\log_2 x})$ is $O(\frac{x}{\log x})$.
Now suppose $r<R$.  Given $q_1,\ldots,q_L$ and $r$, the number of possibilities for $q_0$
is 
\[
 \ll \frac{x}{q_1\cdots q_L r \log x}.
\]
Applying Lemma \ref{sumvol}, followed by Lemmas \ref{TL} and \ref{xL most}, gives
\[
 \sum \frac{1}{q_1\cdots q_L} \le R_L(\bxi) \ll Z(x)e^{-\frac14 \Psi^2}\ll Z(x)
(\log_2 x)^{-\frac14 \eta}.
\]
For $r\le y\le R$, we have $\om(r)\ge 10\log_2 R\ge 10\log_2 y$.  Hence, the number
of possible $r\le y$ is $O(y/\log^2 y)$ by Lemma \ref{Omega lem}.  Therefore,
$\sum_r 1/r = O(1)$ and we conclude that
\be\label{U(x)}
|U| \ll V(x) (\log_2 x)^{-\frac1{10} \eta}.
\ee

Assume now that a totient $m\not\in U$.  Since every prime factor of a
preimage $n$ is $S$-normal,
\[
 \Omega(m) = (1+x_1+\cdots+x_L)\log_2 x + O\( (\log_2 x)^{\frac12} (\log_3 x)^{\frac32} \).
\]
Since (3) fails, Lemma \ref{xz} implies
\[
 \sum_{1\le i\le L} x_i \le \sum_{i\le L_0/3} \rho^i(1+0.82\eta)+\sum_{L_0/3<i\le L}
5\rho^{\fl{L_0/3}}\le \frac{\rho}{1-\rho}+0.98\eta
\]
and
\begin{align*}
  \sum_{1\le i\le L} x_i &\ge \sum_{i\le L_0/3} \b_i (1-\eps_i) \ge
\sum_{i\le L_0/3} \rho^i(1-0.82\eta)-\sum_{i\ge 1}\frac{i\rho^i}{L_0}-\sum_{i\ge 2.46L_0\eta}
\frac{i\rho^i}{3L_0} \\
&\ge \frac{\rho}{1-\rho}(1-0.82\eta) - 4\rho^{L_0/3-1} - \frac{5}{L_0} \ge  
\frac{\rho}{1-\rho}-0.98\eta.
\end{align*}
Therefore, if $x$ is large then $|\om(m)-\frac{1}{1-\rho}\log_2 x| \le 0.99 \eta\log_2 x$
for $m\not\in U$.
This proves the first part of Theorem \ref{Omega normal}.
The second part follows easily, since a totient $m\not \in U$ is  $S$-nice and hence
\[
\om(m) - \omega(m) \le  \sum_{i=0}^L \om(q_i-1,1,S) + \om(r)
\ll (\log_2 x)^{1/2}. \qedhere
\]
\end{proof}


\begin{proof}[Proof of Corollary \ref{Omega normal cor}]
It suffices to prove the theorem with $g(m)=\om(m)$.  Divide the totients
$m\le x$ into three sets, $S_1$, those with $\om(m)\ge 10\log_2 x$, $S_2$,
those not in $S_1$ but with $| \om(m) - \log_2 x/(1-\rho)| \ge \frac13
\log_2 x$, and $S_3$, those not counted in $S_1$ or $S_2$.  By Lemma
\ref{Omega lem}, 
$|S_1| \ll \frac{x}{\log^2 x}$
and by Theorem \ref{Omega normal},
$|S_2| \ll V(x) (\log_2 x)^{-1/30}.$  Therefore
\be\label{S3 size}
|S_3| = V(x) (1 - O((\log_2 x)^{-1/30}) )
\ee
and also
\be\label{sum S1S2}
\sum_{m\in S_1\cup S_2} \om(m) \ll |S_1| \log x + |S_2| \log_2 x \ll
V(x)(\log_2 x)^{2/3}.
\ee
For each $m\in S_3$, let
$$
\e_m = \frac{\om(m)}{\log_2 x} - \frac{1}{1-\rho}
$$
and for each integer $N\ge 0$,
let $S_{3,N}$ denote the set of $m\in S_3$ with $N\le |\e_m|\log_3 x< N+1$.
By Theorem \ref{Omega normal}, \eqref{S3 size} and \eqref{sum S1S2},
\begin{align*}
\sum_{m\in \fancyV(x)} \om(m) &= O(V(x)\sqrt{\log_2 x}) +
\sum_{0\le N\le \frac12\log_3 x}\,\, \sum_{m\in S_{3,N}} \om(m) \\
&= \frac{\log_2 x}{1-\rho} |S_3| +
O\( V(x) \frac{\log_2 x}{\log_3 x} \sum_N (N+1) e^{-N/10} \) \\
&= \frac{V(x)\log_2 x}{1-\rho} \( 1 + O \( \frac{1}{\log_3 x} \) \).
\qedhere
\end{align*}
\end{proof}

%
%
%
%
\section{The distribution of $A(m)$}

\subsection{Large values of $A(m)$}

\begin{proof}[Proof of Theorem \ref{A(m) bounded}]
First we note the trivial bound
\[
 | \{ m\le x: A(m)\ge N \}| \ll \frac{x \log_2 x}{N} \ll V(x) \frac{\log x}{N},
\]
which implies the theorem when $N\ge \log^2 x$.  Suppose next that $N<\log^2 x$.
Suppose $x$ is sufficiently large and set $\Psi=\lceil \log\log N \rceil$ and
$L=L_0(x)-\Psi$.  Note that $\Psi < \frac34 L_0(x)$.
Define $\xi_i$ by \eqref{xi def}.
By Theorem \ref{vx norm}, the number of totients $m\le x$
with a pre-image $n$ satisfying $\vx(n)\not\in \fancyS_L(\vxi)$ is
$O(V(x)e^{-\frac14 \Psi^2})$ (here $\vx(n) = (x_1(n;x), \ldots, x_L(n;x))$). 
For other totients $m$, all preimages $n$ satisfy $\vx(n)\in \fancyS_L(\vxi)$.
By Lemma \ref{xz}, $x_L=x_L(n) \le 1/g_L$.  For integer $b\in \{0,1,\ldots,L-1\}$,
let $N_b$ be the number of these remaining totients $m\le x$ with a preimage $n$
satisfying 
$$
\frac{b}{L g_L} \le x_L < \frac{b+1}{Lg_L}.
$$
Put $Y_b=\frac{b+1}{Lg_L}\log_2 x$.
Write $n=q_0 \cdots q_{L} t$, so that $\log_2 P^+(t)\le Y_b$, and let $r=\phi(t)$.
 Also note that $\log_2 Y_b \ll b\rho^{_M}$.  
As in the proof of \eqref{Nb}, using Lemmas \ref{sumvol} and \ref{xL most}, together with
\eqref{sum1v-2} and Corollary \ref{TL cor}, gives
\begin{align*}
N_b(x) &\ll \frac{x}{\log x} R_{L}(\fancyS_L(\vxi) \cap \{x_L\ge b/(Lg_L)\};x) \sum_r
\frac{1}{r} \\
&\ll \frac{x}{\log x} e^{-C_0 b}T_L e^{C(\log Y_b)^2}
\ll  V(x) \exp \left\{ -C_0 b + \Psi \log b  + O(\Psi+\log^2 b) \right\}.
\end{align*}
Put $b_0=\lceil \Psi^2/C_0 \rceil$.  The number of totients with $x_L \ge b_0/(Lg_L)$ 
is therefore $ \ll V(x)e^{-\Psi^2+O(\Psi\log \Psi)} \ll V(x)e^{-\frac12 \Psi^2}$.
The remaining totients have all of their preimages of the form $n=q_0\cdots q_L t$ with
 $\log_2 P^+(t) \le Y_{b_0}$.  The number of such preimages is
\[
 \ll \frac{x}{\log x} R_L(\fancyS_L(\bxi);x) \sum_{\log_2  P^+(t)\le Y_{b_0}} \frac{1}{\phi(t)}
\ll V(x) e^{-C_0 b - \frac{1}{4C}\Psi^2+Z_{b_0}}.
\]
Hence, the number of totients $m$ having at least $N$ such preimages is
\[
 \ll \frac{V(x)}{N} e^{-C_0 b - \frac{1}{4C}\Psi^2+Z_{b_0}} \ll \frac{V(x)}{N^{1/2}}.
\qedhere
\]
\end{proof}

\subsection{Sierpi\'nski's Conjecture}

Schinzel's argument for deducing Sierpi\'nski's Conjecture for a given $k$ from
Hypothesis H requires the simultaneous primality of $\gg k$ polynomials
of degrees up to $k$.  Here we preset a different approach, which is
considerably simpler and requires only the simultaneous primality
of three linear polynomials.
We take a number $m$ with $A(m)=k$ and
construct an $l$ with $A(lm)=k+2$.  Our method is motivated
by the technique used in Section 5 where many numbers with multiplicity
$\kappa$ are constructed from a single example.

\begin{lem} \label{k to k+2} Suppose $A(m)=k$ and $p$ is a prime
satisfying
\begin{enumerate}
\item[(i)] $p > 2m+1$,
\item[(ii)] $2p+1$ and $2mp+1$ are prime,
\item[(iii)] $dp+1$ is composite for all $d|2m$ except
$d=2$ and $d=2m$.
\end{enumerate}
Then $A(2mp) = k+2$.
\end{lem}

\begin{proof}  Suppose $\phi^{-1}(m) = \{x_1, \ldots, x_k\}$ and
$\phi(x)=2mp$.  Condition (i) implies $p\nmid x$, hence $p|(q-1)$ for some
prime $q$ dividing $x$.  Since $(q-1)|2mp$, we have $q=dp+1$ for some divisor
$d$ of $2m$.  We have $q>2p$, so $q^2 \nmid x$ and $\phi(x)=
(q-1)\phi(x/q)$.
By conditions (ii) and (iii), either $q=2p+1$ or $q=2mp+1$.
In the former case, $\phi(x/q)=m$, which has solutions $x=(2p+1)x_i$
$(1 \le i \le k)$.  In the latter case, $\phi(x/q)=1$, which has solutions
$x=q$ and $x=2q$.
 \end{proof}

Suppose $A(m)=k$, $m\equiv 1\pmod{3}$,
 and let $d_1, \ldots, d_j$ be the divisors of $2m$
with $3\le d_i<2m$.
Let $p_1, \ldots, p_j$ be distinct primes satisfying $p_i>d_i$ for each $i$.
Using the Chinese Remainder Theorem, let $a \mod b$ denote the
intersection of the residue classes $-d_i^{-1} \mod p_i$ $(1\le i\le j)$.
For every $h$ and $i$, $(a+bh)d_i+1$ is divisible by $p_i$,
hence composite for large enough $h$.
The Prime $k$-tuples
Conjecture implies that there are infinitely many numbers $h$ so that
$p=a+hb$, $2p+1$ and $2mp+1$ are simultaneously prime.
By Lemma \ref{k to k+2}, $A(2mp)=k+2$.
As $p\equiv 2\pmod{3}$, $2mp \equiv 1\pmod{3}$.
Starting with $A(1)=2$, $A(2)=3$, and $A(220)=5$,
Sierpi\'nski's Conjecture follows by induction on $k$.


\afterpage{\clearpage}
{\tiny
\begin{table}
\begin{tabular}{|lr|lr|lr|lr|lr|lr|lr|lr|}
\hline
 $ k$ & $ m_k$\hfil
& $ k$ & $ m_k$\hfil
& $ k$ & $ m_k$\hfil
& $ k$ & $ m_k$\hfil
& $ k$ & $ m_k$\hfil
& $ k$ & $ m_k$\hfil
& $ k$ & $ m_k$\hfil
& $ k$ & $ m_k$\hfil\\
\hline
    2 & 1 &   77 & 9072 &  152 & 10080 &  227 & 26880 &  302 & 218880 &  377 & 165888 &  452 & 990720 &  527 & 2677248 \\
    3 & 2 &   78 & 38640 &  153 & 13824 &  228 & 323136 &  303 & 509184 &  378 & 436800 &  453 & 237600 &  528 & 5634720 \\
    4 & 4 &   79 & 9360 &  154 & 23760 &  229 & 56160 &  304 & 860544 &  379 & 982080 &  454 & 69120 &  529 & 411840 \\
    5 & 8 &   80 & 81216 &  155 & 13440 &  230 & 137088 &  305 & 46080 &  380 & 324000 &  455 & 384000 &  530 & 2948400 \\
    6 & 12 &   81 & 4032 &  156 & 54720 &  231 & 73920 &  306 & 67200 &  381 & 307200 &  456 & 338688 &  531 & 972000 \\
    7 & 32 &   82 & 5280 &  157 & 47040 &  232 & 165600 &  307 & 133056 &  382 & 496800 &  457 & 741888 &  532 & 2813184 \\
    8 & 36 &   83 & 4800 &  158 & 16128 &  233 & 184800 &  308 & 82944 &  383 & 528768 &  458 & 86400 &  533 & 3975552 \\
    9 & 40 &   84 & 4608 &  159 & 48960 &  234 & 267840 &  309 & 114048 &  384 & 1114560 &  459 & 1575936 &  534 & 368640 \\
   10 & 24 &   85 & 16896 &  160 & 139392 &  235 & 99840 &  310 & 48384 &  385 & 1609920 &  460 & 248832 &  535 & 529920 \\
   11 & 48 &   86 & 3456 &  161 & 44352 &  236 & 174240 &  311 & 43200 &  386 & 485760 &  461 & 151200 &  536 & 2036736 \\
   12 & 160 &   87 & 3840 &  162 & 25344 &  237 & 104832 &  312 & 1111968 &  387 & 1420800 &  462 & 1176000 &  537 & 751680 \\
   13 & 396 &   88 & 10800 &  163 & 68544 &  238 & 23040 &  313 & 1282176 &  388 & 864864 &  463 & 100800 &  538 & 233280 \\
   14 & 2268 &   89 & 9504 &  164 & 55440 &  239 & 292320 &  314 & 239616 &  389 & 959616 &  464 & 601344 &  539 & 463680 \\
   15 & 704 &   90 & 18000 &  165 & 21120 &  240 & 93600 &  315 & 1135680 &  390 & 1085760 &  465 & 216000 &  540 & 2042880 \\
   16 & 312 &   91 & 23520 &  166 & 46656 &  241 & 93312 &  316 & 274560 &  391 & 264960 &  466 & 331776 &  541 & 3018240 \\
   17 & 72 &   92 & 39936 &  167 & 15840 &  242 & 900000 &  317 & 417600 &  392 & 470016 &  467 & 337920 &  542 & 2311680 \\
   18 & 336 &   93 & 5040 &  168 & 266400 &  243 & 31680 &  318 & 441600 &  393 & 400896 &  468 & 95040 &  543 & 1368000 \\
   19 & 216 &   94 & 26208 &  169 & 92736 &  244 & 20160 &  319 & 131040 &  394 & 211200 &  469 & 373248 &  544 & 3120768 \\
   20 & 936 &   95 & 27360 &  170 & 130560 &  245 & 62208 &  320 & 168480 &  395 & 404352 &  470 & 559872 &  545 & 1723680 \\
   21 & 144 &   96 & 6480 &  171 & 88128 &  246 & 37440 &  321 & 153600 &  396 & 77760 &  471 & 228096 &  546 & 1624320 \\
   22 & 624 &   97 & 9216 &  172 & 123552 &  247 & 17280 &  322 & 168000 &  397 & 112320 &  472 & 419328 &  547 & 262080 \\
   23 & 1056 &   98 & 2880 &  173 & 20736 &  248 & 119808 &  323 & 574080 &  398 & 1148160 &  473 & 762048 &  548 & 696960 \\
   24 & 1760 &   99 & 26496 &  174 & 14400 &  249 & 364800 &  324 & 430560 &  399 & 51840 &  474 & 342720 &  549 & 1889280 \\
   25 & 360 &  100 & 34272 &  175 & 12960 &  250 & 79200 &  325 & 202752 &  400 & 152064 &  475 & 918720 &  550 & 734400 \\
   26 & 2560 &  101 & 23328 &  176 & 8640 &  251 & 676800 &  326 & 707616 &  401 & 538560 &  476 & 917280 &  551 & 842400 \\
   27 & 384 &  102 & 28080 &  177 & 270336 &  252 & 378000 &  327 & 611520 &  402 & 252000 &  477 & 336000 &  552 & 874368 \\
   28 & 288 &  103 & 7680 &  178 & 11520 &  253 & 898128 &  328 & 317952 &  403 & 269568 &  478 & 547200 &  553 & 971520 \\
   29 & 1320 &  104 & 29568 &  179 & 61440 &  254 & 105600 &  329 & 624960 &  404 & 763776 &  479 & 548352 &  554 & 675840 \\
   30 & 3696 &  105 & 91872 &  180 & 83520 &  255 & 257040 &  330 & 116640 &  405 & 405504 &  480 & 129600 &  555 & 4306176 \\
   31 & 240 &  106 & 59040 &  181 & 114240 &  256 & 97920 &  331 & 34560 &  406 & 96768 &  481 & 701568 &  556 & 1203840 \\
   32 & 768 &  107 & 53280 &  182 & 54432 &  257 & 176256 &  332 & 912000 &  407 & 1504800 &  482 & 115200 &  557 & 668160 \\
   33 & 9000 &  108 & 82560 &  183 & 85536 &  258 & 264384 &  333 & 72576 &  408 & 476928 &  483 & 1980000 &  558 & 103680 \\
   34 & 432 &  109 & 12480 &  184 & 172224 &  259 & 244800 &  334 & 480000 &  409 & 944640 &  484 & 1291680 &  559 & 2611200 \\
   35 & 7128 &  110 & 26400 &  185 & 136800 &  260 & 235872 &  335 & 110880 &  410 & 743040 &  485 & 1199520 &  560 & 820800 \\
   36 & 4200 &  111 & 83160 &  186 & 44928 &  261 & 577920 &  336 & 1259712 &  411 & 144000 &  486 & 556416 &  561 & 663552 \\
   37 & 480 &  112 & 10560 &  187 & 27648 &  262 & 99360 &  337 & 1350720 &  412 & 528000 &  487 & 359424 &  562 & 282240 \\
   38 & 576 &  113 & 29376 &  188 & 182400 &  263 & 64800 &  338 & 250560 &  413 & 1155840 &  488 & 1378080 &  563 & 3538944 \\
   39 & 1296 &  114 & 6720 &  189 & 139104 &  264 & 136080 &  339 & 124416 &  414 & 4093440 &  489 & 2088000 &  564 & 861120 \\
   40 & 1200 &  115 & 31200 &  190 & 48000 &  265 & 213120 &  340 & 828000 &  415 & 134400 &  490 & 399168 &  565 & 221760 \\
   41 & 15936 &  116 & 7200 &  191 & 102816 &  266 & 459360 &  341 & 408240 &  416 & 258048 &  491 & 145152 &  566 & 768000 \\
   42 & 3312 &  117 & 8064 &  192 & 33600 &  267 & 381024 &  342 & 74880 &  417 & 925344 &  492 & 2841600 &  567 & 2790720 \\
   43 & 3072 &  118 & 54000 &  193 & 288288 &  268 & 89856 &  343 & 1205280 &  418 & 211680 &  493 & 1622880 &  568 & 953856 \\
   44 & 3240 &  119 & 6912 &  194 & 286848 &  269 & 101376 &  344 & 192000 &  419 & 489600 &  494 & 1249920 &  569 & 7138368 \\
   45 & 864 &  120 & 43680 &  195 & 59904 &  270 & 347760 &  345 & 370944 &  420 & 1879200 &  495 & 2152800 &  570 & 655200 \\
   46 & 3120 &  121 & 32400 &  196 & 118800 &  271 & 124800 &  346 & 57600 &  421 & 1756800 &  496 & 2455488 &  571 & 3395520 \\
   47 & 7344 &  122 & 153120 &  197 & 100224 &  272 & 110592 &  347 & 1181952 &  422 & 90720 &  497 & 499200 &  572 & 3215520 \\
   48 & 3888 &  123 & 225280 &  198 & 176400 &  273 & 171360 &  348 & 1932000 &  423 & 376320 &  498 & 834624 &  573 & 2605824 \\
   49 & 720 &  124 & 9600 &  199 & 73440 &  274 & 510720 &  349 & 1782000 &  424 & 1461600 &  499 & 1254528 &  574 & 1057536 \\
   50 & 1680 &  125 & 15552 &  200 & 174960 &  275 & 235200 &  350 & 734976 &  425 & 349920 &  500 & 2363904 &  575 & 1884960 \\
   51 & 4992 &  126 & 4320 &  201 & 494592 &  276 & 25920 &  351 & 473088 &  426 & 158400 &  501 & 583200 &  576 & 3210240 \\
   52 & 17640 &  127 & 91200 &  202 & 38400 &  277 & 96000 &  352 & 467712 &  427 & 513216 &  502 & 1029600 &  577 & 1159200 \\
   53 & 2016 &  128 & 68640 &  203 & 133632 &  278 & 464640 &  353 & 556800 &  428 & 715392 &  503 & 2519424 &  578 & 4449600 \\
   54 & 1152 &  129 & 5760 &  204 & 38016 &  279 & 200448 &  354 & 2153088 &  429 & 876960 &  504 & 852480 &  579 & 272160 \\
   55 & 6000 &  130 & 49680 &  205 & 50688 &  280 & 50400 &  355 & 195840 &  430 & 618240 &  505 & 1071360 &  580 & 913920 \\
   56 & 12288 &  131 & 159744 &  206 & 71280 &  281 & 30240 &  356 & 249600 &  431 & 772800 &  506 & 3961440 &  581 & 393120 \\
   57 & 4752 &  132 & 16800 &  207 & 36288 &  282 & 157248 &  357 & 274176 &  432 & 198720 &  507 & 293760 &  582 & 698880 \\
   58 & 2688 &  133 & 19008 &  208 & 540672 &  283 & 277200 &  358 & 767232 &  433 & 369600 &  508 & 1065600 &  583 & 2442240 \\
   59 & 3024 &  134 & 24000 &  209 & 112896 &  284 & 228480 &  359 & 40320 &  434 & 584640 &  509 & 516096 &  584 & 6914880 \\
   60 & 13680 &  135 & 24960 &  210 & 261120 &  285 & 357696 &  360 & 733824 &  435 & 708480 &  510 & 616896 &  585 & 695520 \\
   61 & 9984 &  136 & 122400 &  211 & 24192 &  286 & 199584 &  361 & 576576 &  436 & 522720 &  511 & 639360 &  586 & 497664 \\
   62 & 1728 &  137 & 22464 &  212 & 57024 &  287 & 350784 &  362 & 280800 &  437 & 884736 &  512 & 4014720 &  587 & 808704 \\
   63 & 1920 &  138 & 87120 &  213 & 32256 &  288 & 134784 &  363 & 63360 &  438 & 1421280 &  513 & 266112 &  588 & 2146176 \\
   64 & 2400 &  139 & 228960 &  214 & 75600 &  289 & 47520 &  364 & 1351296 &  439 & 505440 &  514 & 2386944 &  589 & 2634240\\
   65 & 7560 &  140 & 78336 &  215 & 42240 &  290 & 238464 &  365 & 141120 &  440 & 836352 &  515 & 126720 &  590 & 4250400 \\
   66 & 2304 &  141 & 25200 &  216 & 619920 &  291 & 375840 &  366 & 399360 &  441 & 60480 &  516 & 2469600 &  591 & 2336256 \\
   67 & 22848 &  142 & 84240 &  217 & 236160 &  292 & 236544 &  367 & 168960 &  442 & 1836000 &  517 & 2819520 &  592 & 1516320 \\
   68 & 8400 &  143 & 120000 &  218 & 70560 &  293 & 317520 &  368 & 194400 &  443 & 866880 &  518 & 354816 &  593 & 268800 \\
   69 & 29160 &  144 & 183456 &  219 & 291600 &  294 & 166320 &  369 & 1067040 &  444 & 1537920 &  519 & 1599360 &  594 & 656640 \\
   70 & 5376 &  145 & 410112 &  220 & 278400 &  295 & 312000 &  370 & 348480 &  445 & 1219680 &  520 & 295680 &  595 & 1032192 \\
   71 & 3360 &  146 & 88320 &  221 & 261360 &  296 & 108864 &  371 & 147840 &  446 & 349440 &  521 & 1271808 &  596 & 4743360 \\
   72 & 1440 &  147 & 12096 &  222 & 164736 &  297 & 511488 &  372 & 641520 &  447 & 184320 &  522 & 304128 &  597 & 4101120 \\
   73 & 13248 &  148 & 18720 &  223 & 66240 &  298 & 132480 &  373 & 929280 &  448 & 492480 &  523 & 3941280 &  598 & 2410560 \\
   74 & 11040 &  149 & 29952 &  224 & 447120 &  299 & 354240 &  374 & 1632000 &  449 & 954720 &  524 & 422400 &  599 & 9922560 \\
   75 & 27720 &  150 & 15120 &  225 & 55296 &  300 & 84480 &  375 & 107520 &  450 & 1435200 &  525 & 80640 &  600 & 427680 \\
   76 & 21840 &  151 & 179200 &  226 & 420000 &  301 & 532800 &  376 & 352512 &  451 & 215040 &  526 & 508032 & & \\
\hline
\end{tabular}
\caption{Smallest solution to $A(m)=k$}
\end{table}
}  

Table 2 of \cite{SW} lists
the smallest $m$, denoted $m_k$, for which $A(m)=k$ for $2\le k\le 100$.  We extend the
computation to $k\le 600$, listing $m_k$ for $k\le 600$ in Table 2.

\subsection{Carmichael's Conjecture}
\medskip
The basis for computations of lower bounds for a counterexample to
Carmichael's Conjecture is the following Lemma of Carmichael \cite{C2},
as refined by Klee \cite{K}.  For short, let $s(n)=\prod_{p|n} p$ denote the
square-free kernel of $n$.

\begin{lem}\label{CK} Suppose $\phi(x)=m$ and $A(m)=1$.  If $d|x$, $(d,x/d)=1$,
$s(\phi(d))|x$, $e | \frac{x/d}{s(x/d)}$ and $P=1+e\phi(d)$ is prime, then
$P^2|x$.
\end{lem}

From Lemma \ref{CK} it is easy to deduce $2^2 3^2 7^2 43^2|x$.  Here,
following Carmichael, we break into two cases: (I) $3^2\parallel x$ and
(II) $3^3|x$.  In case (I) it is easy to show that $13^2|x$.  From this point
onward Lemma \ref{CK} is used to generate a virtually unlimited set of
primes $P$ for which $P^2|x$.  In case (I) we search for $P$ using
$d=1,e=6k$ or $d=9,e=2k$, where $k$ is a product of distinct primes
(other than 2 or 3) whose squares we already know divide $x$.  That is,
if $6k+1$ or $12k+1$ is prime, its square divides $x$.  In case (II) we try
$d=1,e=6k$ and $d=1,e=18k$, i.e. we test whether or not $6k+1$ and $18k+1$
are prime.


As in \cite{SW}, certifying that a number $P$ is prime is accomplished
with the following lemma of Lucas, Lehmer, Brillhart and Selfridge.

\begin{lem}\label{LLBS} Suppose, for each prime $q$ dividing $n-1$,
there is a number $a_q$ satisfying $a_q^{n-1} \equiv 1$ and 
$a_q^{(n-1)/q} \not\equiv 1 \pmod{n}$.  Then $n$ is prime.
\end{lem}

The advantage of using Lemma \ref{LLBS} in our situation is that for a
given $P$ we are testing, we already know the prime factors of $P-1$
(i.e. 2,3 and the prime factors of $k$).

Our overall search strategy differs from \cite{SW}.  In each case, we first
find a set of 32 ``small'' primes $P$ (from here on, $P$ will represent a
prime generated from Lemma \ref{CK} for which $P^2|x$, other than 2 or 3).
Applying Lemma \ref{CK}, taking $k$ to be all possible products of
1,2,3 or 4 of these 32 primes yields a set $S$ of 1000 primes $P$, 
which we order $p_1< \cdots < p_{1000}$.  This set will be our base set.
In particular, $p_{1000}=796486033533776413$ in case (I)
 and $p_{1000}=78399428950769743507519$ in case
(II).  The calculations are then divided into ``runs''.  For run \#0, we take
for $k$ all possible combinations of 1,2 or 3 of the primes in $S$.
For $j\ge 1$, run \#$j$ tests every $k$ which is the product of $p_j$ and
three larger primes in $S$.   Each candidate $P$ is
first tested for divisibility by small primes and must pass the strong
pseudoprime test with bases 2,3,5,7,11 and 13 before attempting to
certify that it is prime.
There are two advantages to this approach.  First, the candidates $P$
are relatively small (the numbers tested in case (I) had an average of
40 digits and the numbers tested in case (II) had an average of 52 digits).
Second, $P-1$ has at most 6 prime factors, simplifying the certification
process. 
To achieve $\prod P^2 > 10^{10^{10}}$,
13 runs  were required in case (I) and 14 runs were
required in case (II).  Together these runs give Theorem \ref{CC lower}.
A total of 126,520,174 primes were found in case (I), and 
104,942,148 primes were found in case (II).
The computer program was written in GNU C, utilizing Arjen Lenstra's
Large Integer Package, and run on a network of 200MHz
Pentium PCs running LINUX O/S in December 1996 (4,765 CPU hours total). 

In 1991, Pomerance (see \cite{P2} and \cite{M}) showed that
\be\label{liminf1/2}
\liminf_{x\to\infty} \frac{V_1(x)}{V(x)} \le \frac12.
\ee
A modification of his argument, combined
with the above computations, yields the much stronger
bound in Theorem \ref{liminf V_1/V}.
Recall that $V(x;k)$ counts the totients $\le x$, all of whose preimages 
are divisible by $k$.

\begin{lem} \label{V a^2}
We have $V(x;a^2) \le V(x/a).$
\end{lem}
\begin{proof}  The lemma is trivial when $a=1$ so assume $a\ge 2$.
Let $n$ be a totient with $x/a<n\le x$.  First we show that
for some integer $s\ge 0$, $a^{-s}n$ is a
 totient with an pre-image not divisible by $a^2$.
Suppose $\phi(m)=n$.  If $a^2\nmid m$, take $s=0$.  Otherwise
we can write $m=a^tr$, where $t\ge 2$ and $a\nmid r$.  Clearly
$\phi(ar)=a^{1-t}n$, so we take $s=t-1$.
Next, if $n_1$ and $n_2$ are two distinct totients in $(x/a,x]$,
then $a^{-s_1}n_1 \ne a^{-s_2}n_2$ (since $n_1/n_2$ cannot be
 a power of $a$), so the mapping from totients in $(x/a,x]$
to totients $\le x$ with a pre-image not divisible by $a^2$ is one-to-one.
Thus
$V(x) - V(x;a^2) \ge V(x) - V(x/a).$
 \end{proof}

The above computations show that if $\phi(x)=n$ and $A(n)=1$,
then $x$ is
divisible by either $a^2$ or $b^2$, where $a$ and $b$ are
numbers greater than $10^{5,001,850,000}$.
Suppose $a\le b$. By Lemma \ref{V a^2}, we have
\be\label{V1 upper}
V_1(x) \le V(x/a) + V(x/b) \le 2V(x/a).
\ee

\begin{lem} \label{liminf}  Suppose $a>1$, $b>0$ and
$V_1(x) \le b V(x/a)$ for all $x$.  Then
$$
\liminf_{x\to\infty} \frac{V_1(x)}{V(x)} \le \frac{b}{a}.
$$
\end{lem}

\begin{proof}
Suppose
$c = \liminf_{x\to\infty} \frac{V_1(x)}{V(x)}>0$.
For every $\e>0$ there is a number $x_0$ such that $x\ge x_0$ implies
$V_1(x)/V(x) \ge c-\e$.  For large $x$, set $n=[\log(x/x_0)/\log a]$.
Then
\begin{align*}
V(x) &= \frac{V(x)}{V(x/a)} \frac{V(x/a)}{V(x/a^2)} \cdots \frac{V(x/a^{n-1})}
{V(x/a^n)} V(x/a^n) \\
&\le b^n \frac{V(x)}{V_1(x)} \frac{V(x/a)}{V_1(x/a)} \cdots \frac{V(x/a^{n-1})}
{V_1(x/a^{n-1})} V(ax_0) \\
&\le b^n(c-\e)^{-n} (ax_0) = O(x^{-\log((c-\e)/b)/\log a}).
\end{align*}
This contradicts the trivial bound $V(x)\gg x/\log x$ if $c>\frac{b}{a}+\e$.
Since $\e$ is arbitrary, the lemma follows.
\end{proof}

Theorem \ref{liminf V_1/V} follows immediately.
Further improvements in the lower bound for a counterexample to Carmichael's
Conjecture will produce corresponding upper bounds on $\liminf_{x\to \infty}
V_1(x)/V(x)$.
Explicit bounds for the $O(1)$ term appearing in Theorem \ref{V(x)}
(which would involve considerable work to obtain) combined with
\eqref{V1 upper} should give  a strong upper bound for
$\limsup_{x\to \infty} V_1(x)/V(x)$.

Next, suppose $d$ is a totient, all of whose
pre-images $m_i$ are divisible by $k$.
The lower bound argument given in Section
5 shows that for at least half of the numbers $b\in \fancyB$, the totient
$\phi(b)d$ has only the pre-images $bm_i$.  In particular, all of the
pre-images of such totients are divisible by $k$ and Theorem \ref{V(x;k)}
follows.

It is natural to ask for which $k$ do there exist totients, all
of whose pre-images are divisible by $k$.
A short search reveals examples for each $k\le 11$ except $k=6$ and $k=10$.
For $k\in\{2,4,8\}$, take $d=2^{18} \cdot 257$, for $k\in\{3,9\}$, take
$d=54=2\cdot 3^3$,
for $k=5$ take $d=12500=4\cdot 5^5$, for $k=7$, take $d=294=6\cdot 7^2$
and for $k=11$, take $d=110$.  It appears that there might not be any
totient, all of whose pre-images are divisible by 6, but I cannot prove this.
Any totient with a unique pre-image must have that pre-image
divisible by 6, so the non-existence of such numbers
implies Carmichael's Conjecture.

I believe that obtaining the asymptotic formula for $V(x)$ will require
simultaneously determining the asymptotics of $V_k(x)/V(x)$ 
(more will be said in section 8) and $V(x;k)/V(x)$ for each $k$.  It may even
be necessary to classify totients more finely.  For instance, taking
$d=4, k=4$ in the proof of Theorem 2 (section 5), the totients $m$ constructed
have $\phi^{-1}(m) = \{ 5n,8n,10n,12n \}$ for some $n$.  On the other hand,
taking $d=6, k=4$ produces a different set of totients $m$, namely those
with $\phi^{-1}(m) = \{ 7n, 9n, 14n, 18n \}$ for some $n$.
Likewise, for any given $d$ with $A(d)=k$, the construction
of totients in Section 5 may miss whole classes of totients with multiplicity
$k$.  There is much further work to be done in this area.

%
%
\section{Generalization to other multiplicative functions}
%
%

The proofs of our theorems easily generalize to
a wide class of multiplicative arithmetic functions with similar behavior on primes,
such as $\sigma(n)$, the sum of divisors function.
If $f:\NN \to \NN$ is a multiplicative arithmetic function,
we analogously define
\be\label{f basic def}\begin{split}
\fancyV_f &= \{ f(n): n\in \NN \}, \quad V_f(x) = |\fancyV_f \cap [1,x]|, \\
f^{-1}(m) &= \{n:f(n)=m\}, \; A_f(m) = |f^{-1}(m)|, \;
V_{f,k}(x) = |\{m \le x : A_f(m)=k \}|.
\end{split}\ee
We now indicate the modifications to the previous argument needed to
prove Theorem \ref{general f}.
By itself, condition \eqref{fp} is enough to prove the lower bound for
$V_f(x)$.  
Condition \eqref{f lower} is used only for the
upper bound argument and the lower bound for $V_{f,k}(x)$.

The function $f(n)=n$, which takes all positive integer values, is an example
of why zero must be excluded from the set in \eqref{fp}.
Condition \eqref{f lower} insures that the values of
$f(p^k)$ for $k\ge 2$ are not too small too often, and thus
have little influence on the size of $V_f(x)$.
It essentially forces $f(h)$ to be a bit larger than $h^{1/2}$ on average.
It's probable that \eqref{f lower} can be relaxed, but not too much.
For
example, the multiplicative function defined by $f(p)=p-1$ for prime $p$,
and $f(p^k)=p^{k-1}$ for $k\ge 2$ clearly takes all integer values, while
$$
\sum_{h\ge 4, \text{ square-full}} \frac1{f(h)(\log_2 h)^2} \ll 1.
$$
Condition \eqref{f lower} also insures that $A(m)$ is finite for each $f$-value
$m$.  For example, a function satisfying $f(p^k)=1$ for infinitely many
prime powers $p^k$ has
the property that $A(m)=\infty$ for every $f$-value $m$.

In general, implied constants will depend on the function $f(n)$.
One change that must be made throughout is to replace every occurrence
of ``$p-1$'' (when referring to $\phi(p)$) with ``$f(p)$'',
for instance in the definition
of $S$-normal primes in Section 2.  Since the possible values of $f(p)-p$
is a finite set, Lemma \ref{normal lem} follows easily
with the new definitions.  The most substantial change to be made in
Section 2, however, is to Lemma \ref{divisor squares}, since we no longer
have the bound $n/f(n) \ll \log_2 n$ at our disposal.

\newtheorem*{lemds*}{Lemma \ref{divisor squares}$^*$}
\begin{lemds*}
  The
number of $m \in \fancyV_f(x)$ for which either
$d^2|m$ or $d^2|n$ for some $n \in f^{-1}(m)$ and  $d>Y$
is $O(x(\log_2 x)^K/Y^{2\del})$, where $K=\max_p (p-f(p))$.
\end{lemds*}

\begin{proof}
The number of $m$ with $d^2|m$ for some $d>Y$ is $O(x/Y)$.  Now suppose $d^2|n$
for some $d>Y$, and let $h=h(n)$ be the square-full part of $n$ (the largest squarefull
divisor of $n$).
In particular, $h(n)>Y^2$.  From the fact that $f(p)\ge p-K$ for all primes $p$,
we have
$$
f(n) = f(h) f(n/h) \gg \frac{f(h) n}{h} (\log_2 (n/h))^{-K}.
$$
Thus, if $f(n)\le x$, then
$$
\frac{n}{h} \( \log_2 \frac{n}{h} \) \ll \frac{x}{f(h)}.
$$
Therefore, the number of possible $n$ with a given $h$ is crudely
$\ll x(\log_2 x)^K/f(h).$
By \eqref{f lower}, the total number of $n$ is at most
\[
\ll x(\log_2 x)^K \sum_{h\ge Y^2} \frac{1}{f(h)}
\ll \frac{x(\log_2 x)^K}{Y^{2\del}} \sum_{h} \frac{h^{\del}}{f(h)}
\ll \frac{x(\log_2 x)^K}{Y^{2\del}}.\qedhere
\]
\end{proof}

Applying Lemma \ref{divisor squares}$^*$ in the proof of Lemma \ref{divisor normal}
with $Y=S^{1/2}$ 
yields the same bound as claimed, since $S>\exp\{ (\log_2 x)^{36} \}$.

In Section 3, the only potential issue is with Lemma \ref{sumvol}, but
the analog of $t_m$ is $\ll \exp \{ -\del e^{m-1} \}$.

The only modification needed in Section 4 comes from 
the use of $\phi(ab)\ge \phi(a)\phi(b)$ in the
argument leading to \eqref{Nb}.  If $q_L\nmid w$, the existing argument
is fine.  If $q_L | w$, let $j=\max\{i\le L: q_i<q_{i-1} \}$.  Since
$q_{L-2}>q_{L}$, $j\in \{L-1,L\}$.  Write $f(q_1\cdots q_L w)=f(q_1\cdots q_{j-1})
f(w')$, where $w'=q_j\cdots q_L w$ and $(x_1,\ldots,x_j)\in{\curly R}_j(\fancyS_j((\xi_1,\ldots, \xi_{j-1}))$.  Put $v=f(w')$, use the analog of \eqref{crude} to bound $\sum 1/v$, and otherwise follow the argument leading to
\eqref{Nb}.

In Section 5, there are several changes.  For Lemma \ref{prodpiqi},
the equation \eqref{piqi} may have trivial solutions coming from pairs
$p,p'$ with $f(p)=f(p')$.
We say a prime $p$ is ``bad''
if $f(p)=f(p')$ for some prime $p'\ne p$ and say
$p$ is ``good'' otherwise.  By \eqref{fp} and Lemma
 \ref{sieve linear factors},
the number of bad primes $\le y$ is $O(y/\log^2 y)$, so
$\sum_{p\, bad} 1/p$ converges.  In Lemma \ref{prodpiqi}, add the hypothesis
that the $p_i$ and $q_i$ are all ``good''. 
Possible small values of $f(p^k)$ for some 
$p^k$ with $k\ge 2$ are another complication.  For each prime $p$, define
\be\label{Qp def}
Q(p) := \min_{k\ge 2} \frac{f(p^k)}{f(p)}.
\ee
Introduce another parameter $d$ (which will be the same $d$ as in Theorem
\ref{V_k(x)}) and suppose $L\le L_0-M$ where $M$ is a sufficiently large
constant depending on $P_0$ and $d$. 
If follows from \eqref{f lower} and \eqref{Qp def} that 
$$
\sum_{Q(p)\le d} \frac1{p} = O(d).
$$
In the definition of $\fancyB$, add the hypothesis that all primes
$p_i$ are ``good'' and replace \eqref{pL lower} by $Q(p_i) \ge \max(d+K+1,17)$
for every $i$.
Of course, \eqref{n range} is changed to $f(n) \le x/d$.
Fortunately, the numbers in $\fancyB$ are square-free by definition.
Consider the analog of \eqref{phi eq}. 
Since $Q(p_i)>d+K$ for each $p_i$,
if $n|n_1$ and one of the primes $q_i$ ($0\le i\le L$)
occurs to a power greater than 1, then
$\phi(n_1) > d\phi(n)$.  Therefore, the $L+1$ largest prime factors of
$n_1$ occur to the first power only, which forces $n_1=nm_i$ for
some $i$ (the trivial solutions).  For nontrivial solutions, we have
at least one index $i$ for which $p_i \ne q_i$, and hence $f(p_i)\ne f(q_i)$
(since each $p_i$ is ``good'').  Other changes are more obvious.: In \eqref{sigma1},
 the phrase
``$rt+1$ and $st+1$ are unequal primes'' is replace by ``$rt+a$ and
$st+a'$ are unequal primes for some pair of numbers $(a,a')$ with $a,a'
\in \curly P$.''   Here $\curly P$ denotes the set of possible
values of $f(p)-p$.
As $\curly P$ is finite, this poses
no problem in the argument.  Similar changes are made in several places in
the argument leading to \eqref{siti}.

Only small, obvious changes are needed for Theorem \ref{vx norm}.  The rest
of Section 6 needs very little attention, as the bounds ultimately rely on
Lemma \ref{sumvol} and the volume computations (which are independent of $f$).

It is not possible to prove analogs of Theorems 
\ref{Car Conj}--\ref{Sierp Conj} for general $f$ 
satisfying the hypotheses of Theorem
\ref{general f}.  One reason is that there might not be
any ``Carmichael Conjecture'' for $f$, e.g. $A_\sigma(3)=1$, where 
$\sigma$ is the sum of divisors function.
Furthermore, the proof of Theorem \ref{Sierp Conj} depends on the identity
$\phi(p^2)=p\phi(p)$ for primes $p$.  If, for some $a\ne 0$,
 $f(p)=p+a$ for all primes $p$, then the argument of \cite{FK} shows that if
the multiplicity $k$ is possible and $r$ is a positive integer, then the
multiplicity $rk$ is possible.  For functions such as $\sigma(n)$, for which
the multiplicity 1 is possible, this completely solves the problem of the
possible multiplicities.  For other functions, it shows at least that a
positive proportion of multiplicities are possible.  If multiplicity 1 is
not possible, and $f(p^2)=pf(p)$, the argument in \cite{F99} shows that
all multiplicities beyond some point are possible.
 
We can, however, obtain information
about the possible multiplicities for more general $f$ by an induction
argument utilizing the next lemma.  Denote by $a_1, \ldots, a_K$
the possible values of $f(p)-p$ for prime $p$.

\newtheorem*{lemkk*}{Lemma \ref{k to k+2}$^*$}
\begin{lemkk*}
Suppose $A_f(m)=k$.  Let $p,q,s$ be
primes and $r\ge 2$ an integer so that
\begin{enumerate}
\item{(i)} $s$ and $q$ are ``good'' primes,
\item{(ii)} $mf(s) = f(q)$,
\item{(iii)} $f(s)=rp$,
\item{(iv)} $p\nmid f(\pi^b)$ for every prime $\pi$, integer $b\ge 2$
with $f(\pi^b)\le mf(s)$,
\item{(v)} $dp-a_i$ is composite for $1\le i\le K$ and $d|rm$ except
$d=r$ and $d=rm$.
\end{enumerate}
Then $A_f(mrp)=k+A_f(1)$.
\end{lemkk*}

\begin{proof}  Let $f^{-1}(m)=\{ x_1, \ldots, x_k \}$ and suppose $f(x)=mrp$.
By condition (iv), $p|f(\pi)$ for some prime $\pi$ which divides $x$ to the
first power.  Therefore, $f(\pi)=dp$ for some divisor $d$ of $mr$.  Condition
(v) implies that the only possibilities for $d$ are $d=r$ or $d=rm$.  If
$d=r$, then $f(\pi)=rp=f(p)$ which forces $\pi=s$ by condition (i).
By conditions (ii) and (iii), we have $f(x/s)=m$, which gives solutions
$x=sx_i$\, $(1\le i\le k)$.  Similarly, if $d=rm$, then $\pi=q$ and
$f(x/q)=1$, which has $A_f(1)$ solutions.
\end{proof}

By the Chinese Remainder Theorem, there is an arithmetic progression
$\fancyA$ so that
condition (v) is satisfied for each number $p\in \fancyA$, while still
allowing each $rp+a_i$ and $rmp+a_i$ to be prime.  To eliminate
primes failing condition (iv), we need the asymptotic form of the Prime
$k$-tuples Conjecture due to Hardy and Littlewood \cite{HL} (actually only the
case where $a_i=1$ for each $i$ is considered in \cite{HL}; the conjectured
asymptotic for $k$ arbitrary polynomials can be found in \cite{BHo}).

\begin{conj}[Prime $k$-tuples Conjecture (asymptotic version]
Suppose $a_1, \ldots,
a_k$ are positive integers and $b_1, \ldots, b_k$ are integers so that no
prime divides 
$(a_1 n+b_1) \cdots (a_k n + b_k)$
for every integer $n$.  Then for some constant $C(\mathbf a,\mathbf b)$,
the number of $n\le x$ for which $a_1n+b_1, \ldots ,a_k n + b_k$ are
simultaneously prime is
$$
\sim C(\mathbf a,\mathbf b) \frac{x}{\log^k x} \qquad (x\ge x_0(\mathbf a,
\mathbf b)).
$$
\end{conj}

Using \eqref{f lower}, we readily obtain
$|\{\pi^b: f(\pi^b)\le y,b\ge 2 \}| \ll y^{1-\del}.$
If $s$ is taken large enough, the number of possible $p\le x$ satisfying
condition (iv) (assuming $r$ and $m$ are fixed and noting condition
(iii)) is $o(x/\log^3 x)$.  The procedure for determining the set of
possible multiplicities with
this lemma will depend on the behavior of the particular
function.  Complications can arise, for instance, if $m$ is even and all of
the $a_i$ are even (which makes condition (ii) impossible) or if the number
of ``bad'' primes is $\gg x/\log^3 x$.

\bigskip\bigskip


\end{document}